\documentclass[12pt]{article}
\usepackage[utf8]{inputenc}
\usepackage{amsthm}
\usepackage{amsmath}
\usepackage{bbm}
\usepackage{amsfonts}
\usepackage{amssymb}
\usepackage[left=2cm,right=2cm,top=2cm,bottom=2cm]{geometry}

\usepackage{tabto}
\TabPositions{2 cm, 4 cm, 6 cm, 8 cm}

\usepackage[justification=centering]{caption}

\usepackage{tikz-cd}

\usepackage{tocloft}
\usepackage{appendix}

\usepackage{enumerate}

\usepackage{hyperref}

\newtheorem{theorem}{Theorem}[section]
\newtheorem{corollary}[theorem]{Corollary}
\newtheorem{lemma}[theorem]{Lemma}
\newtheorem{proposition}[theorem]{Proposition}

\theoremstyle{definition}
\newtheorem{definition}[theorem]{Definition}
\newtheorem{example}[theorem]{Example}
\newtheorem*{definition*}{Definition}

\newtheorem*{lemma*}{Lemma}
\newtheorem*{proposition*}{Proposition}
\newtheorem*{theorem*}{Theorem}
\newtheorem*{corollary*}{Corollary}

\theoremstyle{remark}
\newtheorem{remark}[theorem]{Remark}
\newtheorem{warning}[theorem]{Warning}

\theoremstyle{definition}

\newtheorem{question}[theorem]{Question}

\DeclareMathOperator{\im}{im}

\title{Normed amenability and bounded cohomology \\ over non-Archimedean fields}
\author{Francesco Fournier-Facio}
\date{\today}

\addtocontents{toc}{\vspace{1cm}}

\begin{document}

\maketitle

\vspace{0.5cm}

\begin{abstract}
We study continuous bounded cohomology of totally disconnected locally compact groups with coefficients in a non-Archimedean valued field $\mathbb{K}$. To capture the features of classical amenability that induce the vanishing of bounded cohomology with real coefficients, we start by introducing the notion of normed $\mathbb{K}$-amenability, of which we prove an algebraic characterization. It implies that normed $\mathbb{K}$-amenable groups are locally elliptic, and it relates an invariant, the norm of a $\mathbb{K}$-amenable group, to the order of its discrete finite $p$-subquotients, where $p$ is the characteristic of the residue field of $\mathbb{K}$. Moreover, we prove a characterization of discrete normed $\mathbb{K}$-amenable groups in terms of vanishing of bounded cohomology with coefficients in $\mathbb{K}$.

The algebraic characterization shows that normed $\mathbb{K}$-amenability is a very restrictive condition, so the bounded cohomological one suggests that there should be plenty of groups with rich bounded cohomology with trivial $\mathbb{K}$ coefficients. We explore this intuition by studying the injectivity and surjectivity of the comparison map, for which surprisingly general statements are available. Among these, we show that if either $\mathbb{K}$ has positive characteristic or its residue field has characteristic $0$, then the comparison map is injective in all degrees. If $\mathbb{K}$ is a finite extension of $\mathbb{Q}_p$, we classify unbounded and non-trivial quasimorphisms of a group and relate them to its subgroup structure. For discrete groups, we show that suitable finiteness conditions imply that the comparison map is an isomorphism; this applies in particular to finitely presented groups in degree $2$.

A motivation as to why the comparison map is often an isomorphism, in stark contrast with the real case, is given by moving to topological spaces. We show that over a non-Archimedean field, bounded cohomology is a cohomology theory in the sense of Eilenberg--Steenrod, except for a weaker version of the additivity axiom which is however equivalent for finite disjoint unions. In particular there exists a Mayer--Vietoris sequence, the main missing piece for computing real bounded cohomology.

\end{abstract}

\pagebreak

\tableofcontents

\vspace{2 cm}

{\noindent \textit{Mathematics Subject Classification}. Primary: 22D05, 20J06, 55N35. Secondary: 22D12, 22D50.}

\vspace{1cm}

\noindent{\textsc{Department of Mathematics, ETH Z\"urich, Switzerland}}

\noindent{\textit{E-mail address:} \texttt{francesco.fournier@math.ethz.ch}}

\pagebreak

\section{Introduction}

Let $G$ be a locally compact Hausdorff group. A \textit{mean} on $G$ is a norm one linear map $m : C_b(G, \mathbb{R}) \to \mathbb{R}$ on the space $C_b(G, \mathbb{R})$ of continuous bounded real-valued functions on $G$, such that $m(\mathbbm{1}_G) = 1$ and $m(f) \geq 0$ whenever $f(x) \geq 0$ for every $x \in G$. A mean is \textit{left-invariant} if $m(g \cdot f) = m(f)$ for any $g \in G, f \in C_b(G, \mathbb{R})$, where $G$ acts on $C_b(G, \mathbb{R})$ by $(g \cdot f)(x) := f(g^{-1} x)$. If a left-invariant mean exists, $G$ is said to be \textit{amenable}.

The class of amenable groups includes compact and abelian groups, and is closed under the constructions of taking closed subgroups, extensions and directed unions. This notion was introduced by von Neumann \cite{vN} in the context of paradoxical decompositions, and since then amenable groups have played a pivotal role in other areas of mathematics, such as measure theory, representation theory, dynamics, and the theory of von Neumann algebras. The reason they are so omnipresent is that they admit a shocking amount of equivalent characterizations coming from different fields of mathematics: the reader is referred to \cite{Pier} for a detailed account of the classical ones. \\

The characterization we are interested in is in terms of bounded cohomology. Let $E$ be a real normed vector space on which $G$ acts by linear isometries: we call such an $E$ a \textit{normed $\mathbb{R}[G]$-module}. Then the \textit{continuous bounded cohomology of $G$ with coefficients in $E$}, denoted $H^\bullet_{cb}(G, E)$, is defined analogously to classical group cohomology, using the homogeneous cochain complex $C^\bullet_b(G, E) := C_b(G^{\bullet + 1}, E)$ of continuous bounded functions with values in $E$. When $G$ is discrete, and so all functions on $G$ are continuous, we simply call it the \textit{bounded cohomology of $G$ with coefficients in $E$}, and denote it by $H^\bullet_b(G, E)$.

Bounded cohomology was initially introduced by Johnson \cite{Johnson} and Trauber in the context of Banach algebras. After Gromov's seminal paper \cite{Grom} and Ivanov's functorial reformulation \cite{Ivanov}, bounded cohomology of discrete groups found applications in differential geometry, hyperbolic geometry, dynamics, and geometric group theory. The theory was then generalized to deal with locally compact groups by Burger and Monod in \cite{BurMon}, and since then it has become a fundamental tool in rigidity theory. We refer the reader to \cite{Frig} and \cite{Monod} for an account of the applications of bounded cohomology of discrete and locally compact groups, respectively. \\

The characterization that relates amenability and bounded cohomology, due to Trauber, is the following \cite[Theorem 2.5]{Johnson}: a group $G$ is amenable if and only if $H^n_{cb}(G, E) = 0$ for every dual normed $\mathbb{R}[G]$-module $E$, and all $n \geq 1$. Here a \textit{dual} normed $\mathbb{R}[G]$-module is the topological dual of a normed $\mathbb{R}[G]$-module equipped with the operator norm and the dual action of $G$. \\

In both of these definitions the underlying field is $\mathbb{R}$, so it becomes natural to ask whether changing this would lead to an interesting theory. Since both definitions are linked to real normed vector spaces, to study a substantially different notion we look at \textit{non-Archimedean} valued fields $\mathbb{K}$ and normed $\mathbb{K}$-vector spaces, which satisfy the ultrametric inequality. Indeed, it turns out that notions of amenability over such fields are much more restrictive. \\

We start by adapting the definition of mean. A $\mathbb{K}$-\textit{mean} on $G$ is a norm one linear map $m : C_b(G, \mathbb{K}) \to \mathbb{K}$ such that $m(\mathbbm{1}_G) = 1$. Here $C_b(G, \mathbb{K})$ denotes the space of $\mathbb{K}$-valued continuous bounded functions on $G$. The positivity requirement ($m(f) \geq 0$ whenever $f \geq 0$) is dropped because many valued fields do not admit an ordering compatible with the norm. Still, this is not more general than the original definition with $\mathbb{K} = \mathbb{R}$: it follows by a standard argument that the positivity requirement is superfluous \cite[Theorem 3.2]{Pier}. It seems then natural to define \textit{$\mathbb{K}$-amenability} of a group $G$ in terms of the existence of an invariant mean. \\

Since $\mathbb{K}$ is non-Archimedean, it is totally disconnected, and therefore every continuous map $G \to \mathbb{K}$ factors through the component group $\pi_0(G)$, which is totally disconnected and locally compact. Thus, the theory of $\mathbb{K}$-amenable groups reduces to that of totally disconnected locally compact (t.d.l.c.) ones. Therefore, if not mentioned otherwise, \textit{$G$ will always denote a t.d.l.c. group}.

The same reduction can be done with bounded cohomology. We analogously define a \textit{normed $\mathbb{K}[G]$-module}, the only difference being that normed vector spaces over $\mathbb{K}$ are required to satisfy the ultrametric inequality, as in \cite{NFA}. Then a normed $\mathbb{K}$-vector space $E$ is an ultrametric space, so totally disconnected, and so every $f \in C^n_b(G, E)$ factors through the component group $\pi_0(G)$. Once again, the theory of bounded cohomology over non-Archimedean valued fields reduces to the setting of t.d.l.c. groups. \\

The very first example shows that this notion of $\mathbb{K}$-amenability is too restrictive to admit a characterization in terms of bounded cohomology. Indeed, let $\mathbb{K} = (\mathbb{Q}_p, | \cdot |_p)$ be the field of $p$-adic numbers with the $p$-adic norm, and let $G$ be a finite group of order $n$. Then there exists a unique linear map $m : C_b(G, \mathbb{Q}_p) \to \mathbb{Q}_p$ that is invariant and satisfies $m(\mathbbm{1}_G) = 1$, but this map has norm $|n|_p^{-1}$ (Theorem \ref{finite}). In particular if $G$ admits elements of order $p$, then it is not $\mathbb{Q}_p$-amenable according to the previous definition. On the other hand, bounded cohomology with coefficients in a $\mathbb{Q}_p[G]$-module vanishes, independently of $p$ (every function on $G$ is automatically bounded).

As a more sophisticated example of how restrictive this notion is, it turns out that no topologically simple non-abelian group can be $\mathbb{Q}_2$-amenable according to this definition (Proposition \ref{2simple}). \\

In order to fix this, we relax the definition of amenability to obtain the following:

\begin{definition}
\label{def nka}

A \textit{normed $\mathbb{K}$-mean} on $G$ is a bounded linear map $m : C_b(G, \mathbb{K}) \to \mathbb{K}$ such that $m(\mathbbm{1}_G) = 1$. A group $G$ is \textit{normed $\mathbb{K}$-amenable} if it admits a normed left-invariant $\mathbb{K}$-mean. The infimum of the norms of such invariant means is denoted by $\| G \|_\mathbb{K}$, and we say that $G$ is \textit{$\mathbb{K}$-amenable of norm $\| G \|_\mathbb{K}$}. If $G$ is not normed $\mathbb{K}$-amenable, we write $\|G\|_\mathbb{K} = \infty$.
\end{definition}

\begin{remark}
\label{rem:attained}

The infimum $\| G \|_\mathbb{K}$ is clearly attained when $\mathbb{K}$ is \emph{discretely valued}, meaning that the norm takes values in a discrete subset of $\mathbb{R}$ (see Subsection \ref{preli field}). This is the case when $\mathbb{K} = \mathbb{Q}_p$. The general case is more subtle, but it is still true (Corollary \ref{norm attain}).
\end{remark}

The case of $\mathbb{K} = \mathbb{Q}_p$ will play a central role in this paper, so we rephrase the definitions in a more compact way:

\begin{definition}
\label{def npa}

A normed $\mathbb{Q}_p$-amenable group will be called \textit{normed $p$-adic amenable}, and we denote $\| G \|_p := \| G \|_{\mathbb{Q}_p}$.
\end{definition}

This is also an analogue of usual amenability over $\mathbb{R}$. Indeed, if a group $G$ admits a bounded linear map $0 \neq m : C_b(G, \mathbb{R}) \to \mathbb{R}$ that is $G$-invariant, then it is automatically amenable, even if this $m$ is not necessarily a mean. This fact is easily proven in the discrete case (see \cite[Lemma 3.2]{Frig} for a proof). In the continuous setting, it follows by taking advantage of the Riesz space structure on the dual space $C_b(G, \mathbb{R})^*$ \cite[20]{Riesz}. Indeed, given a left-invariant element $0 \neq m \in C_b(G, \mathbb{R})^*$, consider its absolute value $|m|$; then $|m|$ is non-zero, left-invariant, and positive. After normalizing it has norm 1, and then it follows by a standard argument that it must send $\mathbbm{1}_G$ to $1$ (see e.g. \cite[Theorem 3.2]{Pier}). Note how in this argument the ordering of $\mathbb{R}$ plays a crucial role.

On the other hand, in the case of $\mathbb{Q}_p$ (which cannot be ordered as a field) the introduction of a norm gives rise to a strictly weaker notion: indeed all finite groups are normed $p$-adic amenable, but as we mentioned before, a finite group with elements of order $p$ is not $\mathbb{Q}_p$-amenable in the sense of our previous definition. \\

This turns out to be the suitable analogue of canonical amenability to study bounded cohomology, at least for discrete groups. This is exemplified in the following result, which is the analogue of Trauber's characterization of amenability.

\begin{theorem}
\label{main2}

A discrete group $G$ is normed $\mathbb{K}$-amenable if and only if $H^n_b(G, E) = 0$ for every dual normed $\mathbb{K}[G]$-module $E$ and all $n \geq 1$.
\end{theorem}

The proof of the ``only if'' part of Theorem \ref{main2} will follow from a more general vanishing theorem, namely the vanishing of \textit{left uniformly continuous bounded cohomology}, beyond the discrete case (Theorem \ref{luc 0}). In particular, the ``only if'' part of Theorem \ref{main2} also holds for compact groups. In the real setting, continuous bounded cohomology and left uniformly continuous bounded cohomology are isomorphic, but so far we cannot prove the same in the non-Archimedean one. We also use Theorem \ref{main2} to prove a vanishing result for profinite representations of discrete groups admitting only finitely many virtual $p$-quotients (Corollary \ref{prof rep}), where $p$ is the characteristic of the residue field of $\mathbb{K}$ (which is just $p$ for the case of $\mathbb{Q}_p$). \\

At the beginning of this project we had started investigating $p$-adic bounded cohomology for discrete groups. We had proven directly that normed $p$-adic amenable groups are periodic and amenable, and proven Theorem \ref{main2} in the $p$-adic case. Only afterwards was Schikhof's paper \cite{Schik} pointed out to us by Nicolas Monod. Here the author characterizes $\mathbb{K}$-amenability in the sense of our first definition, where a mean is required to have norm $1$. By Corollary \ref{norm attain} (see Remark \ref{rem:attained}), a group is $\mathbb{K}$-amenable in the sense of Schikhof if and only if it is $\mathbb{K}$-amenable of norm $1$, in the sense of Definition \ref{def nka}.

In order to state Schikhof's characterization, we need two definitions:

\begin{definition}
A (t.d.l.c.) group $G$ is called \textit{locally elliptic} if it satisfies one of the following equivalent conditions:
\begin{enumerate}
\item $G$ is a directed union of compact open subgroups.
\item Every compact subset of $G$ is contained in a compact subgroup.
\item Every finite subset of $G$ is contained in a compact subgroup.
\end{enumerate}
\end{definition}

\begin{remark}
Locally elliptic groups are also called \textit{torsional} \cite{Schik, NHA} or \textit{compactly ruled} or \textit{topologically locally finite} \cite{period}. See \cite[Proposition 1.3]{period} for a proof of the equivalence of these three conditions.
\end{remark}

\begin{definition}
Let $n$ be an integer. The group $G$ is \emph{$n$-free} if for every pair of open subgroups $V \leq U \leq G$, the index $[U:V]$, whenever finite, is \emph{not} divisible by $n$. A group is \emph{$n^\mathbb{N}$-free} if it is $n^k$ free for some $k \geq 1$, hence $n^l$-free for any $l \geq k$. Clearly every group is 0-free.
\end{definition}

Here is Schikhof's characterization of $\mathbb{K}$-amenability in the case $\mathbb{K} = \mathbb{Q}_p$:

\begin{theorem}[Schikhof]
A group $G$ is $\mathbb{Q}_p$-amenable (in the sense of Schikhof) if and only if it is locally elliptic and $p$-free.
\end{theorem}

While weaker than Schikhof's notion of $\mathbb{K}$-amenability, our notion of normed $\mathbb{K}$-amenability still admits a strong characterization, which has moreover the advantage of giving an algebraic interpretation to the numerical invariant $\| G \|_\mathbb{K}$. Once again, we state it in the case $\mathbb{K} = \mathbb{Q}_p$.

\begin{theorem}
\label{main1}

A group $G$ is normed $p$-adic amenable if and only if it is locally elliptic and $p^\mathbb{N}$-free. In this case, $\|G\|_p = \min \{ p^k \geq 1 \mid G \text{ is } p^{k+1} \text{-free} \} = \max \{ p^k \mid \text{there exist open subgroups } V \leq U \leq G \text{ such that } p^k \mid [U:V] \}$.
\end{theorem}

Note that given open subgroups $V \leq U \leq G$ such that $[U : V]$ is finite, up to taking a finite-index open subgroup of $V$ we may assume that $V$ is normal in $U$. Then the condition $p^k \mid [U : V]$ is equivalent to the existence of a subgroup of order $p^k$ in the finite discrete group $U/V$, which must be of the form $U'/V$, for an open subgroup $V \leq U' \leq V$. Therefore $\| G \|_p$ is the maximal order of a finite discrete $p$-quotient of an open subgroup of $G$.

Since compact groups are amenable, and directed unions of amenable groups are amenable, it follows from Theorem \ref{main1} that every normed $p$-adic amenable group is amenable (Corollary \ref{npa amenable}), something that was already remarked by Schikhof for $\mathbb{Q}_p$-amenability \cite[End of Section 3]{Schik}. \\

We will provide many examples of normed $p$-adic amenable groups, ranging from matrix groups over locally finite fields or local fields, to groups of tree automorphisms, many of which do not have norm $1$. In particular we will provide, for each $p > 2$, an infinite discrete simple group $G$ such that $\| G \|_p = 1$ (Examples \ref{PSL p > 3} and \ref{Suzuki}), and an infinite discrete simple group $G$ such that $\| G \|_2 = 4$ (Example \ref{PSL p 2}). These examples are optimal, in the sense that there exists no non-abelian topologically simple group with $\| G \|_2 < 4$ (Proposition \ref{2simple}). \\

In Schikhof's paper the main theorem is proven by reducing $\mathbb{Q}_p$-amenability to $\mathbb{F}_p$-amenability, where $\mathbb{F}_p$ is seen as a trivially valued field. This simplifies somewhat the proof of the hardest implication, namely that $\mathbb{Q}_p$-amenable groups are locally elliptic. Such a reduction is not possible in our case, but some of the other techniques do carry over. In particular we will encounter an analogue of Reiter's property (Theorem \ref{Reiter}). Surprisingly, this analogue only requires, for each compact subset $K$ of $G$, a compactly supported function that varies by less than $1$ under elements of $K$, and directly implies the existence of invariant functions. By contrast, in the real setting, the whole point of Reiter's property is that one can choose an arbitrarily small threshold for almost-invariance, and that invariant functions do not necessarily exist \cite[2.6.B]{Pier}. \\

A version of Theorem \ref{main1} holds for a general non-Archimedean valued field $\mathbb{K}$ (see Subsection \ref{preli field} for the relevant definitions). The property of being locally elliptic plays the same role as in Theorem \ref{main1} in the characterization over spherically complete fields (Theorem \ref{main sc}), while for non-spherically complete fields this condition is replaced by compactness (Theorem \ref{main nsc}). We prove these theorems by reducing normed $\mathbb{K}$-amenability to normed $\mathbb{K}_0$-amenability, where $\mathbb{K}_0$ is the closure of the prime field of $\mathbb{K}$. When $\mathbb{K}_0$ is isomorphic to either $\mathbb{F}_p$ or $\mathbb{Q}$ (both trivially valued) we can use Schikhof's Theorem, and prove that the introduction of a norm does not lead to a weaker notion of amenability. The only other possibility is that $\mathbb{K}_0$ is isomorphic to $\mathbb{Q}_p$, in which case we can use Theorem \ref{main1}. Therefore the weakening of the notion of amenability is of interest only in this case, and normed $p$-adic amenability already contains the necessary information.

For non-Archimedean ordered fields with a compatible norm, one may add the positivity requirement present in the usual definition of amenability. We prove that, just as in the real setting, this additional requirement is superfluous for any non-trivially valued field (Proposition \ref{ordered}). \\

While more flexible than $\mathbb{K}$-amenability in the sense of Schikhof, normed $\mathbb{K}$-amenability is still a very restrictive condition, as shown by Theorem \ref{main1}. Therefore Theorem \ref{main2} suggests that there should be plenty of groups with rich bounded cohomology with coefficients in normed $\mathbb{K}$-vector spaces. Motivated by this, we look at bounded cohomology with trivial $\mathbb{K}$-coefficients. \\

To put the following results in perspective, the reader should note that bounded cohomology with trivial $\mathbb{R}$-coefficients is infamously hard to compute. In fact, until recently there were only few non-amenable discrete groups for which $H^\bullet_b(G, \mathbb{R})$ was computed in every degree \cite{MM, bcd} (several new examples have emerged in the year following the first version of this paper \cite{newcomp1, newcomp2, newcomp3, newcomp4}). By contrast, in the non-Archimedean setting we obtain strong results about the bounded cohomology of general groups, which allow to understand it completely in the case of discrete groups satisfying certain finiteness conditions. \\

Given a group $G$ and a $\mathbb{K}[G]$-module $E$, the inclusion of the bounded cochain complex into the standard one induces a map $c^\bullet : H^\bullet_{cb}(G, E) \to H^\bullet_c(G, E)$ from continuous bounded cohomology to ordinary continuous cohomology, called the \emph{comparison map}. When $E = \mathbb{K} = \mathbb{R}$ and $G$ is discrete, a fruitful (at least in low degrees) method for studying $H^n_b(G, \mathbb{R})$ is to understand separately the kernel and the image of $c^n$. The kernel can be very large for groups exhibiting some negative curvature: for instance if $G$ is acylindrically hyperbolic, the kernels of $c^2$ and $c^3$ are uncountably dimensional \cite{FPS}. On the other hand, the surjectivity of $c^n$ (especially for $n = 2$) is typically a hyperbolic feature, and when taking into account more general coefficient modules it provides a characterization of hyperbolic groups \cite{min1, min2}. \\

Contrary to the real case, when $\mathbb{K}$ is a non-Archimedean valued field the comparison map is very often injective. In fact, for a very large class of fields it is \emph{always} injective:

\begin{theorem}[Theorem \ref{qz split}]
Let $G$ be a group, $\mathbb{K}$ a non-Archimedean valued field such that $\mathbb{K}$ has either positive characteristic or a residue field of characteristic $0$ (the latter class includes all ordered fields with a compatible norm). Then the comparison map $c^n : H^n_{cb}(G, \mathbb{K}) \to H^n_c(G, \mathbb{K})$ is injective for all $n \geq 1$.
\end{theorem}

The other non-Archimedean valued fields are those that have characteristic $0$ and a residue field of positive characteristic, such as $\mathbb{Q}_p$. For those the comparison map is not always injective (in fact, there is always a group for which $c^2$ is not injective), but we can still say a lot, especially in degree $2$. In this case, the kernel of the comparison map $c^2 : H^2_{cb}(G, \mathbb{K}) \to H^2_c(G, \mathbb{K})$ is described by \emph{quasimorphisms}, namely continuous maps $f : G \to \mathbb{K}$ such that the quantity $|f(gh) - f(g) - f(h)|_\mathbb{K}$ is uniformly bounded. More precisely, the kernel of $c^2$ is the space of quasimorphisms modulo \emph{trivial} ones: those that are at a bounded distance from a homomorphism. Our analysis of quasimorphisms of $\mathbb{Q}_p$ will lead to a classification result (Theorem \ref{class qm}), which generalizes to finite extensions of $\mathbb{Q}_p$ (Corollary \ref{qm ext Qp}). These results show that the existence of a non-trivial quasimorphism on $G$ (equivalently, the non-injectivity of $c^2$) imposes strong restrictions on the subgroup structure of $G$, by means of a ``virtual descending chain condition''. \\

When restricting to discrete groups, we can prove injectivity of the comparison map in higher degrees for fields not covered by the previous theorem, under a mild homological finiteness condition. The same condition, shifted by one degree, implies surjectivity as well:

\begin{theorem}[Proposition \ref{surj_comp}, Corollary \ref{inj Fn}]
\label{main_comp}

Let $\mathbb{K}$ be a non-Archimedean valued field of characteristic $0$. Let $G$ be a discrete group such that $H_n(G, \mathbb{Z})$ is a finitely generated abelian group. Then $c^n : H^n_b(G, \mathbb{K}) \to H^n(G, \mathbb{K})$ is surjective, and $c^{n+1} : H^{n+1}_b(G, \mathbb{K}) \to H^{n+1}(G, \mathbb{K})$ is injective.

The same holds for non-Archimedean valued fields of characteristic $p$, with $\mathbb{Z}$ replaced by $\mathbb{F}_p$.
\end{theorem}

This theorem implies that if $G$ is a finitely presented group, then $c^2$ is an isomorphism. If $G$ is only finitely generated, then it still tells us that $c^2$ is injective. With a different approach, this can be generalized to compactly generated groups and general coefficient modules:

\begin{proposition}[Corollary \ref{inj comp}]
Let $\mathbb{K}$ be a non-Archimedean valued field and $G$ a compactly generated group. Then for every normed $\mathbb{K}[G]$-module $E$, the comparison map $c^2 : H^2_{cb}(G, E) \to H^2_c(G, E)$ is injective.
\end{proposition}

A consequence of a more topological flavour is the following: the comparison map is an isomorphism in all degrees for \emph{groups of type $F_\infty$}; that is, fundamental groups of aspherical CW-complexes with a finite number of cells in each dimension. This result can also be justified by looking at bounded cohomology of a topological space $X$ with coefficients in a non-Archimedean valued field $\mathbb{K}$, which we denote by $H^n_b(X; \mathbb{K})$. This is defined in terms of the complex of \emph{bounded singular cochains}, that is, $\mathbb{K}$-valued bounded maps on the set of singular simplices in $X$. We can similarly define bounded cohomology of pairs $(X, A)$, denoted by $H^n_b(X, A; \mathbb{K})$. \\

Bounded cohomology of topological spaces is very useful in the real setting, for instance for the computation of the simplicial volume of a manifold. However it is very hard to compute, given the absence of a Mayer--Vietoris sequence. Indeed, the proof of the excision axiom of singular cohomology breaks down when restricting to bounded cochains, because it goes through the process of barycentric subdivision, which can replace a given singular simplex by an arbitrarily large number thereof. This is however not a problem when the field we are working on is equipped with an ultrametric. We will prove:

\begin{theorem}[Theorem \ref{main_top}]
Bounded cohomology with coefficients in a non-Archimedean valued field $\mathbb{K}$ satisfies the Eilenberg--Steenrod axioms for (unreduced) cohomology, except additivity. This is replaced by an injective homomorphism
$$\prod\limits_{i \in I} H^n_b(j_i) : H^n_b \left( \bigsqcup\limits_{i \in I} X_i , \bigsqcup\limits_{i \in I} A_i; \mathbb{K} \right) \to \prod\limits_{i \in I} H^n_b(X_i, A_i; \mathbb{K})$$
whose image is the subspace of sequences admitting representatives of uniformly bounded norm.
\end{theorem}

Note that this weaker form of additivity makes a difference only when considering infinite disjoint unions. Therefore, for CW-complexes \emph{of type $F_\infty$}, that is, having only finitely many cells in each dimension, the proof of the isomorphism between cellular and singular homology carries over. We obtain:

\begin{corollary}[Corollary \ref{cell}]
Let $\mathbb{K}$ be a non-Archimedean valued field and $X$ a CW-complex of type $F_\infty$. Then $H^n_b(X; \mathbb{K})$ is naturally isomorphic to the cellular cohomology of $X$ with coefficient group $\mathbb{K}$.
\end{corollary}

In particular the comparison map is an isomorphism in all degrees for such spaces. Together with the aforementioned consequence of Theorem \ref{main_comp}, and the corresponding fact for ordinary cohomology, we deduce that if $X$ is an aspherical CW-complex of type $F_\infty$, then the bounded cohomology of $X$ is naturally isomorphic to the bounded cohomology of $\pi_1(X)$ with coefficients in $\mathbb{K}$. This is true in general, as the proof in the real case carries over (see \cite[Theorem 5.5]{Frig} for a proof, and Subsection \ref{ss_bcXG} for a more detailed discussion). However our approach in this case is novel and very specific to the non-Archimedean nature of the setting.

Note that this corollary implies that the asphericity hypothesis is necessary: an oriented closed simply connected $n$-manifold is a counterexample. Therefore an analogue of Gromov's Mapping Theorem (stating that bounded cohomology with real coefficients is a $\pi_1$-invariant \cite{Grom, Ivanov, multicomplexes}) cannot hold over non-Archimedean valued fields.  \\

Given such strong results on bounded cohomology of topological spaces, it is tempting to try and apply them to a non-Archimedean notion of simplicial volume. We will show that a natural notion that is compatible with bounded cohomology via duality - in contrast to the one defined in \cite{pSV} - is identically $1$ (Proposition \ref{NASV}). \\

\textbf{Outline.} We start in Section \ref{preli} with some preliminaries on non-Archimedean valued fields and vector spaces, t.d.l.c. groups, harmonic analysis over $\mathbb{Q}_p$ and divisible groups. In Section \ref{constructions} we start our study of normed $\mathbb{K}$-amenability, treating the case of compact groups and showing how the norm behaves under the constructions of taking subgroups, quotients, extensions and directed unions. In Section \ref{sequiv} we focus on $\mathbb{Q}_p$, proving Theorem \ref{main1} and some corollaries. Section \ref{examples} is devoted to giving examples of normed $p$-adic amenable groups. We then move to other fields in Section \ref{other fields}, and prove the analogues of Theorem \ref{main1} in full generality. In Section \ref{smain2} we prove Theorem \ref{main2}, with an application to profinite representations of discrete groups. In Section \ref{BC} we discuss further bounded cohomology with $\mathbb{K}$-coefficients, focusing on quasimorphisms in Section \ref{s_qm}. Bounded cohomology of topological spaces is treated in Section \ref{s_top}, which is mostly independent. \\

\textbf{Remark.} The results in this paper are part of the author's PhD project. \\

\textbf{Acknowledgements.} I will start by thanking the two people at the origin of this project: my advisor Alessandra Iozzi and my quasi-advisor Konstantin Golubev, who pushed me to study bounded cohomology over the reals, and explore it over the $p$-adics. Thank you to them and to Marc Burger for the many useful discussions throughout the development of this project, in Berkeley, in Z\"urich and over Zoom.

This work benefited from conversations with various people, from casual discussions after lunch to important suggestions that strongly determined the direction of the paper. For this I thank Ilaria Castellano, Roberto Frigerio, Yannick Krifka, Adrien Le Boudec, Clara L\"oh, Nicol\'as Matte Bon, Nicolas Monod, Marco Moraschini, Maxim Mornev, Francesco Russo and Alessandro Sisto.

T.d.l.c. groups were the subject of a graduate course taught by Waltraud Lederle at ETH, that I took during my Master's degree: I would not have gone further than discrete groups in this work were it not for her enthusiastic teaching. She was also extremely available and helpful by providing examples and helping me navigate the literature.

\pagebreak

\section{Preliminaries}
\label{preli}

\textbf{Notations and conventions.} In the sequel, unless stated otherwise, $G$ will always denote a t.d.l.c. group. For a $p$-adic number $x$ the $p$-adic norm of $x$ will be denoted by $|x|_p$. For an integer $x$, we will denote by $\nu_p(x)$ the $p$-adic valuation of $x$, so that $|x|_p = p^{- \nu_p(x)}$. For general fields $\mathbb{K}$, we will include them in the notation, writing $|\cdot|_\mathbb{K}$ for the norm. 

Given a group $G$ and a subset $Y \subset G$, the characteristic function of $Y$ is denoted by $\mathbbm{1}_Y$. If $Y$ is clopen, then $\mathbbm{1}_Y$ is continuous. 

Natural numbers are assumed to start from $1$, their set is denoted by $\mathbb{N}$. 

The identity element of a group $G$ will be denoted by $1$, unless the group is abelian, in which case it will be denoted by $0$.

\subsection{Non-Archimedean valued fields}
\label{preli field}

For proofs and more detail, see \cite[Section 1]{NFA} or \cite[Section 1]{NOT}. For a proof of Ostrowski's Theorem see \cite[Theorem 1.50]{Katok}. Let $\mathbb{K}$ be a field, and let $\chi(\mathbb{K})$ denote its characteristic.

\begin{definition}
A \emph{non-Archimedean norm} or \emph{value} on $\mathbb{K}$ is a map $| \cdot |_\mathbb{K} : \mathbb{K} \to \mathbb{R}_{\geq 0}$ such that:
\begin{enumerate}
\item $|x|_\mathbb{K} = 0$ if and only if $x = 0$;
\item $|x + y|_\mathbb{K} \leq \max \{ |x|_\mathbb{K}, |y|_\mathbb{K} \}$;
\item $|xy|_\mathbb{K} = |x|_\mathbb{K} \cdot |y|_\mathbb{K}$.
\end{enumerate}
A field equipped with a non-Archimedean norm is a \emph{non-Archimedean valued field}.

\begin{remark}
We insist on \emph{valued} to make the distinction with non-Archimedean \emph{ordered} fields. These will make a short appearance in Subsection \ref{ss:ordered}.
\end{remark}

The norm induces an ultrametric on $\mathbb{K}$, which in turn induces a totally disconnected topology. If this metric is complete, we say that $\mathbb{K}$ is a \emph{non-Archimedean complete valued field}. The map $|\cdot|_\mathbb{K} : \mathbb{K}^\times \to \mathbb{R}^\times$ is a group homomorphism, thus the image is a subgroup of $\mathbb{R}^\times$, which is either dense or discrete. In the latter case, we say that $\mathbb{K}$ is \emph{discretely valued}.
\end{definition}

Two norms $| \cdot |_{1, 2}$ on $\mathbb{K}$ are \emph{equivalent} if there exists some $\alpha \in \mathbb{R}_{>0}$ such that $|\cdot|_1^\alpha = |\cdot|_2$ (there are other equivalent definitions, see \cite[Proposition 1.10]{Katok}). Every field admits a \emph{trivial} non-Archimedean norm, namely the characteristic function of $\mathbb{K}^\times$, which is the only norm in its equivalence class.

Multiplicativity implies that $| 1 |_\mathbb{K} = 1$, that the norm commutes with taking inverses, and that the only norm on $\mathbb{F}_p$ is the trivial one. In characteristic 0, the field $\mathbb{Q}$ admits as norms the trivial one and the $p$-adic ones (see the next subsection). \emph{Ostrowski's Theorem} states that, up to equivalence, these are the only non-Archimedean norms on $\mathbb{Q}$. \\

Given a non-Archimedean valued field $\mathbb{K}$, let $\mathbb{K}_0$ be the \emph{closure of its prime field}, that is, the closure of the subfield of $\mathbb{K}$ generated by $1$. This is also a non-Archimedean valued field, and if $\mathbb{K}$ is complete, so is $\mathbb{K}_0$. In this case, if $\chi(\mathbb{K}) = p > 0$, then $\mathbb{K}_0 = \mathbb{F}_p$ with the trivial norm. If instead $\chi(\mathbb{K}) = 0$, then $\mathbb{K}_0$ is a non-Archimedean completion of $\mathbb{Q}$; therefore by Ostrowski's Theorem, up to equivalence, it is either $\mathbb{Q}$ with the trivial norm, or $\mathbb{Q}_p$ with the $p$-adic norm. We will see in the next subsection that if the norm on $\mathbb{K}_0$ is trivial, this does not imply that the norm on $\mathbb{K}$ is trivial, and if the norm on $\mathbb{K}_0$ is the $p$-adic norm, this does not imply that $\mathbb{K}$ is discretely valued.

There are some cases in which $\mathbb{K}$ is not complete, but the closure of the prime field is, in which case we use the same notation and have the same trichotomy. An example is the field of rational functions: see the next subsection.

\begin{definition}
\label{def sph}
The non-Archimedean valued field $\mathbb{K}$ is \emph{spherically complete} if every decreasing sequence of balls $B_1 \supseteq B_2 \supseteq \cdots$ in $\mathbb{K}$ has nonempty intersection. A spherically complete field is complete.
\end{definition}

Every complete discretely valued field is spherically complete \cite[Proposition 1.86]{NOT}; this is the case in particular for \emph{local fields}, which will be discussed in the next subsection. The completion $\mathbb{C}_p$ of the algebraic closure of $\mathbb{Q}_p$ is not spherically complete, even though it is complete \cite[Proposition 1.94]{NOT}.

\begin{definition}
The closed 1-ball in $\mathbb{K}$ is a ring, called the \emph{ring of integers} and denoted by $\mathfrak{o}$. Its units are $\mathfrak{o}^\times = \{ x \in \mathbb{K} \mid |x|_\mathbb{K} = 1 \}$. The open 1-ball, which coincides with the non-units, is the unique maximal ideal of $\mathfrak{o}$, denoted by $\mathfrak{m}$. It is a principal ideal if and only if $\mathbb{K}$ is discretely valued, and in this case a generator is called a \emph{uniformizer}, generally denoted by $\pi$. The \emph{residue field} of $\mathbb{K}$ is the field $\mathfrak{r} := \mathfrak{o}/\mathfrak{m}$. The \emph{residual characteristic} of $\mathbb{K}$ is the characteristic of $\mathfrak{r}$, denoted by $\chi(\mathfrak{r})$.
\end{definition}

\begin{remark}
Although the terminology reflects the one in the real case, $\mathfrak{o}$ and $\mathfrak{m}$ are both clopen.
\end{remark}

Note that $\chi(\mathfrak{r}) = 0$ if for all $n \in \mathbb{N}$ we have $| n \cdot 1 |_{\mathbb{K}} = 1$, and $\chi(\mathfrak{r}) = \min \{ p \in \mathbb{N} \mid |p \cdot 1|_\mathbb{K} < 1 \}$ otherwise. Up to equivalence of norms, if $\chi(\mathfrak{r}) = p > 0$ but $\chi(\mathbb{K}) = 0$, we may always assume that $|p|_\mathbb{K} = p^{-1}$. It then follows easily that the restriction of $|\cdot|_\mathbb{K}$ to $\mathbb{Q}$ is precisely the $p$-adic norm: see the proof of Ostrowski's Theorem \cite[Theorem 1.50]{Katok}. If $\mathbb{K}$ is complete, so $\mathbb{K}_0 \cong \mathbb{F}_p, \mathbb{Q}$ or $\mathbb{Q}_p$, this convention implies that the restriction of the norm on $\mathbb{K}$ to $\mathbb{K}_0$ is the standard norm (trivial or $p$-adic). This is automatic in the first two cases, since the trivial norm is the only one in its equivalence class. Note that when $\mathbb{K}$ is complete, $\chi(\mathfrak{r})$ only depends on $\mathbb{K}_0$, so $\chi(\mathfrak{r}) = 0$ if $\mathbb{K}_0 \cong \mathbb{Q}$, and $\chi(\mathfrak{r}) = p > 0$ if $\mathbb{K}_0 \cong \mathbb{F}_p$ or $\mathbb{K}_0 \cong \mathbb{Q}_p$.

\subsection{Examples}

We refer to \cite[1.2.2-1.3]{NOT} for proofs and more details. \\

We have already seen the examples of trivial norm on any field, which makes it discretely valued and spherically complete. \\

The \emph{$p$-adic norm} on $\mathbb{Q}$ is defined by $|a/b|_p = |p^k \cdot x/y|_p = p^{-k}$, where $p$ does not divide $x$ nor $y$. This is a non-Archimedean norm, and $\mathbb{Q}_p$ is the corresponding completion. It is a discretely valued field, since the norm takes values in $p^{\mathbb{Z}} \cup \{ 0 \}$. Being complete and discretely valued, it is spherically complete. The ring of integers is $\mathbb{Z}_p$, which is equivalently described as the closure of $\mathbb{Z} \subset \mathbb{Q} \subset \mathbb{Q}_p$, or abstractly as the inverse limit of the system of finite groups $\{ \mathbb{Z} / p^k \mathbb{Z} \mid k \geq 1 \}$ with the reduction maps. The maximal ideal is $p \mathbb{Z}_p$, a uniformizer is $p \in \mathbb{Z}_p$, and the residue field is $\mathbb{Z}_p / p \mathbb{Z}_p \cong \mathbb{F}_p$. \\

A norm on a non-Archimedean valued field extends uniquely to its algebraic closure. In particular the norm on $\mathbb{Q}_p$ extends uniquely to $\overline{\mathbb{Q}_p}$. This valued field is not complete, its completion is denoted by $\mathbb{C}_p$, which turns out to also be algebraically closed. Every intermediate field $\mathbb{K}$, such as any finite algebraic extension of $\mathbb{Q}_p$, satisfies $\mathbb{K}_0 \cong \mathbb{Q}_p$. Since the residue field of $\mathbb{Q}_p$ is $\mathbb{F}_p$, it follows that every such field has residual characteristic $p$. However it need not be equal to $\mathbb{F}_p$: the norm on $\mathbb{K}$ may take more values than just $p^\mathbb{Z}$, and this will be reflected in the residue class field. For instance, $\overline{\mathbb{Q}_p}$ is not discretely valued, and the residue field is the algebraic closure $\overline{\mathbb{F}_p}$ of $\mathbb{F}_p$. \\

Let $\mathbb{F}$ be an abstract field. Consider the field $\mathbb{F}(X)$ of rational functions, and define the norm $|P/Q| = e^{deg(P) - def(Q)}$ (any other basis larger than 1 will define an equivalent norm). The restriction of this norm to $\mathbb{F} \subset \mathbb{F}(X)$ is the trivial one, in particular it makes $\mathbb{F}$ into a complete field. So even if $\mathbb{F}(X)$ is not complete, it still makes sense to look at $\mathbb{F}(X)_0 = \mathbb{F}_0$, which is either $\mathbb{F}_p$ or $\mathbb{Q}$ with the trivial norm, so $\chi(\mathbb{F}_0) = \chi(\mathbb{F})$. This shows that the norm on the closure of the prime field can be trivial even when the norm on the field is non-trivial.

Now $\mathfrak{o} = \{ P / Q \mid deg(P) \leq deg(Q) \}$ and $\mathfrak{m} = \{ P / Q \mid deg(P) < deg(Q) \}$. Given $P/Q \in \mathfrak{o}$, we can perform Euclidean division to obtain $P = \lambda Q + R$ where $deg(R) < deg(Q)$ and $\lambda = 0$ if and only if $deg(P) < deg(Q)$. Then $P/Q \equiv \lambda \mod \mathfrak{m}$. This implies that the map $\mathfrak{r} = \mathfrak{o}/\mathfrak{m} \to \mathbb{F} : P/Q \mod \mathfrak{m} \mapsto \lambda$ is an isomorphism. In particular this shows that $\chi(\mathfrak{r}) = \chi(\mathbb{F})$, as predicted by the isomorphism type of $\mathbb{F}(X)_0$.

This field is discretely valued, but not complete, so in particular it is not spherically complete. Its completion is the field of Laurent series $\mathbb{F}((1/X))$, which is our next example (up to replacing $X$ by $1/X$, which ``reverses'' the norm). \\

Let $\mathbb{F}$ be a field, and let $\mathbb{F}((X))$ be the field of Laurent series on $\mathbb{F}$. We define a norm as $|\sum\limits a_i X^i| = e^{-n}$, where $n$ is the smallest integer such that $a_n \neq 0$. Then $\mathfrak{o} = \mathbb{F}[[X]]$, the ring of Laurent polynomials, and $\mathfrak{m} = X \cdot \mathbb{F}[[X]]$, so $\mathfrak{r} \cong \mathbb{F}$. This field is complete and discretely valued, therefore it is spherically complete.

This applies in particular to $\mathbb{F} = \mathbb{F}_q$, where $q$ is a prime power. By convention we let $q$ be the basis for the norm instead of $e$. In this case the topology is moreover locally compact. Similarly to the case of $\mathbb{Q}_p$, this norm extends uniquely to algebraic extensions. \\

A class of fields of special interest is the following.

\begin{definition}
A valued field is called a \emph{local field} if the induced topology is locally compact and non-discrete.
\end{definition}

This includes the Archimedean examples of $\mathbb{R}$ and $\mathbb{C}$. The non-Archimedean ones are precisely the complete discretely valued fields with finite residue field. These are classified, and split into two families that we have already described: they are finite extension of either $\mathbb{Q}_p$ or $\mathbb{F}_q((X))$ \cite[I.3-I.4]{Weil}. Since a local field is complete and discretely valued, it is spherically complete (see the discussion after Definition \ref{def sph}).

\subsection{T.d.l.c. groups}

For more general information on t.d.l.c. groups, see \cite[Section 2]{Walt}. For the subclass of profinite groups, which will be relevant to this paper, see \cite{prof}. \\

We will often use the following fundamental result about t.d.l.c. groups:

\begin{theorem}[Van Danzig]

Let $G$ be a t.d.l.c. group. Then the collection of compact open subgroups of $G$ is a neighbourhood basis of the identity.
\end{theorem}

One can say more in the case of profinite groups. A group is \emph{profinite} if it is an inverse limit of finite groups. The class of compact totally disconnected groups coincides with the class of profinite groups: see \cite[Theorem 2.1.3]{prof} for a proof in a more general context.

\begin{corollary}
\label{cor_vd}

Let $G$ be a profinite group. Then the collection of (compact) open \emph{normal} subgroups of $G$ is a neighbourhood basis of the identity.
\end{corollary}

\begin{proof}
An open subgroup of $G$ (which is automatically closed, so compact) has finite index because $G$ is compact. Every open finite-index subgroup of a topological group contains an open normal finite-index subgroup.
\end{proof}

We will need the following lemmas.

\begin{lemma}
\label{cinco}

Let $G$ be a t.d.l.c. group, $H$ a compact subgroup. Then $H$ is contained in some compact open subgroup.
\end{lemma}

\begin{proof}
By (a very weak version of) Van Danzig's Theorem, there exists a compact open subgroup $V \leq G$. Let $N := \bigcap\limits_{g \in H} gVg^{-1}$. This is a compact subgroup of $G$ which is normalized by $H$. Therefore $HN$ is a compact group containing $H$. We will show that $N$ is open, from which it will follow that $HN$ is open.

Let $g \in H$. Since $g^{-1} \cdot 1 \cdot g = 1 \in V$, and $V$ is open, there exist neighbourhoods $W_g$ of $g$ and $V_g$ of 1 such that $W_g^{-1} V_g W_g \subseteq V$. Now $(W_g)_{g \in H}$ is an open cover of $H$. By compactness there exist $g_1, \ldots, g_n \in H$ such that $H \subseteq \bigcup\limits_{i = 1}^n W_{g_i}$. Then for all $g \in W_{g_i}$, we have $g^{-1} V_{g_i} g \subseteq V$. Thus
$$N = \bigcap\limits_{g \in H} g V g^{-1} \supseteq \bigcap\limits_{i = 1}^n \bigcap\limits_{g \in W_{g_i}} g V g^{-1} \supseteq \bigcap\limits_{i = 1}^n V_{g_i}.$$
We showed that $N$ contains an open set, so it is open.
\end{proof}

\begin{lemma}
\label{lopinco}

Let $G$ be a profinite group, $H \leq G$ a closed subgroup. Then $H = \bigcap\limits_{H \leq O \text{ open}} O$.
\end{lemma}

\begin{proof}
The inclusion $\subseteq$ is clear. Now let $g$ be an element of the right-hand-side set, and suppose by contradiction that $g \notin H$. Since $H$ is closed, by Corollary \ref{cor_vd} there exists some compact open normal subgroup $U \leq G$ such that $gU \cap H = \emptyset$. Since $U$ is normal, $HU$ is an open subgroup of $G$ containing $H$, so it appears in the intersection, which implies that $g \in HU$. Let $h \in H$ be such that $g \in h U$. Then $h \in g U$, which contradicts $gU \cap H = \emptyset$.
\end{proof}

\begin{lemma}
\label{opinco}

Let $G$ be a t.d.l.c. group, $H$ a compact subgroup and $V$ an open subgroup of $H$. Then there exists a compact open subgroup $U$ of $G$ such that $V = H \cap U$.
\end{lemma}

\begin{proof}
Let $K$ be a compact open subgroup of $G$ containing $H$, which exists by Lemma \ref{cinco}. Since $K$ is profinite, and all open subgroups of $K$ are also open in $G$, by Lemma \ref{lopinco} we have
$$V = \bigcap\limits_{V \leq O \text{ open in } K} O = \bigcap\limits_{V \leq O \text{ compact open in } G} O.$$
Now $V$ is open in $H$, so $H \, \backslash \, V$ is compact. Each $O$ from the intersection above is open in $G$, so closed, so $H \, \backslash \, O$ is also open in $H$. Using the equality above, and extracting a subcover by compactness:
$$H \, \backslash \, V = \bigcup\limits_{V \leq O \text{ compact open in } G} \left( H \, \backslash \, O \right) = \bigcup\limits_{i = 1}^n \left( H \, \backslash \, O_i \right) = H \, \backslash \, \left( \bigcap\limits_{i = 1}^n O_i \right) .$$
Therefore $V = H \bigcap \left (\bigcap\limits_{i = 1}^n O_i \right)$, and $U := \bigcap\limits_{i = 1}^n O_i$ is a compact open subgroup of $G$ as we wanted.
\end{proof}

\subsection{The local prime content}

The local prime content is a local invariant introduced in \cite{lpc_def}, which is very useful in the study of the local structure of t.d.l.c. groups. The following definition is equivalent to the original one \cite[Lemma 2.3]{lpc_char}:

\begin{definition}
Let $G$ be a t.d.l.c. group, $p$ a prime. We say that $p$ belongs to the \emph{local prime content} of $G$ if there exists a nested sequence
$$ \cdots \leq U_{-n} \leq \cdots \leq U_{-1} \leq U_0$$
of compact open subgroups of $G$ such that, for all $k \geq 1$, there exists $n \geq 1$ such that $p^k \mid [U_0 : U_{-n}]$. We denote by $\mathbb{L}(G)$ the local prime content of $G$.
\end{definition}

\begin{remark}
The indexing may seem backwards, but it will look natural when we relate this definition to $p^\mathbb{N}$-freeness.
\end{remark}

The local prime content is a local invariant, meaning that if $U \leq G$ is open, then $\mathbb{L}(U) = \mathbb{L}(G)$. In particular, a discrete group always has empty local prime content. \\

Although this definition was introduced in 2006, an equivalent definition was present in the literature as early as 1967 \cite[Definition 2.1.6]{NHA}:

\begin{definition}
Let $G$ be a t.d.l.c. group, $p$ a prime. We say that $G$ is \emph{$p$-finite from above} if there exists a compact open subgroup $U \leq G$ which is $p$-free: that is, for every open subgroup $V \leq U$, the (finite) index $[U : V]$ is not divisible by $p$.
\end{definition}

\begin{lemma}
\label{lpc pfin}

The group $G$ is $p$-finite from above if and only if $p \notin \mathbb{L}(G)$.
\end{lemma}

\begin{proof}
Suppose that $G$ is $p$-finite from above, and let $U$ be a compact open $p$-free subgroup of $G$. Then $p \notin \mathbb{L}(U) = \mathbb{L}(G)$.

Suppose that $p \notin \mathbb{L}(G)$. Then given a nested sequence as in the definition of $\mathbb{L}$, there exists $n\geq 1$ such that for all $m \geq n$ we have $|[U_0 : U_{-m}]|_p = |[U_0 : U_{-n}]|_p$. It follows that $U_{-n}$ is $p$-free, and so $G$ is $p$-finite from above.
\end{proof}

We will see in Theorem \ref{Haar} that this notion admits yet another equivalent characterization, in terms of the $p$-adic Haar integral. \\

Our next goal is to give a characterization of $p^\mathbb{N}$-freeness of locally elliptic groups, which relates it to the local prime content. It makes sense to restrict to locally elliptic groups in light of Theorem \ref{main1}: normed $p$-adic amenable groups will turn out to be locally elliptic (Theorem \ref{main1}).

\begin{lemma}
\label{free_co}
Let $G$ be a locally elliptic group. Suppose that $G$ is not $p^k$-free, that is, there exist open subgroups $V \leq U \leq G$ such that $p^k \mid [U:V]$. Then $U$ and $V$ may be chosen to be compact (and open).
\end{lemma}

\begin{proof}
Let $V \leq U \leq G$ be open subgroups of $G$ such that $p^k \mid [U : V] < \infty$. Up to passing to a finite-index subgroup, we may assume that $V$ is normal in $U$. The quotient $U/V$ is a finite discrete group, which by the first Sylow Theorem admits a subgroup of order $p^k$. Choose a representative for each element of this group, and denote by $P$ the closure of the subgroup of $G$ that they generate. Since $G$ is locally elliptic, $P$ is compact. Moreover, $p^k$ divides the order of $PV / V \cong P / P \cap V$. Now $P \cap V$ is an open subgroup of the compact group $P$. By Lemma \ref{opinco}, there exists a compact open subgroup $W \leq G$ such that $P \cap V = P \cap W$. Let $N := \cap_{g \in P} \, gWg^{-1} \leq W$. The same argument as in Lemma \ref{cinco} shows that $N$ is a compact open subgroup of $G$ which is normalized by $P$. Therefore $PN \leq G$ is compact and open, and $|PN/N| = |P / P \cap N|$ is a multiple of $|P / P \cap W| = |P / P \cap V|$, which in turn is a multiple of $p^k$. Since $PN$ and $N$ are both compact and open, we conclude.
\end{proof}

\begin{proposition}
\label{free lpc}

Let $G$ be a locally elliptic t.d.l.c. group. Then the following are equivalent:
\begin{enumerate}
\item $G$ is not $p^\mathbb{N}$-free.
\item There exists a nested sequence
$$ \cdots \leq U_{-n} \leq \cdots \leq U_{-1} \leq U_0 \leq U_1 \leq \cdots \leq U_n \leq \cdots$$
of compact open subgroups of $G$ such that, for all $k \geq 1$, there exists $n \geq 1$ such that $p^k \mid [U_n : U_{-n}]$.
\item Either $p \in \mathbb{L}(G)$, or there exists a nested sequence
$$ U_0 \leq U_1 \leq \cdots \leq U_n \leq \cdots$$
of compact open subgroups of $G$ such that, for all $k \geq 1$, there exists $n \geq 1$ such that $p^k \mid [U_n : U_0]$.
\end{enumerate}
\end{proposition}

\begin{proof}
Clearly $3.$ implies $1.$ Given a sequence as in $2.$, we have $|[U_n : U_{-n}]|_p \to 0$. Since $|[U_n : U_{-n}]|_p = |[U_n : U_0]|_p \cdot |[U_0 : U_{-n}]|_p$, and both these quantities are decreasing, one of them must converge to 0. This proves that $2.$ implies $3.$, using the definition of local prime content.

We are left to show that $1.$ implies $2.$ So suppose that $G$ is not $p^\mathbb{N}$-free. By Lemma \ref{free_co}, for all $k \geq 1$ there exist compact open subgroups $V_{-k} \leq V_k \leq G$ such that $p^k \mid [V_k : V_{-k}]$. We define inductively the sequence $(U_k)_{k \in \mathbb{Z}}$. First, we set $U_i = V_i$ for $i \in \{ \pm 1 \}$ (and, say, $U_0 = V_1$). Assume we have constructed $U_{-k} \leq \cdots \leq U_k$ so that $p^k \mid [U_k : U_{-k}]$. Set $U_{k+1}$ to be the subgroup generated by $U_k$ and $V_{k+1}$. Since $G$ is locally elliptic, $U_{k+1}$ is compact, and moreover it is open since it contains the open subgroup $U_k$. Set $U_{-(k+1)} := V_{-(k+1)} \cap U_{-k}$, which again is compact and open. Then $p^{k+1} \mid [V_{k+1} : V_{-(k+1)}] \mid [U_{k+1} : U_{-(k+1)}]$, and $U_{-(k+1)} \leq U_{-k} \leq U_k \leq U_{k+1}$.
\end{proof}

\subsection{Normed vector spaces over non-Archimedean valued fields}
\label{preli Banach}

For a detailed account, see \cite{NFA} or \cite{NOT}. Let $(\mathbb{K}, |\cdot|_\mathbb{K})$ be a non-Archimedean valued field.

\begin{definition}
\label{def_nvs}

Let $E$ be a $\mathbb{K}$-vector space. A \emph{norm} on $E$ is a map $\| \cdot \| : E \to \mathbb{R}_{\geq 0}$ such that:
\begin{enumerate}
\item $\|x\| = 0$ if and only if $x = 0$;
\item $\|x + y\| \leq \max\{\|x\|, \|y\|\}$;
\item $\|\alpha x\| = |\alpha|_\mathbb{K} \cdot \|x\|$ for all $\alpha \in \mathbb{K}$.
\end{enumerate}
\end{definition}

The most important difference with real normed vector spaces is that the triangle inequality is replaced by the ultrametric inequality. This has a useful implication:

\begin{lemma}
\label{ultrametric cor}

Let $x, y \in E$ be such that $\| x \| \neq \| y \|$. Then $\| x + y \| = \max \{ \|x\|, \|y\| \}$.
\end{lemma}

\begin{proof}
Suppose without loss of generality that $\| x \| < \| y \|$. Then
$$\| y \| = \| (x + y) - x \| \leq \max \{ \|x + y\|, \|x\| \} \leq \|y\|.$$
It follows that $\|x + y\| = \|y\|$.
\end{proof}

\begin{warning}
\label{multiples}

Notice that we do not require that the norm takes values in $|\mathbb{K}|_\mathbb{K}$. It follows that vectors in $E$ do not necessarily have multiples of norm one.
\end{warning}

\begin{definition}
A \emph{normed $\mathbb{K}$-vector space} is a $\mathbb{K}$-vector space $E$ equipped with a norm and the corresponding metric. If this metric is complete, $E$ is called a \emph{$\mathbb{K}$-Banach space}.

Given a group $G$, a \emph{normed $\mathbb{K}[G]$-module} is a normed $\mathbb{K}$-vector space on which $G$ acts by linear isometries. Notice that in this definition there is no continuity requirement on the action.
\end{definition}

The metric on a normed $\mathbb{K}$-vector space is an ultrametric, so the induced topology is totally disconnected. \\

Linear maps behave as nicely as in the real case: a linear map $T : E \to F$ between two normed $\mathbb{K}$-vector spaces is continuous if and only if it is \emph{bounded}; that is, there exists $C \geq 0$ such that for all $x \in E$ we have $\|Tx\|_F \leq C \|x\|_E$.

Given a normed $\mathbb{K}$-vector space $E$, we can consider its \emph{topological dual} $E^* := \{ f : E \to \mathbb{K} \mid f \text{ linear and continuous} \}.$ There is a subtlety on how to define the dual norm. One has two natural possibilities:
$$\|T\|_{op} := \inf \{ C \geq 0 \mid |Tx|_\mathbb{K} \leq C\|x\|_E \} ; \,\,\,\,\,\,\, \|T\|_{op}' := \sup\limits_{\|x\|_E \leq 1} |Tx|_\mathbb{K}.$$
These two norms are a priori different. This is due to the fact that an element does not necessarily have a multiple of norm 1 (Warning \ref{multiples}). However they are equivalent. One special case where they coincide is with $E = C_b(G, \mathbb{K})$ and $\mathbb{K}$ is discretely valued; more generally, whenever $\| \cdot \|_E$ takes the same values as $|\cdot|_\mathbb{K}$ does. \\

Following the standard conventions \cite{NFA, NOT}, we will work with $\| \cdot \|_{op}$. This is important since when dealing with bounded cohomology over $\mathbb{K}$ we look at dual normed $\mathbb{K}[G]$-modules. These are equipped with a linear \emph{isometric} action of $G$, so it is not enough to work with norms up to equivalence. However, when checking whether a given map is continuous or bounded, we can use either norm. In any case, $E^*$ is always a Banach space whenever $\mathbb{K}$ is complete. \\

If $E$ is a normed $\mathbb{K}[G]$-module, then the dual $E^*$ naturally becomes a normed $\mathbb{K}[G]$-module when endowed with the dual action, given by $(g \cdot \lambda)(x) = \lambda(g^{-1} \cdot x)$.

\begin{definition}
A \emph{dual normed $\mathbb{K}[G]$-module} is a normed $\mathbb{K}[G]$-module $E$ which is isometrically $G$-isomorphic to the dual of a normed $\mathbb{K}[G]$-module.
\end{definition}

The following fundamental results hold analogously to the real case (see \cite[Theorem 2.1.14]{NFA} for Theorem \ref{norm fin dim} and \cite[Theorem 2.1.17]{NFA} for the Open Mapping Theorem):

\begin{theorem}
\label{norm fin dim}

Let $\mathbb{K}$ be a complete valued field. Every finite-dimensional normed vector space over $\mathbb{K}$ is Banach, and all norms are equivalent.
\end{theorem}

\begin{theorem}[Open Mapping Theorem]
Let $E, F$ be $\mathbb{K}$-Banach spaces, $T : E \to F$ a bounded surjective linear map. Then $T$ is open. In particular, if $T$ is bijective, it is a homeomorphism.
\end{theorem}

The following fundamental result requires an additional hypothesis \cite[Corollary 4.1.2]{NFA}:

\begin{theorem}[Hahn--Banach Theorem]
Suppose that $\mathbb{K}$ is spherically complete. Let $E$ be a normed $\mathbb{K}$-vector space, $V$ a (not necessarily closed) linear subspace, and $\lambda : V \to \mathbb{K}$ a bounded linear map. Then $\lambda$ may be extended to a linear map on all of $E$ of the same norm.
\end{theorem}

The hypothesis of spherical completeness is strictly necessary. Indeed, if $\mathbb{K}$ is not spherically complete, then there exist non-zero Banach spaces whose dual is trivial \cite[Theorem 4.1.12]{NFA}. So this theorem even characterizes spherical completeness. \\

The following result is usually only proven in the real case. However the classical proof (see for example \cite[Theorem 1.1.28]{Zimmer}) only uses that closed balls in $\mathbb{R}$ or $\mathbb{C}$ are compact. This is true for all local fields, in particular for $\mathbb{Q}_p$:

\begin{theorem}[Banach--Alaoglu Theorem]
Suppose that $\mathbb{K}$ is a local field. Let $E$ be a normed $\mathbb{K}$-vector space, and $E^*$ its topological dual. Let $B$ be a closed ball in $E^*$. Then $B$ is compact with respect to the weak-$*$ topology on $E^*$.
\end{theorem}

Here the \emph{weak-$*$ topology} is the coarsest topology such that all evaluation maps $\{ E^* \to \mathbb{K} : \lambda \mapsto \lambda(x) \mid x \in E \}$ are continuous. The hypothesis that the topology on the field is locally compact is strictly necessary. Indeed, $\mathbb{K}^*$ with the weak-$*$ topology is linearly homeomorphic to $\mathbb{K}$ with its norm topology, so in order for $E = \mathbb{K}$ to satisfy this theorem the closed balls in $\mathbb{K}$ need to be compact.

\subsection{The $p$-adic Haar integral}
\label{preli Haar}

Here we introduce the fundamental tools from non-Archimedean harmonic analysis that will be used to characterize normed $p$-adic amenability in Section \ref{sequiv}. We will not need them for other fields, so we only focus on $\mathbb{Q}_p$. For a detailed account, see \cite[Chapter 2]{NHA}. \\

Denote by $C_{00}(G, \mathbb{Q}_p)$ the space of compactly supported continuous $\mathbb{Q}_p$-valued functions on $G$. Since the image of a compact set under a continuous function is bounded, $C_{00}(G, \mathbb{Q}_p) \subseteq C_b(G, \mathbb{Q}_p)$.

\begin{definition}
A \emph{$p$-adic Haar integral} on $G$ is a non-zero bounded linear map $C_{00}(G, \mathbb{Q}_p) \to \mathbb{Q}_p : f \mapsto \int f(x) dx$ which is left-invariant, meaning
$$\int (g \cdot f)(x) dx = \int f(x) dx$$
for all $f \in C_{00}(G, \mathbb{Q}_p)$ and all $g \in G$.
\end{definition}

Notice that the main difference with the real case is that the $p$-adic Haar integral is bounded only in terms of the supremum, and not of the support, of the function it takes as an argument (see \cite[2.2.6]{NHA}). This is clear in the discrete case (in which the Haar integral is just a sum) by the ultrametric inequality. \\

Let $U \leq G$ be a compact open subgroup, and let $\mathbbm{1}_U$ denote the characteristic function of $U$, which is an element of $C_{00}(G, \mathbb{Q}_p)$. Since there is no relation with Borel measures, it is not immediately clear from the definitions that $\int \mathbbm{1}_U(x) dx \neq 0$. However, this is the case \cite[2.2.6]{NHA}:

\begin{lemma}
\label{measure U}

Let $f \mapsto \int f(x) dx$ be a $p$-adic Haar integral on $G$. Then for every compact open subgroup $U \leq G$ we have $\int \mathbbm{1}_U(x) dx \neq 0$.
\end{lemma}

The following is a combination of \cite[Theorem 2.2.3]{NHA} and Lemma \ref{lpc pfin}:

\begin{theorem}[Schikhof]
\label{Haar}

There exists a $p$-adic Haar integral on $G$ if and only if $p \notin \mathbb{L}(G)$. Moreover, a $p$-adic Haar integral is unique up to scalar.
\end{theorem}

Now suppose that $p \notin \mathbb{L}(G)$, and let $f \mapsto \int f(x) dx$ denote a choice of a $p$-adic Haar integral. Fix $g \in G$. Then $C_{00}(G, \mathbb{Q}_p) \to \mathbb{Q}_p : f \mapsto \int f(x g^{-1}) dx$ is again a $p$-adic Haar integral, so by uniqueness there exists a scalar $\Delta(g) \in \mathbb{Q}_p$ such that $\int f(x g^{-1}) dx = \Delta(g) \int f(x) dx$ for all $f \in C_{00}(G, \mathbb{Q}_p)$. Notice that since the $p$-adic Haar integral is unique up to scalar, the definition of $\Delta$ is independent of the choice of $p$-adic Haar integral.

\begin{definition}
The map $\Delta : G \to \mathbb{Q}_p$ is called the \emph{$p$-adic modular function} on $G$.
\end{definition}

Analogously to the real case, $\Delta$ is a continuous homomorphism. In the non-Archimedean case, we can say even more \cite[Theorem 2.4.2]{NHA}:

\begin{theorem}
\label{mod fct}

The $p$-adic modular function $\Delta : G \to \mathbb{Q}_p^\times$ is a continuous homomorphism. Moreover, $\Delta$ takes values in $\mathbb{Z}_p^\times = \{ \alpha \in \mathbb{Q}_p \mid |\alpha|_p = 1 \}$.
\end{theorem}

The following formula will be useful when moving from left to right actions of $G$ on $C_{00}(G, \mathbb{Q}_p)$ \cite[Theorem 2.4.3]{NHA}:

\begin{lemma}
\label{int f'}

For all $f \in C_{00}(G, \mathbb{Q}_p)$, we have
$$\int f(x) dx = \int f(x^{-1}) \Delta(x^{-1}) dx.$$
\end{lemma}

\subsection{Divisible groups}
\label{ss_div}

We recall some basic facts about divisible groups, which will be useful in Subsection \ref{ss_surj} and especially in Section \ref{s_qm}. Given a discrete abelian group $G$ and a set $I$, we denote by $G^I$ the direct sum of copies of $G$ indexed by $I$, as a discrete group.

\begin{definition}
A discrete abelian group $G$ is \emph{divisible} if for all $n \geq 1$ and all $g \in G$ there exists $h \in G$ such that $nh = g$.
\end{definition}

The first natural examples are torsion-free:

\begin{example}
The additive groups of $\mathbb{Q}$, $\mathbb{R}$ and $\mathbb{C}$ are divisible. More generally, the additive group of a $\mathbb{Q}$-vector space is divisible.
\end{example}

The other fundamental example will play an essential role in Section 9, so let us recall its definition and main properties:

\begin{definition}
Let $p$ be a prime. For each $k \geq 0$, the map $\mathbb{Z}/p^k\mathbb{Z} \to \mathbb{Z}/p^{k+1}\mathbb{Z} : x \mod p^k \mapsto px \mod p^{k+1}$ is an injective homomorphism. The \emph{Pr\"ufer $p$-group}, denoted by $\mathbb{Z}(p^\infty)$, is the directed union of the groups $(\mathbb{Z}/p^k\mathbb{Z})_{k \geq 1}$ along these embeddings.
\end{definition}

See \cite[Chapter 10]{Rotman} for the many interesting properties of $\mathbb{Z}(p^\infty)$, starting from \cite[Theorem 10.13]{Rotman}. In particular:

\begin{example}
$\mathbb{Z}(p^\infty)$ is divisible.
\end{example}

This group admits various different characterizations. For instance it is characterized, up to the prime $p$, by the fact that its subgroups are totally ordered by inclusion. It is isomorphic to the additive groups $\mathbb{Z}[1/p]/\mathbb{Z} \cong \mathbb{Q}_p / \mathbb{Z}_p$, as well as to the $p$-primary component of the circle group. \\

Going back to divisible groups, their class is easily seen to be closed under taking quotients. These groups are precisely the injective objects in the category of abelian groups \cite[Theorem 10.23]{Rotman}:

\begin{theorem}[Baer]
\label{div inj}
Let $A \leq B$ be discrete abelian groups, and let $D$ be a divisible group. Then every homomorphism $A \to D$ extends to a homomorphism $B \to D$.
\end{theorem}

This implies a homological-algebraic result that will be useful when applying the Universal Coefficient Theorem in Subsection \ref{ss_surj}:

\begin{corollary}
\label{UCT div}

Let $A$ be a discrete abelian group, and $D$ a divisible group. Then $Ext^1_\mathbb{Z}(A, D) = 0$. Therefore, for every discrete group $G$, the Universal Coefficient Theorem in Cohomology \cite[Chapter 0]{Brown} gives an isomorphism $H^n(G, D) \cong Hom_\mathbb{Z}(H_n(G, \mathbb{Z}), D)$.
\end{corollary}

Theorem \ref{div inj} can be adapted to the continuous setting:

\begin{corollary}
\label{div ext}

Let $G$ be a topological group, $U$ an open subgroup, and $D$ a topological group whose underlying discrete group is divisible. Then every continuous homomorphism $U \to D$ extends to a continuous homomorphism $G \to D$.
\end{corollary}

\begin{proof}
Let $f : U \to D$ be continuous, and extend it to a homomorphism $f^\# : G \to D$ using Theorem \ref{div inj}. Since $f$ is continuous and $U$ is open in $G$, it follows that $f^\#$ is continuous at the identity element, and so being a homomorphism it is continuous everywhere.
\end{proof}

Divisible groups are classified \cite[Theorem 10.28]{Rotman}:

\begin{theorem}[Classification of divisible groups]
Let $G$ be a divisible group. Then there exist index sets $I_0$ and $I_p$ for all primes $p$ such that
$$G \cong \mathbb{Q}^{I_0} \bigoplus \left( \bigoplus\limits_p \mathbb{Z}(p^\infty)^{I_p} \right).$$
\end{theorem}

In particular a divisible $p$-group is a direct sum of copies of $\mathbb{Z}(p^\infty)$.

\pagebreak

\section{First examples and constructions}
\label{constructions}

In this section we characterize normed $\mathbb{K}$-amenability of compact groups, and prove that normed $\mathbb{K}$-amenability is preserved under some basic constructions.

Throughout the rest of this section $\mathbb{K}$ is a non-Archimedean valued field (not necessarily complete) with residue field $\mathfrak{r}$. Recall that if $\chi(\mathbb{K}) = 0$ and $\chi(\mathfrak{r}) = p > 0$, then up to equivalence we assume that $|\cdot|_\mathbb{K}$ coincides with the $p$-adic norm $| \cdot |_p$ when restricted to $\mathbb{Q}$ (see the end of Subsection \ref{preli field}). \\

We recall the definition of normed $\mathbb{K}$-amenability (Definition \ref{def nka}) for the reader's convenience:

\begin{definition}
Let $G$ be a t.d.l.c. group. A \emph{normed $\mathbb{K}$-mean} on $G$ is a bounded linear map $m : C_b(G, \mathbb{K}) \to \mathbb{K}$ such that $m(\mathbbm{1}_G) = 1$. We say that $G$ is \emph{normed $\mathbb{K}$-amenable} if it admits a left-invariant normed $\mathbb{K}$-mean. If this is the case, we denote by $\| G \|_\mathbb{K} \geq 1$ the infimum of the operator norms of such means, and say that $G$ is \emph{$\mathbb{K}$-amenable of norm} $\| G \|_\mathbb{K}$.

By convention, we will write $\| G \|_\mathbb{K} = \infty$ if $G$ is not normed $\mathbb{K}$-amenable. This way the function $\| \cdot \|_\mathbb{K}$ is well-defined on all t.d.l.c. groups, and takes finite values precisely on normed $\mathbb{K}$-amenable groups.
\end{definition}

Schikhof's notion of $\mathbb{K}$-amenability from \cite{Schik} requires the invariant mean to have norm $1$. It turns out that this coincides with normed $\mathbb{K}$-amenability of norm $1$, since the infimum $\| G \|_\mathbb{K}$ is always attained. While this is clear when $\mathbb{K}$ is discretely valued, we will have to wait until Corollary \ref{norm attain} for a proof of the general case. Therefore, for the moment, we will state Schikhof's Theorem \cite[Theorem 2.1, Theorem 3.6]{Schik} as a characterization of groups admitting a left-invariant $\mathbb{K}$-mean of norm $1$.

\begin{theorem}[Schikhof]
Let $\mathbb{K}$ be a non-Archimedean valued field, and let $\mathfrak{r}$ be its residue field. Then:
\begin{enumerate}
\item If $\mathbb{K}$ is not spherically complete, then $G$ admits a left-invariant $\mathbb{K}$-mean of norm $1$ if and only if $G$ is compact and $\chi(\mathfrak{r})$-free.
\item If $\mathbb{K}$ is spherically complete, then $G$ admits a left-invariant $\mathbb{K}$-mean of norm $1$ if and only if $G$ is locally elliptic and $\chi(\mathfrak{r})$-free.
\end{enumerate}
\end{theorem}

\begin{remark}
Recall that every group is $0$-free.
\end{remark}

Our proof of Theorem \ref{main1} (the characterization of normed $p$-adic amenability) will rely on Schikhof's Theorem for $\mathbb{K} = \mathbb{Q}_p$ only in the compact case. Our proof of the characterization of normed amenability over spherically complete fields (Theorem \ref{main sc}) will furthermore rely on Schikhof's Theorem for $\mathbb{K} = \mathbb{Q}$ or $\mathbb{F}_p$, both trivially valued. This is the case that is handled explicitly in Schikhof's proof \cite[Theorem 3.6]{Schik}. For the non-spherically complete case (Theorem \ref{main nsc}), the same proof as Schikhof's, relying on general results on harmonic analysis over non-spherically complete fields \cite{vR}, will work.

\subsection{Constructions}

Since $\| G \|_\mathbb{K} = \infty$ if $G$ is not normed $\mathbb{K}$-amenable, whenever we write $\| H \|_\mathbb{K} \leq \| G \|_\mathbb{K}$, we mean: if $G$ is normed $\mathbb{K}$-amenable, then $H$ is normed $\mathbb{K}$-amenable, and the minimal norm of a mean on $H$ is bounded by the minimal norm of a mean on $G$. \\

We start with a technical lemma that allows us to treat two constructions at once.

\begin{lemma}
\label{constrphi}

Let $G$ and $H$ be (t.d.l.c.) groups. Let $\varphi : G \to H$ be a continuous surjective map, and suppose that there exists a (not necessarily continuous) section $s : H \to G$ such that $h^{-1} \varphi(g) = \varphi(s(h)^{-1} g)$ for all $g \in G$ and all $h \in H$. Then $\|H\|_\mathbb{K} \leq \|G\|_\mathbb{K}$.
\end{lemma}

\begin{proof}
Let $m$ be a left-invariant $\mathbb{K}$-mean on $G$. Define $\tilde{m} : C_b(H, \mathbb{K}) \to \mathbb{K}$ by $\tilde{m}(f) := m(f \circ \varphi)$, which is well-defined because $f \circ \varphi \in C_b(G, \mathbb{K})$. Now $\tilde{m}$ is clearly linear, and
$$\|\tilde{m}\|_{op} = \inf \{ C \geq 0 \mid |\tilde{m}(f)|_\mathbb{K} = |m(f \circ \varphi)|_\mathbb{K} \leq C \|f\|_\infty \text{ for all } f \} \leq \|m\|_{op},$$
since $\|f \circ \varphi\|_\infty \leq \|f\|_\infty$. Moreover $\tilde{m}(\mathbbm{1}_H) = m(\mathbbm{1}_H \circ \varphi) = m(\mathbbm{1}_G) = 1$.

We are left to show that $\tilde{m}$ is $H$-left-invariant. Let $f \in C_b(H, \mathbb{K})$ and $h \in H$. Then for all $g \in G$:
$$((h \cdot f) \circ \varphi)(g) = f(h^{-1} \varphi(g)) = (f \circ \varphi)(s(h)^{-1}g) = (s(h) \cdot (f \circ \varphi))(g).$$
Thus:
$$\tilde{m}(h \cdot f) = m((h \cdot f) \circ \varphi) = m(s(h) \cdot (f \circ \varphi)) = m(f \circ \varphi) = \tilde{m}(f),$$
where we used that $m$ is $G$-left-invariant.
\end{proof}

We immediately obtain quotients:

\begin{proposition}
\label{quot}

Let $\varphi : G \to Q$ be a continuous surjective group homomorphism. Then $\| Q \|_\mathbb{K} \leq \|G\|_\mathbb{K}$. This applies to $Q = G/H$, where $H$ is a closed normal subgroup of $G$.
\end{proposition}

\begin{proof}
Since $\varphi$ is a homomorphism, if $s$ is \emph{any} section, for all $q \in Q$ and all $g \in G$ we have: $q^{-1} \varphi(g) = \varphi(s(q))^{-1} \varphi(g) = \varphi(s(q)^{-1} g)$. We conclude by Lemma \ref{constrphi}.
\end{proof}

In the real setting, the amenability of a closed subgroup is usually the hardest hereditary property to show, for almost all definitions of amenability one can choose as a starting point. We will prove that closed subgroups of normed $\mathbb{K}$-amenable groups are normed $\mathbb{K}$-amenable (Corollary \ref{clsub}); but for the moment we only treat open subgroups, which is all we need for the proof of Theorem \ref{main1}.

\begin{proposition}
\label{opsub}

Let $H$ be an open subgroup of $G$. Then $\|H\|_\mathbb{K} \leq \|G\|_\mathbb{K}$.
\end{proposition}

\begin{proof}
Fix a set $R$ of representatives of the coset space $H \, \backslash \, G$, so that every $g \in G$ may be uniquely written as $hr$ for $h \in H$ and $r \in R$. Define $\varphi : G \to H : g = hr \mapsto h$. This map is clearly surjective; moreover it is continuous: let $U$ be an open subset of $H$, which is also open in $G$ because $H$ is open, then $\varphi^{-1}(U) = UR$ is open in $G$. Define $s : H \to G$ to be the inclusion, which is a section of $\varphi$. Then for all $h \in H$ and all $g = h'r \in G$ we have: $h^{-1} \varphi(g) = h^{-1}h' = \varphi(h^{-1}h' r) = \varphi(s(h)^{-1} g)$. We conclude by Lemma \ref{constrphi}.
\end{proof}

We move to extensions. Also in this case we prove the proposition for open normal subgroups, which is enough for the proof of Theorem \ref{main1}, but we will later show that it is true for closed normal subgroups as well (Corollary \ref{extt}).

\begin{proposition}
\label{ext}

Let $H$ be an open normal subgroup of $G$. Then $\|G\|_\mathbb{K} \leq \|H\|_\mathbb{K} \cdot \|G/H\|_\mathbb{K}$.
\end{proposition}

\begin{proof}
Let $Q := G/H$ and let $m_H$, $m_Q$ be left-invariant $\mathbb{K}$-means on $H$ and $Q$, respectively. For $f \in C_b(G, \mathbb{K})$ and $g \in G$, define $f_g : H \to \mathbb{K} : h \mapsto f(gh)$. Then $\| f_g \|_\infty \leq \| f \|_\infty$, and $f_g$ is continuous because the left translation by $g$ and the function $f$ are both continuous. Next we define $f_Q : Q \to \mathbb{K}$ by $f_Q(q) := m_H(f_g)$, for some lift $g \in G$ of $q$. This definition does not depend on the choice of $g$. Indeed, if $g'$ is another lift of $q$, then there exists some $h \in H$ such that $g' = gh$, and $m_H(f_{g'}) = m_H(h^{-1} \cdot f_g) = m_H(f_g)$, by $H$-left-invariance of $m_H$. So $f_Q$ is well-defined. It is bounded by $\| m \|_{op} \cdot \| f \|_\infty$, and the continuity requirement is void since $Q$ is discrete. Moreover, for all $g, g' \in G$, if $\pi : G \to Q$ denotes the canonical projection:
$$(g \cdot f)_Q(\pi(g')) = m_H((g \cdot f)_{g'}) = m_H(f_{g^{-1}g'}) = f_Q(\pi(g)^{-1} \pi(g')) = (\pi(g) \cdot f_Q)(\pi(g')),$$
so $(g \cdot f)_Q = \pi(g) \cdot f_Q$.

This allows us to define $m : C_b(G, \mathbb{K}) \to \mathbb{K} : f \mapsto m_Q(f_Q)$. Then for all $g \in G$ and all $f \in C_b(G, \mathbb{K})$:
$$m(g \cdot f) = m_Q((g \cdot f)_Q) = m_Q(\pi(g) \cdot f_Q) = m_Q(f_Q) = m(f)$$
by $Q$-left-invariance of $m_Q$; so $m$ is $G$-left-invariant. It is easily seen to be a $\mathbb{K}$-mean of norm $\|m\|_{op} \leq \|m_H\|_{op} \cdot \|m_Q\|_{op}$.
\end{proof}

We will prove that actually $\| G \|_\mathbb{K} = \| H \|_\mathbb{K} \cdot \| G/H \|_\mathbb{K}$: this is Corollary \ref{extt comp} for the compact case, and Corollary \ref{extt} for the general case with $\mathbb{K} = \mathbb{Q}_p$. For other fields it follows from the $p$-adic case and Theorems \ref{main nsc} and \ref{main sc} (Remark \ref{examples general}). \\

Finally, we consider directed unions. Here closed subgroups do not pose a problem, however we need the Banach--Alaoglu Theorem for the proof. Therefore we will only prove this under the additional assumption that $\mathbb{K}$ is a local field; in particular, this will hold for $\mathbb{Q}_p$. We will deal differently with other spherically complete fields (see Corollary \ref{K-dirun}), and the result is false for non-sperically complete fields (see Theorem \ref{main nsc}).

\begin{proposition}
\label{dirun}

Suppose that $\mathbb{K}$ is a local field. Let $G$ be the directed union of the closed subgroups $(G_i)_{i \in I}$. Then $\|G\|_\mathbb{K} \leq \sup \{ \|G_i\|_\mathbb{K} \mid i \in I \}$.
\end{proposition}

\begin{proof}
Let $m_i$ be a left-invariant $\mathbb{K}$-mean on $G_i$, which is an element of $C_b(G_i, \mathbb{K})^*$. Let $M := \sup \{ \|m_i\|_{op} \mid i \in I \}$, and suppose that this value is finite (otherwise there is nothing to prove). Define:
$$K_i := \{ \varphi \in C_b(G, \mathbb{K})^* : \varphi(\mathbbm{1}_G) = 1; \, \|\varphi\|_{op} \leq M; \, \varphi \text{ is } G_i \text{-left-invariant} \}.$$
Every $K_i$ is a closed subset of the closed ball of radius $M$ in $C_b(G, \mathbb{K})^*$, which is weak-$*$ compact by the Banach--Alaoglu Theorem. Each $K_i$ is non-empty, since $f \mapsto m_i(f|_{G_i})$ is an element of $K_i$. If $G_i \subseteq G_j$, then $K_j \subseteq K_i$. Since the union is directed, finite intersections of $K_i$ are non-empty. Therefore, by compactness, there exists $m \in \bigcap\limits_{i \in I} K_i$, which is a normed $G$-left-invariant $\mathbb{K}$-mean. So $G$ is normed $\mathbb{K}$-amenable. Since $m \in K_i$ for all $i$, we have $\|m\|_{op} \leq M$.
\end{proof}

By Proposition \ref{opsub}, equality holds whenever all of the $G_i$ are open. We will see that equality holds in general (Corollary \ref{dirunn}).

\subsection{$p^\mathbb{N}$-freeness}

Next, we characterize normed $\mathbb{K}$-amenability of finite groups. This simple result will be fundamental in the characterization of normed $\mathbb{K}$-amenability.

\begin{theorem}
\label{finite}

Let $G$ be a finite group of order $n$. Then:
\begin{enumerate}
\item If $\chi(\mathbb{K}) = p > 0$, then $G$ is normed $\mathbb{K}$-amenable if and only if $p$ does not divide $n$, in which case $\| G \|_\mathbb{K} = 1$.
\item If $\chi(\mathbb{K}) = \chi(\mathfrak{r}) = 0$, then $\| G \|_\mathbb{K} = 1$.
\item If $\chi(\mathbb{K}) = 0$ and $\chi(\mathfrak{r}) = p > 0$, then $\| G \|_\mathbb{K} = |n|_\mathbb{K}^{-1} = |n|_p^{-1}$.
\end{enumerate}
If a left-invariant normed $\mathbb{K}$-mean exists, then it is unique. In particular, $\| G \|_\mathbb{K}$ is attained.
\end{theorem}

\begin{proof}
Let $G := \{g_1, \ldots, g_n\}$. Every linear map $m : C_b(G, \mathbb{K}) \to \mathbb{K}$ must be a linear combination of evaluation maps, so there exist coefficients $\alpha_i \in \mathbb{K}$ such that $m(f) = \sum \alpha_i f(g_i)$ for all $f \in C_b(G, \mathbb{K})$. For $m$ to be left-invariant, all coefficients must coincide. For $m(\mathbbm{1}_G) = 1$ to hold, they must all be equal to $1/n$. Therefore the only possible candidate for a normed left-invariant $\mathbb{K}$-mean is
$$m : C_b(G, \mathbb{K}) \to \mathbb{K} : f \mapsto \frac{1}{n} \sum\limits_{i = 1}^n f(g_i).$$

If $\chi(\mathbb{K}) \mid n$, then $n$ is not divisible in $\mathbb{K}$, and so $m$ is not well-defined. On the other hand, if $\chi(\mathbb{K}) \nmid n$, then $m$ is well-defined and it is a left-invariant normed $\mathbb{K}$-mean. \\

By the previous discussion this is the unique left-invariant normed $\mathbb{K}$-mean, so we are left to calculate $\|m\|_{op}$, assuming that $\chi(\mathbb{K}) \nmid n$. By the ultrametric inequality $\|m\|_{op} \leq |n|_\mathbb{K}^{-1}$. Evaluating at a Dirac mass yields $\|m\|_{op} \geq |n|_\mathbb{K}^{-1}$. We conclude that $\|G\|_\mathbb{K} = |n|_\mathbb{K}^{-1}$. This norm is equal to $|n|_p^{-1}$ in the third case, and to $1$ in the first two cases.
\end{proof}

This, together with the previous results, implies that normed $p$-adic amenable groups are $p^\mathbb{N}$-free, which is part of Theorem \ref{main1}. More generally:

\begin{proposition}
\label{amen p^N free}

Suppose that $\chi(\mathfrak{r}) = p > 0$, and let $G$ be a normed $\mathbb{K}$-amenable group.
\begin{enumerate}
\item If $\chi(\mathbb{K}) = p$, then $G$ is $p$-free.
\item If $\chi(\mathbb{K}) = 0$, then $G$ is $p^\mathbb{N}$-free. More precisely, if $\| G \|_\mathbb{K} < p^k$, then $G$ is $p^k$-free.
\end{enumerate}
\end{proposition}

\begin{proof}
Let $V \leq U \leq G$ be open subgroups of $G$, and suppose that $[U : V] < \infty$. Up to finite index, let us assume that $V$ is normal in $U$. By Propositions \ref{opsub} and \ref{quot} we have: $\| U / V \|_\mathbb{K} \leq \| U \|_\mathbb{K} \leq \| G \|_\mathbb{K} < \infty$. Moreover, $U/V$ is finite and discrete. By Theorem \ref{finite}, we split into two cases.

If $\chi(\mathbb{K}) = p$, then $p$ cannot divide $|U/V|$, so $G$ is $p$-free. If instead $\chi(\mathbb{K}) = 0$, let $k \geq 1$ be large enough so that $\| G \|_\mathbb{K} < p^k$. Then $\|U/V\|_\mathbb{K} = |U/V|_p^{-1} < p^k$, so $p^k$ cannot divide $|U/V|$, and $G$ is $p^k$-free.
\end{proof}

In the second case of Theorem \ref{finite} there are no restrictions, and so we do not obtain restrictions on the possible values of $[U : V]$.

\begin{example}
$\mathbb{Z}$ is not normed $p$-adic amenable for any $p$. Indeed, $[\mathbb{Z} : \mathbb{Z}/p^k \mathbb{Z}] = p^k$ can be arbitrarily large.
\end{example}

\begin{example}
\label{npa period}

Combining Proposition \ref{opsub} with the previous example, we deduce that a discrete normed $p$-adic amenable group cannot contain $\mathbb{Z}$ as a subgroup. In other words, a discrete normed $p$-adic amenable group is \emph{periodic}: each element has finite order.

It is possible to prove directly that the analogous result holds for general t.d.l.c. groups as well. With a bit more work, one can adapt the proof of Proposition \ref{opsub} to show the corresponding result for discrete subgroups. A t.d.l.c. group $G$ is \emph{periodic} if for all $g \in G$ the closure of the group generated by $g$ is compact. Such groups are the subject of a rich theory \cite{period}. By Weil's Lemma \cite[Proposition 7.43]{compact}, the closure of a cyclic group in a locally compact group is either compact or a discrete copy of $\mathbb{Z}$. Therefore a normed $p$-adic amenable group is periodic.

The much stronger fact that normed $p$-adic amenable groups are locally elliptic is part of Theorem \ref{main1}, hence our choice to only quickly mention this without a complete proof.
\end{example}

In compact groups, $p^k$-freeness passes to closed subgroups.

\begin{lemma}
\label{sub p free}

Let $G$ be compact and $p^k$-free, and let $H \leq G$ be a closed subgroup. Then $H$ is $p^k$-free.
\end{lemma}

\begin{proof}
Since every open subgroup of $H$ has finite index, it suffices to show that $p^k$ does not divide $[H : V]$ for any open subgroup $V$ of $H$. We can write $V = U \cap H$, for some compact open subgroup $U$ of $G$, by Lemma \ref{opinco}. Up to finite index, we may assume that $U$ is normal in $G$. Now $H/V = H/ (U \cap H) \cong HU / U \leq G/U$. Since $G$ is $p^k$-free, $p^k$ does not divide $[G : U]$, and so it does not divide $[H : V]$ either.
\end{proof}

\subsection{Compact groups}

We will now use Schikhof's Theorem, Theorem \ref{finite}, and Proposition \ref{ext} to characterize normed $\mathbb{K}$-amenability of compact groups. Recall that the compact t.d.l.c. groups are precisely the \emph{profinite} groups: inverse limits of finite groups. Let us start by stating the compact case of Schikhof's Theorem \cite[Theorem 1.4]{Schik}:

\begin{proposition}[Schikhof]
A profinite group $G$ admits a left-invariant $\mathbb{K}$-mean of norm $1$ if and only if it is $\chi(\mathfrak{r})$-free.
\end{proposition}

Therefore if $\chi(\mathfrak{r}) = 0$, then every profinite group admits a left-invariant $\mathbb{K}$-mean of norm $1$, and if $\chi(\mathfrak{r}) = p > 0$, then a profinite group admits a left-invariant $\mathbb{K}$-mean of norm $1$ if and only if it is $p$-free. \\

We can now move to normed $\mathbb{K}$-amenability. We will reduce to the case of norm $1$ by means of the following lemma:

\begin{lemma}
\label{comp U p}

Let $G$ be a profinite group, and let $k \geq 0$ be such that $G$ is $p^{k+1}$-free but not $p^k$-free. Then there exists an open normal subgroup $U \leq G$ such that $U$ is $p$-free and $|[G : U]|_p = p^{-k}$.
\end{lemma}

\begin{proof}
Since $G$ is not $p^k$-free, there exist open subgroups $V \leq U \leq G$ with $p^k \mid [U : V]$. Now $G$ is compact so $V$, being open, has finite index in $G$, and $p^k \mid [G : V]$ as well. Also $p^{k+1} \nmid [G : V]$, since $G$ is $p^{k+1}$-free. Up to passing to a finite-index subgroup, we may assume that $V$ is normal in $G$. Now if $W \leq W' \leq V$ are open subgroups of $V$, then they are also open subgroups of $G$. If $p \mid [W' : W]$, then $p^{k+1} = p \cdot p^k \mid [V : W] \cdot [G : V] = [G : W]$, which contradicts the assumption that $G$ is $p^{k+1}$-free. Therefore $V$ must be $p$-free.
\end{proof}

\begin{theorem}
\label{comp}
Let $G$ be a profinite group.
\begin{enumerate}
\item If $\chi(\mathbb{K}) = p > 0$, then $G$ is normed $\mathbb{K}$-amenable if and only if it is $p$-free, in which case it is $\mathbb{K}$-amenable of norm 1.
\item If $\chi(\mathbb{K}) = \chi(\mathfrak{r}) = 0$, then $G$ is $\mathbb{K}$-amenable of norm 1.
\item If $\chi(\mathbb{K}) = 0$ and $\chi(\mathfrak{r}) = p > 0$, then $G$ is normed $\mathbb{K}$-amenable if and only if it is $p^\mathbb{N}$-free, in which case $\| G \|_\mathbb{K} = \min \{ p^k \geq 1 \mid G \text{ is } p^{k+1}\text{-free}\}$.
\end{enumerate}
In all cases, $\| G \|_\mathbb{K}$ is attained.
\end{theorem}

\begin{proof}
$1.$ Suppose that $G$ is normed $\mathbb{K}$-amenable. By Proposition \ref{amen p^N free} we know that $G$ is $p$-free. So by Schikhof's Theorem it is $\mathbb{K}$-amenable of norm $1$, and $\| G \|_\mathbb{K}$ is attained. The converse follows again from Schikhof's Theorem. \\

$2.$ By Schikhof's Theorem, all profinite groups are $\mathbb{K}$-amenable of norm $1$, and $\| G \|_\mathbb{K}$ is attained. \\

$3.$ For the rest of the proof, let $p(G) := \min \{ p^k \mid G \text{ is } p^{k+1}\text{-free}\}$. In other words, $G$ is $p \cdot p(G)$-free but not $p(G)$-free. Proposition \ref{amen p^N free} states that $p(G) \leq \| G \|_\mathbb{K}$. 

Now let $G$ be $p^\mathbb{N}$-free, so let $k \geq 0$ be such that $p(G) = p^k$: $G$ is $p^{k+1}$-free but not $p^k$-free. By the previous lemma there exists an open normal subgroup $U \leq G$ that is $p$-free and such that $|[G : U]|_p = p^{-k}$. By Schikhof's Theorem, $U$ is $\mathbb{K}$-amenable of norm $1$; and $G/U$ is $\mathbb{K}$-amenable of norm $p^k$ by Theorem \ref{finite}. Thus $\| G \|_\mathbb{K} \leq \| U \|_\mathbb{K} \cdot \| G/U \|_\mathbb{K} = p^k = p(G)$ by Proposition \ref{ext}.

Finally, $\| U \|_\mathbb{K} = 1$ is attained by Schikhof's Theorem, $\| G/U \|_\mathbb{K}$ is attained by Theorem \ref{finite}, and the proof of Proposition \ref{ext} is constructive, so $\| G \|_\mathbb{K}$ is also attained.
\end{proof}

We end this section by proving the stronger version of Proposition \ref{ext} for compact groups:

\begin{corollary}
\label{extt comp}

Let $G$ be a profinite group, $H$ a closed normal subgroup of $G$. Then $\| G \|_\mathbb{K} = \| H \|_\mathbb{K} \cdot \| G/H \|_\mathbb{K}$.
\end{corollary}

\begin{remark}
The finite case is clear by Theorem \ref{finite}.
\end{remark}

\begin{proof}
By Theorem \ref{comp}, it is enough to consider $\mathbb{K} = \mathbb{Q}_p$. Indeed, the first case coincides with $p$-adic amenability of norm $1$, in the second case there is nothing to prove, and the third case coincides with normed $p$-adic amenability. Note that we are letting $H$ be an arbitrary closed normal subgroup, therefore we are not allowed to appeal to Proposition \ref{ext} for one inequality. \\

We start by showing that $\| G \|_p \leq \| H \|_p \cdot \| G/H \|_p$. By Theorem \ref{comp}, we have $\| G \|_p = \sup\{ \| G/U \|_p \mid U \leq G \text{ open normal} \}$: since $G$ is compact, every open subgroup has finite index and it contains a finite-index open normal subgroup. Therefore we need to show that $\| G/U \|_p \leq \| H \|_p \cdot \| G/H \|_p$ for every such $U$.

So let $U \leq G$ be open and normal. Then $G/U$ is a finite group with normal subgroup $HU/U$ and quotient $G/HU$, so by Theorem \ref{finite} we have $\| G/U \|_p = \| HU/U \|_p \cdot \| G/HU \|_p$. On the one hand $HU/U \cong H/(H \cap U)$ is a quotient of $H$, and so by Proposition \ref{quot} we have $\| HU/U \|_p \leq \| H \|_p$. On the other hand $G/HU$ is a quotient of $G/H$ and so again by Proposition \ref{quot} we have $\| G/HU \|_p \leq \| G/H \|_p$. All in all $\| G/U \|_p \leq \| H \|_p \cdot \| G/H \|_p$. This being true for all $U$, we conclude. \\

Next, we show the equality $\| G \|_p = \| H \|_p \cdot \| G/H \|_p$. If $\| H \|_p$ or $\| G/H \|_p$ are infinite, then so is $\| G \|_p$: for $G/H$ this follows from Proposition \ref{quot}, and for $H$ this follows from Theorem \ref{comp} and Lemma \ref{sub p free}. Therefore we may assume that both $H$ and $G/H$ are normed $p$-adic amenable, and then by the previous paragraph $G$ is also normed $p$-adic amenable. By Proposition \ref{amen p^N free} it is $p^\mathbb{N}$-free. By Lemma \ref{comp U p} and Theorem \ref{comp}, there exists an open normal subgroup $U \leq G$ that is $p$-free and such that $\| G \|_p = \| G/U \|_p$. As in the previous paragraph, we have $\| G/U \|_p = \| HU/U \|_p \cdot \| G/HU \|_p$.

On the one hand $HU/U \cong H/(H \cap U)$. Now $H \cap U$ is a closed subgroup of $U$, which is $p$-free, so it is itself $p$-free by Lemma \ref{sub p free}. Therefore $\| H \cap U \|_p = 1$, and by Propositions \ref{ext} and \ref{quot}:
$$\| H \|_p \leq \| H \cap U \|_p \cdot \| H/(H \cap U) \|_p = \| H/(H \cap U)\|_p \leq \| H \|_p.$$
It follows that $\| HU/U\|_p = \| H \|_p$. On the other hand $G/HU \cong (G/H)/(HU/H)$. Now $HU/H \cong U/(H \cap U)$, and so $\| HU/H \|_p \leq \| U \|_p = 1$ by Proposition \ref{quot}. Therefore, by Propositions \ref{ext} and \ref{quot}:
$$\| G/H \|_p \leq \| HU/H \|_p \cdot \| G/HU \|_p = \| G/HU \|_p \leq \|G/H \|_p.$$
It follows that $\| G/HU \|_p = \| G/H \|_p$. So
$$\| G \|_p = \| G/U \|_p = \| HU/U \|_p \cdot \| G/HU \|_p = \| H \|_p \cdot \| G/H \|_p.$$
This concludes the proof.
\end{proof}

This result will be generalized to arbitrary t.d.l.c. groups in Corollary \ref{extt}.

\pagebreak

\section{Characterization of normed $p$-adic amenability}
\label{sequiv}

In this section we focus on $\mathbb{Q}_p$ and prove a more general version of Theorem \ref{main1}. We will prove that a group $G$ is normed $p$-adic amenable if and only if it is $p^\mathbb{N}$-free and locally elliptic, which is the statement of Theorem \ref{main1}. On the way we will use a $p$-adic analogue of Reiter's property, which has much stronger consequences than the one over the reals \cite[2.6.B]{Pier}. Indeed, almost-invariance with an error smaller than $1$ implies true invariance. This will imply that a normed $p$-adic amenable group is locally elliptic, the most surprising part of Theorem \ref{main1}. \\

We can now state the main theorem of this section:

\begin{theorem}
\label{equiv}

Let $G$ be such that $p \notin \mathbb{L}(G)$. Consider the following statements:
\begin{enumerate}
\item $G$ is normed $p$-adic amenable.
\item $G$ satisfies \emph{Reiter's property}: for every compact subset $K \subset G$ and every $\varepsilon > 0$, there exists $f \in C_{00}(G, \mathbb{Q}_p)$ such that $\| f \|_\infty = 1$ and $\| g \cdot f - f \|_\infty < \varepsilon$ for all $g \in K$.
\item $G$ satisfies Reiter's property for $\varepsilon = 1$.
\item $G$ is locally elliptic.
\end{enumerate}
Then $1. \Rightarrow 2. \Rightarrow 3. \Rightarrow 4.$
\end{theorem}

\subsection{Proof of Theorem \ref{equiv}}

Let $G$ be such that $p \notin \mathbb{L}(G)$. By Theorem \ref{Haar}, there exists a $p$-adic Haar integral $f \mapsto \int f(x)dx$ on $G$. Denote by $\Delta$ the corresponding modular function. We fix these objects throughout this subsection. \\

The \underline{proof of $1. \Rightarrow 2.$} will follow closely Schikhof's approach \cite[3.4 - 3.6]{Schik}. The main difference is that in this case we are not able to move directly from $\mathbb{Q}_p$ to $\mathbb{F}_p$ (and Theorem \ref{main sc} shows that this is not possible), and so instead of invariance, for Reiter's property we only get almost-invariance (invariance will hold by Theorem \ref{Reiter}). We need to introduce the tools of \emph{convolution}, which will allow us to find the almost-invariant functions in Reiter's property. \\

For $g \in G$ let $\delta_g : C_b(G, \mathbb{Q}_p) \to \mathbb{Q}_p : f \mapsto f(g)$ be the evaluation map. Then $\delta_g \in C_b(G, \mathbb{Q}_p)^*$. Let $E(G)$ be the span of $\{ \delta_g \}_{g \in G}$, which is a (not necessarily closed) linear subspace of $C_b(G, \mathbb{Q}_p)^*$. As usual, $G$ acts on $C_{00}(G, \mathbb{Q}_p)$ by $(g \cdot f)(x) = f(g^{-1}x)$.

Let $f \in C_{00}(G, \mathbb{Q}_p)$ and let $\mu = \sum \lambda_i \delta_{x_i} \in E(G)$. We define the involutions:
$$f'(x) = f(x^{-1}) \Delta(x^{-1}); \,\,\,\,\,\,\,\,\,\, \mu' = \sum \lambda_i \delta_{x_i^{-1}}.$$

\begin{lemma}
\label{f f' mu mu'}

For all $f \in C_{00}(G, \mathbb{Q}_p)$, we have $\| f \|_\infty = \| f' \|_\infty$ and $\int f(x) dx = \int f'(x) dx$.

For all $\mu \in E(G)$, we have $\mu(\mathbbm{1}_G) = \mu'(\mathbbm{1}_G)$.
\end{lemma}

\begin{proof}
The first statement follows from Theorem \ref{mod fct}: the modular function satisfies $|\Delta(x)|_p = 1$ for all $x \in G$. The second statement is the content of Lemma \ref{int f'}. The third statement is a direct calculation: if $\mu = \sum \lambda_i \delta_{x_i}$, then $\mu(\mathbbm{1}_G) = \sum \lambda_i$.
\end{proof}

Next we define convolution between elements of $E(G)$ and compactly supported functions. For $\mu = \sum \lambda_i \delta_{x_i} \in E(G)$ and $f \in C_{00}(G, \mathbb{Q}_p)$, define:
$$(\mu * f)(x) = \sum \lambda_i f(x_i^{-1} x); \,\,\,\,\,\,\,\,\,\, (f * \mu)(x) = \sum \lambda_i f(x x_i^{-1})\Delta(x_i^{-1}).$$

\begin{lemma}
\label{conv f mu}

Convolution gives well-defined bilinear maps $E(G) \times C_{00}(G, \mathbb{Q}_p) \to C_{00}(G, \mathbb{Q}_p)$ and $C_{00}(G, \mathbb{Q}_p) \times E(G) \to C_{00}(G, \mathbb{Q}_p)$. Moreover, the following hold for all $\mu \in E(G)$ and all $f \in C_{00}(G, \mathbb{Q}_p)$:
$$(\mu * f)' = f' * \mu'; \,\,\,\,\,\,\,\,\,\, (g \cdot f) * \mu = g \cdot (f * \mu); \,\,\,\,\,\,\,\,\,\, \| \mu * f \|_\infty, \| f * \mu \|_\infty \leq |\mu(\mathbbm{1}_G)|_p \cdot \| f \|_\infty;$$
$$\int (\mu * f)(x) dx = \int (f * \mu)(x) dx = \mu(\mathbbm{1}_G) \int f(x)dx;$$
\end{lemma}

\begin{proof}
For the first equality, let $\mu = \sum \lambda_i \delta_{x_i}$. Then, using that $\Delta$ is a homomorphism with abelian target, for all $x \in G$:
$$(f' * \mu')(x) = \sum \lambda_i f'(x x_i) \Delta(x_i) = \sum \lambda_i f(x_i^{-1} x^{-1}) \Delta(x_i^{-1} x^{-1}) \Delta(x_i) = $$
$$ = \sum \lambda_i f(x_i^{-1} x^{-1}) \Delta(x^{-1}) = (\mu * f)(x^{-1}) \Delta(x^{-1}) = (\mu * f)'(x).$$
The rest is easy to show. Notice that by using the first equality together with Lemma \ref{f f' mu mu'}, it suffices to prove a statement for $\mu * f$ to obtain the analogous one for $f * \mu$.
\end{proof}

Next we define convolution between two elements of $E(G)$. For $\mu = \sum \lambda_i \delta_{x_i}$ and $\nu = \sum \tau_i \delta_{y_j}$ define
$$(\mu * \nu) = \sum \lambda_i \tau_j \delta_{x_i y_j} : C_{00}(G, \mathbb{Q}_p) \to \mathbb{Q}_p : f \mapsto \sum \lambda_i \tau_j f(x_i y_j).$$

The following is straightforward:

\begin{lemma}
\label{conv mu nu}

Convolution gives a well-defined bilinear map $E(G) \times E(G) \to E(G)$. Moreover, the following hold for all $\mu, \nu, \xi \in E(G)$ and all $f \in C_{00}(G, \mathbb{Q}_p)$:
$$(\mu * \nu)(\mathbbm{1}_G) = \mu(\mathbbm{1}_G) \nu(\mathbbm{1}_G); \,\,\,\,\,\,\,\, \mu * (\nu * \xi) = (\mu * \nu) * \xi;$$
$$(\mu * \nu) * f = \mu * (\nu * f); \,\,\,\,\,\,\,\, f * (\mu * \nu) = (f * \mu) * \nu.$$
\end{lemma}

We are now ready to prove the key technical lemma that will imply that normed $p$-adic amenable groups satisfy Reiter's property.

\begin{lemma}
\label{lemma reiter}

Let $G$ be normed $p$-adic amenable, and let $f_0 \in C_{00}(G, \mathbb{Q}_p)$ be such that $\int f_0(x) dx = 0$. Then $\inf\limits_{\mu(\mathbbm{1}_G) = 1} \| f_0 * \mu \|_\infty = 0$.
\end{lemma}

\begin{proof}
First, notice that if $\int f_0(x) = 0$ then $\int f'_0(x) = 0$ by Lemma \ref{f f' mu mu'}; and that $\| f_0 * \mu \|_\infty = \| (f_0 * \mu)' \|_\infty = \| \mu' * f'_0 \|_\infty$ by Lemma \ref{conv f mu}. So it suffices to prove that if $\int f_0(x) dx = 0$, then $\inf\limits_{\mu(\mathbbm{1}_G) = 1} \| \mu * f_0 \|_\infty = 0$. We prove the contrapositive: assume that $f_0$ is such that this infimum is greater than 0. Then there exists a finite constant $C > 0$ such that for all $\mu \in E(G)$ we have $|\mu(\mathbbm{1}_G)|_p \leq C \cdot \| \mu * f_0 \|_\infty$. We need to show that $\int f_0(x) dx \neq 0$. This will be done by constructing a functional $\psi \in C_{00}(G, \mathbb{Q}_p)^*$ such that $\psi(f_0) \neq 0$, and then showing that it is a multiple of the Haar integral using the uniqueness property (Theorem \ref{Haar}). \\

Consider the subspace $E(G) * f_0 \subset C_{00}(G, \mathbb{Q}_p)$. On it we define the linear form $\phi : \mu * f_0 \mapsto \mu(\mathbbm{1}_G)$. This is well-defined: if $\mu * f_0 = \nu * f_0$ then
$$|(\mu - \nu)(\mathbbm{1}_G)|_p \leq C \cdot \|(\mu - \nu) * f_0\|_\infty = 0$$
and so $\mu(\mathbbm{1}_G) = \nu(\mathbbm{1}_G)$. Moreover, for all $\mu \in E(G)$ we have
$$|\phi(\mu * f_0)|_p = |\mu(\mathbbm{1}_G)|_p \leq C \cdot  \| \mu * f_0 \|_\infty,$$
so $\| \phi \|_{op} \leq C < \infty$. By Hahn--Banach, $\phi$ extends to a linear map $\phi \in C_{00}(G, \mathbb{Q}_p)^*$ of the same norm. \\

For an arbitrary $f \in C_{00}(G, \mathbb{Q}_p)$ define $\lambda_f : G \to \mathbb{Q}_p : x \mapsto \phi(x^{-1} \cdot f)$. In particular
$$\lambda_{f_0}(x) = \phi(x^{-1} \cdot f_0) = \phi(\delta_{x^{-1}} * f_0) = \delta_{x^{-1}}(\mathbbm{1}_G) = 1,$$
so $\lambda_{f_0} = \mathbbm{1}_G$. Let us show that this is an element of $C_b(G, \mathbb{Q}_p)$. First, it is bounded:
$$|\lambda_f(x)|_p = |\phi(x^{-1} \cdot f)|_p \leq \| \phi \|_{op} \cdot \| x^{-1} \cdot f \|_\infty = \| \phi \|_{op} \cdot \| f \|_\infty.$$ Moreover it is continuous (actually, it is even \emph{left uniformly continuous}: this will be used in Proposition \ref{LUnpa}). Fix $\varepsilon > 0$. Since $f$ is compactly supported, it is left uniformly continuous: there exists a neighbourhood $U$ of the identity such that $\| f - u^{-1} \cdot f \|_\infty < \varepsilon / \| \phi \|_{op}$ for all $u \in U$. Then
$$|\lambda_f(x) - \lambda_f(ux)|_p = |\phi(x^{-1} \cdot f) - \phi((ux)^{-1} \cdot f)|_p \leq $$
$$\leq \| \phi \|_{op} \cdot \|x^{-1} \cdot f - x^{-1} u^{-1} \cdot f \|_\infty = \| \phi \|_{op} \cdot \|f - u^{-1} \cdot f\|_\infty < \varepsilon.$$
We conclude that $\lambda_f \in C_b(G, \mathbb{Q}_p)$ for all $f \in C_{00}(G, \mathbb{Q}_p)$. This allows to consider the map $\psi : C_{00}(G, \mathbb{Q}_p) \to \mathbb{Q}_p : f \mapsto m(\lambda_f)$, where $m$ is a normed left-invariant $p$-adic mean on $G$. \\

Note that
$$(g \cdot \lambda_f)(x) = \lambda_f(g^{-1} x) = \phi((x^{-1} g) \cdot f) = \phi(x^{-1} \cdot (g \cdot f)) = \lambda_{g \cdot f}(x).$$
Therefore
$$\psi(g \cdot f) = m(\lambda_{g \cdot f}) = m(g \cdot \lambda_f) = m(\lambda_f) = \psi(f),$$
so $\psi$ is left-invariant. Moreover, it is bounded:
$$|\psi(f)|_p = |m(\lambda_f)|_p \leq \|m\|_{op} \|\lambda_f\|_\infty \leq \| m \|_{op} \|\phi\|_{op} \|f\|_\infty.$$
By uniqueness of the Haar integral, there exists $c \in \mathbb{Q}_p$ such that $\psi(f) = c \int f(x)dx$ for all $f \in C_{00}(G, \mathbb{Q}_p)$. In particular
$$c \int f_0(x)dx = \psi(f_0) = m(\lambda_{f_0}) = m(\mathbbm{1}_G) = 1,$$
and so $\int f_0(x)dx \neq 0$ as we wanted.
\end{proof}

We are ready to prove that normed $p$-adic amenable groups satisfy Reiter's property.

\begin{proof}[Proof of $1. \Rightarrow 2$]
Let $G$ be normed $p$-adic amenable, and fix $\varepsilon > 0$ and $K \subset G$ compact. We start by finding $f \in C_{00}(G, \mathbb{Q}_p)$ such that $\int f(x) dx = 1$ and $\| g \cdot f - f \|_\infty < \varepsilon$ for all $g \in K$. Fix a compact open subgroup $U \leq G$; by Lemma \ref{measure U} we have $\int \mathbbm{1}_U(x) dx \neq 0$. Up to replacing $\mathbbm{1}_U$ by a scalar multiple of it in the rest of the proof, let us assume that $\int \mathbbm{1}_U(x) dx = 1$. \\

Since $K$ is compact, there exist $a_1, \ldots, a_n \in K$ such that $K \subseteq U \bigsqcup \left( \bigsqcup_i a_i U \right)$. Since these sets are disjoint, for all $i$ we have $\int (\mathbbm{1}_U - \mathbbm{1}_{a_i U})(x) dx = 0$, and so for every $\mu \in E(G)$ also
$$\int ((\mathbbm{1}_U - \mathbbm{1}_{a_i U}) * \mu)(x) dx = \mu(\mathbbm{1}_G) \int (\mathbbm{1}_U - \mathbbm{1}_{a_i U})(x) dx = 0$$
by Lemma \ref{conv f mu}. Using this fact, and Lemma \ref{lemma reiter} inductively, we can find $\mu_1, \ldots, \mu_n \in E(G)$ such that for all $1 \leq i \leq n$:
$$\| (\mathbbm{1}_U - \mathbbm{1}_{a_i U}) * \mu_1 * \cdots * \mu_i \|_\infty < \varepsilon$$
and $\mu_i(\mathbbm{1}_G) = 1$. \\

Define $f := \mathbbm{1}_U * \mu_1 * \cdots * \mu_n$. Then
$$\int f(x) dx = (\mu_1 * \cdots * \mu_n)(\mathbbm{1}_G) \int \mathbbm{1}_U(x) dx = \prod \mu_i(\mathbbm{1}_G) = 1,$$
by Lemmas \ref{conv f mu} and \ref{conv mu nu}. Finally, let $g \in K$. First assume that $g \notin U$, then there exists $1 \leq i \leq n$ and $u \in U$ such that $g = a_i u$. So:
$$\| g \cdot f - f \|_\infty = \|(g \cdot \mathbbm{1}_U - \mathbbm{1}_U) * \mu_1 * \cdots * \mu_n \|_\infty = $$
$$= \| (\mathbbm{1}_{a_i u U} - \mathbbm{1}_U) * \mu_1 * \cdots * \mu_n \|_\infty = \|((\mathbbm{1}_U - \mathbbm{1}_{a_i U}) * \mu_1 * \cdots * \mu_i) * \cdots * \mu_n \|_\infty \leq $$
$$\leq \| (\mathbbm{1}_U - \mathbbm{1}_{a_i U}) * \mu_1 * \cdots * \mu_i \|_\infty \cdot \left| \prod\limits_{i < j \leq n} \mu_j(\mathbbm{1}_G) \right|_p < \varepsilon,$$
where we used Lemmas \ref{conv f mu} and \ref{conv mu nu}. If instead $g \in U$, then $g \cdot \mathbbm{1}_U = \mathbbm{1}_U$, and so
$$(g \cdot f - f) = (g \cdot \mathbbm{1}_U - \mathbbm{1}_U) * \mu_1 * \cdots * \mu_n = 0.$$

We have thus found $f \in C_{00}(G, \mathbb{Q}_p)$ such that $\int f(x)dx = 1$ and $\| g \cdot f - f \|_\infty < \varepsilon$ for all $g \in K$. Now, let $C > 0$ be a constant bounding the norm of the Haar integral, so $1 = \left| \int f(x) dx \right|_p \leq C \| f \|_\infty$. We then get $\left\| \frac{f}{\| f \|_\infty}\right\|_\infty = 1$, and for all $g \in K$:
$$\left\| g \cdot \frac{f}{\|f\|_\infty} - \frac{f}{\|f\|_\infty} \right\| = \frac{1}{\| f \|_\infty} \| g \cdot f - f \|_\infty < \frac{\varepsilon}{\| f \|_\infty} \leq C \varepsilon.$$
Since $C$ is uniform and $\varepsilon$ is arbitrary, picking $f / \| f \|_\infty$ we conclude.
\end{proof}

The implication $2. \Rightarrow 3.$ is trivial. Finally, we \underline{prove $3. \Rightarrow 4.$}, namely that groups satisfying Reiter's property for $\varepsilon = 1$ are locally elliptic.

\begin{proof}[Proof of $3. \Rightarrow 4$]
Let $K \subset G$ be compact. We need to show that $K$ is contained in a compact subgroup of $G$. Let $f$ be as in Reiter's property for this $K$ and $\varepsilon = 1$; that is, $\| f \|_\infty = 1$ and for all $g \in K$ we have $\| g \cdot f - f \|_\infty < 1$. Let $H := \{ g \in G \mid \| g \cdot f - f \|_\infty < 1 \}$. By the ultrametric inequality this is a subgroup of $G$, and by construction $K \subseteq H$. So we are left to show that $H$ has compact closure. \\

Let $S := \{ s \in G \mid |f(s)|_p = 1 \}$. Since $f$ is compactly supported, this set is compact. Moreover $1 = \| f \|_\infty$, and $\| \cdot \|_\infty$ takes values in $p^\mathbb{Z} \cup \{ 0 \}$, which implies that $S$ is nonempty. If now $s \in S$ and $h \in H$, we have
$$1 > \| (h \cdot f) - f \|_\infty \geq |(h \cdot f - f)(s)|_p = |f(h^{-1} s) - f(s)|_p.$$
Since $|f(s)|_p = 1$, this is only possible if $|f(h^{-1} s)|_p = |f(s)|_p = 1$ as well, by Lemma \ref{ultrametric cor}. This shows that $HS \subseteq S$. So $H \subseteq SS^{-1}$ is contained in a compact set, and so it has compact closure.
\end{proof}

\subsection{Proof of Theorem \ref{main1}}

For the rest of this section, we denote by $p(G) := \min \{ p^k \geq 1 \mid G \text{ is } p^{k+1} \text{-free} \}$: the quantity from Theorem \ref{main1}. Let us recall the statement for the reader's convenience:

\begin{theorem}[Theorem \ref{main1}]
A group $G$ is normed $p$-adic amenable if and only if it is locally elliptic and $p^\mathbb{N}$-free. In this case, $\| G \|_p = p(G)$.
\end{theorem}

\begin{proof}
Let $G$ be normed $p$-adic amenable. By Proposition \ref{amen p^N free} it is $p^\mathbb{N}$-free and $p(G) \leq \| G \|_p$. In particular $p \notin \mathbb{L}(G)$. So we can use Theorem \ref{equiv} and conclude that $G$ is locally elliptic.

Let $G$ be locally elliptic and $p^\mathbb{N}$-free, and let $p^k := p(G)$. Since $G$ is locally elliptic, it is a directed union of compact open subgroups $U_i$. Each $U_i$ is also $p^k$-free, and so by Theorem \ref{comp} $U_i$ is $p$-adic amenable of norm at most $p^k$. It follows from Proposition \ref{dirun} that $G$ is normed $p$-adic amenable and $\| G \|_p \leq \sup \{ \| U_i \|_p \} \leq p^k = p(G)$.
\end{proof}

Theorem \ref{equiv} implies that one can replace the property of being locally elliptic in the theorem by Reiter's property. Here is a more general statement:

\begin{theorem}
\label{Reiter}

Let $G$ be a t.d.l.c. group (we do not assume that $p \notin \mathbb{L}(G)$). The following are equivalent:
\begin{enumerate}
\item $G$ is locally elliptic.
\item $G$ satisfies Reiter's property for $\varepsilon = 0$.
\item $G$ satisfies Reiter's property for some $\varepsilon \in [0, 1]$.
\item $G$ statisfies Reiter's property for $\varepsilon = 1$.
\end{enumerate}
\end{theorem}

\begin{proof}
$1. \Rightarrow 2.$ If $G$ is locally elliptic, given $K \subset G$ compact, there exists a compact open subgroup $K \subseteq U \leq G$ by Lemma \ref{cinco}. Then $\mathbbm{1}_U$ works. The implications $2. \Rightarrow 3. \Rightarrow 4.$ are clear, and $4. \Rightarrow 1.$ is proven as the implication $3. \Rightarrow 4.$ of Theorem \ref{equiv}. In fact, this implication did not need the hypothesis $p \notin \mathbb{L}(G)$, as it did not use the $p$-adic Haar integral or the modular function.
\end{proof}

Let us end this subsection by mentioning a third equivalent characterization, with only a sketch of proof. Consider the space $LUC_b(G, \mathbb{Q}_p)$ of left uniformly continuous bounded $\mathbb{Q}_p$-valued functions on $G$ (see Definition \ref{luc def}). This is $G$-left-invariant (Lemma \ref{luc_inv}), and so one may ask about normed means on this space: bounded linear maps $m : LUC_b(G, \mathbb{Q}_p) \to \mathbb{Q}_p$ which are left-invariant and satisfy $m(\mathbbm{1}_G) = 1$. Let us say temporarily that if such an $m$ exists, then $G$ is \emph{$LU$ normed $p$-adic amenable}.

\begin{proposition}
\label{LUnpa}

A group is normed $p$-adic amenable if and only if it is $LU$ normed $p$-adic amenable. Moreover, the minimal norm of a left-invariant mean on $LUC_b(G, \mathbb{Q}_p)$ coincides with $\| G \|_p$.
\end{proposition}

\begin{proof}[Sketch of proof]
Clearly a normed $p$-adic amenable group is $LU$ normed $p$-adic amenable. Now let $m : LUC_b(G, \mathbb{Q}_p) \to \mathbb{Q}_p$ be as in the definition of $LU$ normed $p$-adic amenability. One can adapt the proofs of Propositions \ref{opsub}, \ref{quot} and \ref{amen p^N free} to show that $G$ is $p^\mathbb{N}$-free; note that Theorem \ref{finite} is automatically adapted since on a discrete group all functions are left uniformly continuous. In the proof of Lemma \ref{lemma reiter} we apply the mean to a left uniformly continuous function (this was pointed out in the proof), and so it still applies to this case. That being the only point in the proof of Theorem \ref{equiv} involving the mean, we conclude that an $LU$ normed $p$-adic amenable group is locally elliptic and $p^\mathbb{N}$-free. So by Theorem \ref{main1} it is normed $p$-adic amenable.
\end{proof}

This notion has the advantage that Proposition \ref{ext} can now be adapted to closed subgroups with only a little more work, analogously to \cite[Proposition 13.4]{Pier}. But these stronger results can also be proven using Theorem \ref{main1}, as we will shortly do, so we will not be concerned with this notion any further.

\subsection{Corollaries}

The following is a direct consequence of the definition of a locally elliptic group:

\begin{corollary}
\label{compgen}

A compactly generated normed $p$-adic amenable group is compact.
\end{corollary}

We generalize Proposition \ref{opsub} to all closed subgroups:

\begin{corollary}
\label{clsub}

Let $H$ be a closed subgroup of $G$. Then $\|H\|_p \leq \|G\|_p$.
\end{corollary}

\begin{proof}
Since $G$ is locally elliptic and $H$ is closed, $H$ is easily seen to be locally elliptic. Therefore $H$ is the directed union of compact groups $H_i$. By Lemma \ref{cinco}, there exist compact open subgroups $U_i$ of $G$ such that $H_i \leq U_i$. So if $\| H_i \|_p \leq \| U_i \|_p$, then we are done by Propositions \ref{opsub} and \ref{dirun}. This follows by Theorem \ref{comp} and Lemma \ref{sub p free}.
\end{proof}

Next, we use the previous corollary to sharpen Proposition \ref{dirun}:

\begin{corollary}
\label{dirunn}

Let $G$ be the directed union of the closed subgroups $(G_i)_{i \in I}$. Then $\|G\|_p = \sup \{\|G_i\|_p \mid i \in I\}$.
\end{corollary}

\begin{proof}
One inequality is Proposition \ref{dirun}, the other follows from Corollary \ref{clsub}.
\end{proof}

Finally, we use the previous corollary, and the compact case (Corollary \ref{extt comp}), to show that we can deal with extensions with closed subgroups, and that the norm is not only submultiplicative but multiplicative:

\begin{corollary}
\label{extt}

Let $H$ be a closed normal subgroup of $G$. Then $\| G \|_p = \| H \|_p \cdot \| G/H \|_p$.
\end{corollary}

\begin{proof}
For the inequality $\| G \|_p \leq \| H \|_p \cdot \| G/H \|_p$, we assume that $H$ and $G/H$ are normed $p$-adic amenable. By Theorem \ref{main1}, they are locally elliptic, and so $G$ is also locally elliptic (see e.g. \cite[Proposition 1.6]{period}). We can therefore write $G$ as a directed union of compact open subgroups $(U_i)_{i \in I}$, and by Corollary \ref{dirunn} we have $\| G \|_p = \sup \{ \|U_i\|_p \mid i \in I \}$.

Now $U_i$ has a normal subgroup $H \cap U_i$ with quotient $U_i / (H \cap U_i)$. By Corollary \ref{extt comp}, we have $\| U_i \|_p = \| H \cap U_i \|_p \cdot \| U_i / (H \cap U_i) \|_p$. On the one hand $\| H \cap U_i \|_p \leq \| H \|_p$ by Proposition \ref{opsub}; and on the other hand $U_i / (H \cap U_i) \cong U_i H / H \leq G/H$, so $\| U_i / (H \cap U_i) \|_p \leq \| G/H \|_p$, again by Proposition \ref{opsub}. Thus $\| U_i \|_p \leq \| H \|_p \cdot \| G/H \|_p$. It follows that $\| G \|_p \leq \| H \|_p \cdot \| G/H \|_p$. \\

For the converse inequality, we assume that $G$ is normed $p$-adic amenable, and it follows from Proposition \ref{quot} and Corollary \ref{clsub} that also $G/H$ and $H$ are. Start by writing $G/H$ as the directed union of the closures of its finitely generated subgroups. Using Corollary \ref{dirunn} one of these groups has the same norm as $G/H$; since $G$ is locally elliptic we can lift the finitely many generators and obtain a compact group $Q \leq G$ such that $\| QH/H \|_p = \| G/H \|_p$. Again by Corollary \ref{dirunn}, let $K \leq H$ be a compact subgroup such that $\| K \|_p = \| H \|_p$. Let $G_0 := \overline{\langle K, Q \rangle} \leq G$. Since $G$ is locally elliptic, $G_0$ is compact. It has a normal subgroup $G_0 \cap H \geq K$ with quotient $G_0 / (G_0 \cap H) \cong G_0 H / H \geq QH/H$. By Corollaries \ref{extt comp} and \ref{clsub}:
$$\| G \|_p \geq \| G_0 \|_p = \| G_0 \cap H \|_p \cdot \| G_0 / (G_0 \cap H) \|_p \geq \| K \|_p \cdot \| QH/H \|_p = \| H \|_p \cdot \| G/H \|_p.$$
\end{proof}

We end on the following curious fact, which was already remark by Schikhof for the norm $1$ case \cite[End of Section 3]{Schik}:

\begin{corollary}
\label{npa amenable}

Normed $p$-adic amenable groups are amenable.
\end{corollary}

\begin{proof}
Locally elliptic groups are amenable.
\end{proof}

This reflects a pattern that properties of topological groups linked to non-Archimedean valued fields are stronger than their analogues over the reals. Another example of this phenomenon is the following. Suppose that $p \notin \mathbb{L}(G)$, so that $G$ admits a Haar integral and so a modular function $\Delta$ (see Subsection \ref{preli Haar}). Suppose that $G$ is $\mathbb{Q}_p$-unimodular, meaning that $\Delta \equiv 1$. Then $G$ is unimodular in the usual sense (over $\mathbb{R}$) \cite[Theorem 2.3.6]{NHA}.

\pagebreak

\section{Examples}
\label{examples}

In this section we go over some examples of normed $p$-adic amenable groups. We will start by discussing discrete groups, for which the $p^\mathbb{N}$-freeness condition reduces to a bound on the order of finite $p$-subgroups. These examples will include many simple groups. We will then move to profinite groups, which behave somewhat dually: here the $p^\mathbb{N}$-freeness condition reduces to a bound on the order of $p$-subgroups of finite discrete quotients. After that we will look at matrix groups over local fields, then at groups of tree automorphisms. These will include more discrete and profinite examples as well. In the last subsection, we will prove that $2$-adic amenability of small norm is an obstruction to simplicity.

\subsection{Discrete groups}

In the discrete setting, the notion of a locally elliptic group corresponds to that of a \emph{locally finite} group: a group in which every finitely generated subgroup is finite. We refer to \cite{loc_fin} for the general theory of locally finite groups. Also the notion of $p$-freeness can be reduced to an easier-to-check property.

\begin{proposition}
\label{disc}

Let $G$ be a discrete group. Then $G$ is normed $p$-adic amenable if and only if it is locally finite and the order of its finite $p$-subgroups is bounded, in which case $\| G \|_p$ is the maximal order of a finite $p$-subgroup.
\end{proposition}

\begin{proof}
Theorem \ref{main1} states that $G$ is normed $p$-adic amenable if and only if it is locally finite and $p^\mathbb{N}$-free. By Proposition \ref{free lpc}, the $p^\mathbb{N}$-freeness of $G$ is determined by ascending sequences of finite subgroups, since descending ones are necessarily eventually constant.
\end{proof}

Any group admits Sylow $p$-subgroups, that is, $p$-subgroups that are maximal with respect to inclusion. The previous characterization implies that the $p$-Sylow structure of a normed $p$-adic amenable group is especially nice:

\begin{corollary}
Let $G$ be a normed $p$-adic amenable discrete group. Then all of its $p$-Sylow subgroups are finite.
\end{corollary}

\begin{proof}
Let $P$ be a Sylow $p$-subgroup of the locally finite group $G$. Then $P$ is locally finite, and all of its subgroups are also $p$-groups. Suppose that $P$ is infinite. For any $n$, choose $n$ distinct elements in $P$ and consider the subgroup generated by them. This will be a finite $p$-group of order at least $n$. By the previous proposition, $G$ is not $p$-adic amenable.
\end{proof}

This allows to prove analogues of the Sylow Theorems:

\begin{proposition}[Sylow Theorems]
Let $G$ be a normed $p$-adic amenable discrete group. Then:
\begin{enumerate}
\item The $p$-Sylow subgroups of $G$ are precisely the finite $p$-subgroups of $G$ of maximal order.
\item All $p$-Sylow subgroups of $G$ are conjugate.
\item If $G$ has finitely many $p$-Sylow subgroups, then the number of such subgroups is congruent to $1 \mod p$.
\end{enumerate}
\end{proposition}

\begin{proof}
1. By the previous corollary, all $p$-Sylow subgroups are finite. Therefore a finite $p$-subgroup of $G$ of maximal order is a $p$-Sylow subgroup. Conversely, let $P$ be $p$-Sylow subgroup of $G$, and let $P'$ be a finite $p$-subgroup of $G$ of maximal order. Let $H$ be the subgroup generated by $P$ and $P'$, which is finite because $G$ is locally finite. By maximality of $P$, and maximality of the order of $P'$, it follows that $P$ and $P'$ are $p$-Sylow subgroups of $H$. So $|P| = |P'|$ by the first Sylow Theorem and thus $P$ is also of maximal order. \\

2. Let $P, P'$ be two $p$-Sylow subgroups of $G$, and let $H$ be the finite subgroup generated by them. Then $P, P'$ are $p$-Sylow subgroups of $H$, so by the Second Sylow Theorem $P$ and $P'$ are conjugate in $H$, hence in $G$. \\

3. Suppose that $P_1, \ldots, P_n$ are the $p$-Sylow subgroups of $G$, and let $H$ be the finite group generated by them. Then $P_1, \ldots, P_n$ is the collection of $p$-Sylow subgroups of $H$, so by the Third Sylow Theorem $n \equiv 1 \mod p$.
\end{proof}

The following lemma allows to produce new examples:

\begin{lemma}
\begin{enumerate}
\item Direct sums of $p$-adic amenable discrete groups of norm 1 are $p$-adic amenable of norm 1.
\item Let $G$ be a finite group which is $p$-adic amenable of norm 1, and let $I$ be an arbitrary index set. Then $\prod\limits_{I} G$ is $p$-adic amenable of norm 1, as a discrete group.
\end{enumerate}
\end{lemma}

\begin{proof}
In both cases, every finitely generated subgroup embeds in a finite direct product. By Proposition \ref{opsub} and Proposition \ref{ext}, every subgroup of a finite direct product of $p$-adic amenable groups of norm 1 is also $p$-adic amenable of norm 1. We conclude by Proposition \ref{dirun}.
\end{proof}

\begin{example}
\label{ex disc ap}

Let $(a_p)_p$ be a sequence indexed on the primes, with $a_p \in p^\mathbb{N} \cup \{ \infty \}$. Then there exists a discrete group $G$ such that $\|G\|_p = a_p$ for all primes $p$. An example is $\bigoplus_p G_p$, where $G_p = \left( \mathbb{Z}/p \mathbb{Z} \right)^{\mathbb{N}}$ if $a_p = \infty$, and $G_p = \mathbb{Z} / a_p \mathbb{Z}$ otherwise.

Notice that if a group is $p$-adic amenable of norm 1 for every $p$, then it is automatically trivial. Otherwise, given $g \neq 1$, it has finite order because $G$ is periodic, and then $G$ cannot be $p$-adic amenable of norm $1$ for the primes $p$ dividing the order of $g$.
\end{example}

We move on to some interesting locally finite or amenable periodic groups, and check if they are normed $p$-adic amenable.

\begin{example}
Let $S_\infty$ be group of bijections of $\mathbb{N}$ with finite support. Equivalently, $S_\infty$ is the directed union of the symmetric groups $S_n$, where $S_n$ embeds into $S_{n+1}$ as the stabilizer of the last letter. This is a locally finite group, but it is not $p^\mathbb{N}$-free for any $p$, so it is not normed $p$-adic amenable.

One can similarly define the group $A_\infty$ of even permutations, the directed union of the alternating groups $A_n$, which is also not $p^\mathbb{N}$-free for any $p$.
\end{example}

\begin{example}
Let $q$ be a prime power, and $\mathbb{F}_q$ the field with $q$ elements. For every $n \geq 1$ there is a natural embedding of $\operatorname{GL}_n(\mathbb{F}_q)$ into $\operatorname{GL}_{n+1}(\mathbb{F}_q)$ by acting on the first $n$ standard basic vectors. The \emph{stable linear group} $\operatorname{GL}^0(\mathbb{F}_q)$ is defined as the directed union of these groups along these embeddings. It is a locally finite group, but it is not normed $p$-adic amenable for any $p$, since $S_n$ embeds into $\operatorname{GL}_n(\mathbb{F}_q)$ as the subgroup of permutation matrices, and so $S_\infty$ embeds into $\operatorname{GL}^0(\mathbb{F}_q)$ as well. Therefore this group is not $p^\mathbb{N}$-free for any $p$.

One can similarly define the \emph{stable special linear group} $\operatorname{SL}^0(\mathbb{F}_q)$ which is not $p^\mathbb{N}$-free for any $p$ since it contains $A_\infty$.
\end{example}

\begin{example}
Grigorchuck's group is an amenable periodic $2$-group. However it is finitely generated infinite, so not locally finite, hence it cannot be normed $p$-adic amenable for any $p$.
\end{example}

\begin{example}
\label{ex pruf}

Consider the Pr\"{u}fer $p$-group $\mathbb{Z}(p^\infty)$ (see Subsection \ref{ss_div}). The subgroups of $\mathbb{Z}(p^\infty)$ are precisely $\mathbb{Z}/p \mathbb{Z} \subset \mathbb{Z}/p^2 \mathbb{Z} \subset \cdots$. Therefore $\mathbb{Z}(p^\infty)$ is not normed $p$-adic amenable, but it is $\ell$-adic amenable of norm $1$ for every prime $\ell \neq p$.
\end{example}

\begin{remark}
$\mathbb{Z}(p^\infty)$ also inherits a group topology from the circle group $\mathbb{R}/\mathbb{Z}$, but this one is not locally compact, so it does not make sense to talk about $p$-adic amenability in this case.
\end{remark}

\subsubsection{Matrix groups over locally finite fields}

A rich source of examples is given by matrix groups over locally finite fields.

\begin{definition}
A field $\mathbb{F}$ is \emph{locally finite} if every finitely generated subfield is finite.
\end{definition}

Equivalently, $\mathbb{F}$ is locally finite if it is a subfield of $\overline{\mathbb{F}_p}$, the algebraic closure of $\mathbb{F}_p$, for some prime $p$. From either description it follows that one can write $\mathbb{F}$ as a chain of finite fields of characteristic $p$; that is, $\mathbb{F} = \bigcup_{i \geq 1} \mathbb{F}_{p^{k_i}}$, where $\mathbb{F}_{p^{k_i}} \subset \mathbb{F}_{p^{k_{i+1}}}$, or equivalently $k_i \mid k_{i+1}$.

A matrix group over a locally finite field is locally finite: each element of $\operatorname{GL}_n(\mathbb{F})$ is contained in $\operatorname{GL}_n(\mathbb{F}_{p^{k_i}})$ for $i$ large enough, and so the same is true of a finite set of elements. \\

The largest possible locally finite fields do not yield normed $p$-adic amenable examples:

\begin{example}
For primes $p, \ell$, the groups $\operatorname{GL}_n(\overline{\mathbb{F}_p})$ and $\operatorname{SL}_n(\overline{\mathbb{F}_p})$ for $n \geq 2$ are not normed $\ell$-adic amenable. By Proposition \ref{opsub} it suffices to show this for $\operatorname{SL}_2(\overline{\mathbb{F}_p})$. For $\ell = p$, we notice that this group contains $\operatorname{SL}_2(\mathbb{F}_{p^k})$ for any $k \geq 1$, which in turn contains the subgroup of upper-triangular matrices with $1$'s on the diagonal. The latter is isomorphic to the additive group of $\mathbb{F}_{p^k}$, which has order $p^k$. So $\operatorname{SL}_2(\overline{\mathbb{F}_p})$ is not $p^\mathbb{N}$-free.

For $\ell \neq p$, since $\overline{\mathbb{F}_p}^\times$ embeds into $\operatorname{SL}_2(\overline{\mathbb{F}_p})$ via the map $x \mapsto diag(x, x^{-1})$, it suffices to show that $\overline{\mathbb{F}_p}^\times$ is not normed $\ell$-adic amenable. To see this, notice that $\overline{\mathbb{F}_p}^\times$ contains $\mathbb{F}_{p^k}^\times$ for any $k \geq 1$, which has order $p^k - 1$. Now for every integer $i \geq 1$, the integer $p$ is a unit modulo $\ell^i$. Therefore there exists some $k$ such that $p^k \equiv 1 \mod \ell^i$, and so $\ell^i \mid (p^k - 1)$. So $\overline{\mathbb{F}_p}^\times$ is not $\ell^\mathbb{N}$-free.

The same proof shows that $\operatorname{GL}_1(\overline{\mathbb{F}_p}) = \overline{\mathbb{F}_p}^\times$ is $p$-adic amenable of norm 1 (as $p$ never divides $p^k - 1$) but not normed $\ell$-adic amenable for any $\ell \neq p$.
\end{example}

This suggests that, in order to find normed $p$-adic amenable matrix groups over locally finite fields, we should look at smaller fields. In the next examples we shall construct the fields in such a way that the corresponding general linear group does not admit elements of certain prescribed orders.

\begin{lemma}
\label{gl2f p lemma}

Let $\ell, p$ be primes such that $\ell$ is a primitive root modulo $p$. Let $a > 1$ be an integer, and define $\mathbb{F}$ to be the directed union of the fields $\mathbb{F}_q$, where $(q = \ell^{a^k})_{k \geq 1}$. Suppose that $\operatorname{GL}_2(\mathbb{F})$ has an element of order $p$. Then $(p - 1)$ divides $2 a^k$ for some $k \geq 1$.
\end{lemma}

\begin{proof}
Every element of $\operatorname{GL}_2(\mathbb{F})$ is contained in $\operatorname{GL}_2(\mathbb{F}_q)$ for some $q$. This group has order $(q^2 - 1)(q^2 - q)$. Since by assumption $\operatorname{GL}_2(\mathbb{F})$ contains an element of order $p$, there exists some $q$ such that $p$ divides this order. Now $p \neq \ell$, so $p$ cannot divide $q$; therefore $p$ must divide $(q^2 - 1)(q - 1)$. Moreover $p$ is prime and $(q - 1)$ divides $(q^2 - 1)$, so $p$ divides $(q^2 - 1)$. In other words $\ell^{2 a^k} = q^2 \equiv 1 \mod p$. Since $\ell$ is a primitive root modulo $p$, the order of $\ell$ in $\mathbb{F}_p^\times$ is $(p - 1)$, and so $(p - 1)$ divides $2 a^k$.
\end{proof}

\begin{remark}
By Dirichlet's Theorem \cite[Chapter VI]{Serre}, for every prime $p$ there exist infinitely many primes $\ell$ such that $\ell$ is a primitive root modulo $p$.
\end{remark}

\begin{example}
Let $p$ be a prime which is not Fermat; that is, $(p - 1)$ is not a power of $2$. Let $\ell$ be a prime which is a primitive root modulo $p$, and let $a = 2$. Then the previous construction gives a field $\mathbb{F}$ such that $\operatorname{GL}_2(\mathbb{F})$ contains no element of order $p$. In particular $\| \operatorname{GL}_2(\mathbb{F}) \|_p = 1$.
\end{example}

\begin{example}
\label{GL p > 3}

Let $p > 3$. Let $\ell$ be a prime which is a primitive root modulo $p$ and $a$ be an integer which is coprime to $(p - 1)$. Then the previous construction gives a field $\mathbb{F}$ such that $\operatorname{GL}_2(\mathbb{F})$ contains no element of order $p$. In particular $\| \operatorname{GL}_2(\mathbb{F}) \|_p = 1$.
\end{example}

This example can be adapted to all $n \geq 2$, and the proof is exactly the same:

\begin{lemma}
Let $\ell, p$ be primes such that $\ell$ is a primitive root modulo $p$. Let $n \geq 2$ be an integer. Let $a > 1$, and define $\mathbb{F}$ to be the directed union of the fields $\mathbb{F}_q$, where $(q = \ell^{a^k})_{k \geq 1}$. Suppose that $\operatorname{GL}_n(\mathbb{F})$ has an element of order $p$. Then $(p - 1)$ divides $i a^k$ for some $2 \leq i \leq n$ and some $k \geq 1$.
\end{lemma}

\begin{example}
Let $p, \ell, a$ be as in the previous examples, so that $\operatorname{GL}_2(\mathbb{F})$ has no element of order $p$. Suppose moreover that $(p - 1)$ is relatively prime with $3, 4, \ldots, n$. Then $\operatorname{GL}_n(\mathbb{F})$ has no element of order $p$ either.
\end{example}

In the previous examples we saw that for every $p > 3$ there exist fields $\mathbb{F}$ such that $\| \operatorname{GL}_2(\mathbb{F})\|_p = 1$. The primes $2$ and $3$ have to be treated slightly differently:

\begin{remark}
Let $q$ be a prime power. Then the reader can check that the order $|\operatorname{GL}_2(\mathbb{F}_q)| = (q^2-1)(q^2 - q) = (q - 1)^2 q (q + 1)$ is always a multiple of $3$. Moreover, if $q \notin \{ 2, 4, 8 \}$, then this order is also a multiple of $16$. It follows that if $\mathbb{F}$ is an infinite, locally finite field, then $\| \operatorname{GL}_2(\mathbb{F}) \|_3 \geq 3$, and $\| \operatorname{GL}_2(\mathbb{F}) \|_2 \geq 16$.
\end{remark}

However, we can do the next best thing.

\begin{lemma}
Let $\ell$ be a prime such that $\ell \equiv 2 \mod 9$. Let $a > 1$ be an integer and define $\mathbb{F}$ to be the directed union of the fields $\mathbb{F}_q$, where $(q = \ell^{a^k})_{k \geq 1}$. Suppose that $\operatorname{GL}_2(\mathbb{F})$ has a subgroup of order $9$. Then $a$ is a multiple of $3$.
\end{lemma}

\begin{proof}
By the same argument as in Lemma \ref{gl2f p lemma}, the hypothesis implies that $9$ must divide $(q^2 - 1)$, and so $\ell^{2 a^k} \equiv 4^{a^k} \equiv 1 \mod 9$. We conclude by noticing that the order of $4$ in $(\mathbb{Z}/9\mathbb{Z})^\times$ is $3$.
\end{proof}

\begin{example}
\label{GL p 3}

Dirichlet's Theorem gives infinitely many possible choices of $\ell$. Choosing $a$ to be an integer not divisible by $3$, the previous construction gives a field $\mathbb{F}$ such that $\| \operatorname{GL}_2(\mathbb{F}) \|_3 = 3$.
\end{example}

\begin{lemma}
Let $\ell$ be a prime such that $\ell \equiv 3 \mod 8$. Let $a > 1$ be an integer and define $\mathbb{F}$ to be the directed union of the fields $\mathbb{F}_q$, where $(q = \ell^{a^k})_{k \geq 1}$. Suppose that $\operatorname{GL}_2(\mathbb{F})$ has a subgroup of order $32$. Then $a$ is even.
\end{lemma}

\begin{proof}
Suppose that $a$ is odd. By hypothesis $\ell \equiv 3 \mod 8$, so $q = \ell^{a^k} \equiv 3^{a^k} \mod 8$. Since each element has order $2$ in $(\mathbb{Z}/8\mathbb{Z})^\times$, we have $3^{a^k} \equiv 3 \mod 8$. Then $(q - 1) \equiv 2 \mod 8$; $(q + 1) \equiv 4 \mod 8$, and so $16$ is the highest power of $2$ dividing $(q - 1)^2 q (q + 1) =  |\operatorname{GL}_2(\mathbb{F}_q)|$.
\end{proof}

\begin{example}
\label{GL p 2}

Once again, Dirichlet's Theorem gives infinitely many possible choices of $\ell$. Choosing $a$ to be an odd integer, the previous construction gives a field $\mathbb{F}$ such that $\| \operatorname{GL}_2(\mathbb{F}) \|_2 = 16$.
\end{example}

Next, we look at more examples of normed $2$-adic amenable groups.

\begin{example}
\label{F quad cl}

Let $\mathbb{F}$ be a locally finite field which is not quadratically closed, meaning that not all elements have square roots in $\mathbb{F}$. We claim that $\operatorname{GL}_n(\mathbb{F})$ is normed $2$-adic amenable for all $n \geq 1$. This fact is stated in \cite{loc_fin} (see the discussion after Theorem 4.20), but we were not able to find a proof in the literature, so we provide one here. We will need the following lemma \cite{LTE}:

\begin{lemma}[Lifting the Exponent Lemma]
Let $u$ be an odd integer, $e$ an integer. If $\nu_2(u - 1) = 1$ and $e$ is even, then $\nu_2(u^e - 1) = \nu_2(u + 1) + \nu_2(e)$. In all other cases, $\nu_2(u^e - 1) = \nu_2(u - 1) + \nu_2(e)$.
\end{lemma}

\begin{proof}[Proof of claim]
First, notice that every finite field of characteristic $2$ is quadratically closed, since the squaring homomorphism on its multiplicative group has trivial kernel. Therefore, a locally finite field of characteristic $2$ is quadratically closed. Now let $\mathbb{F}$ be a locally finite field of odd characteristic $p$, and write it as $\mathbb{F} = \bigcup_{i \geq 1} \mathbb{F}_{p^{k_i}}$ where $k_i \mid k_{i+1}$. Every finite subgroup of $\operatorname{GL}_n(\mathbb{F})$ is contained in $\operatorname{GL}_n(\mathbb{F}_{p^{k_i}})$ for some $i \geq 1$. We want to bound the order of finite $2$-subgroups of $\operatorname{GL}_n(\mathbb{F})$, for which we only need to bound the order of Sylow $2$-subgroups of $\operatorname{GL}_n(\mathbb{F}_q)$, for $q = p^{k_i}$, under the additional assumption that $\mathbb{F}$ is not quadratically closed. This is equivalent to finding a bound to $\nu_2(|\operatorname{GL}_n(\mathbb{F}_q)|)$. Since $p$ is odd:
$$\nu_2(|\operatorname{GL}_n(\mathbb{F}_q)|) = \nu_2(q^n - 1) + \nu_2(q^n - q) + \cdots + \nu_2(q^n - q^{n-1}) = $$
$$ = \nu_2(q^n - 1) + \nu_2(q^{n-1} - 1) + \cdots + \nu_2(q - 1).$$
By the Lifting The Exponent Lemma, choosing $u = p$, which is odd, and $e = t k_i$ for $1 \leq t \leq n$, we get:
$$\nu_2(q^t - 1) = \nu_2(p^{t k_i} - 1) = \nu_2(p \pm 1) + \nu_2(t k_i) = \nu_2(p \pm 1) + \nu_2(t) + \nu_2(k_i).$$
The first two terms of the last sum are uniformly bounded in terms of $p$ and $n$. So we are left to bound $\nu_2(k_i)$.

Let $x \in \mathbb{F}$ be an element that does not admit a square root in $\mathbb{F}$, which exists by hypothesis. Let $i \geq 1$ be such that $x \in \mathbb{F}_{k_i}$. If $2$ divides $k_j/k_i$ for some $j > i$, then this would imply that $\mathbb{F}_{p^{k_i}} \subset \mathbb{F}_{p^{2k_i}} \subseteq \mathbb{F}_{p^{k_j}} \subset \mathbb{F}$. But $\mathbb{F}_{p^{2k_i}}$ contains all square roots of elements of $\mathbb{F}_{p^{k_i}}$ being its only quadratic extension, so this contradicts the choice of $x$. We conclude that $k_j/k_i$ is odd for all $j > i$. Therefore $\nu_2(k_j) = \nu_2(k_i)$ for all $j > i$, and we are done.
\end{proof}
\end{example}

\begin{example}
The same proof shows that if $\mathbb{F}$ is infinite and quadratically closed, then $\operatorname{GL}_n(\mathbb{F})$ is \emph{not} normed $2$-adic amenable for any $n \geq 2$.

First, if $\mathbb{F}$ is of characteristic $2$, then $|\operatorname{GL}_n(\mathbb{F}_q)|$ is divisible by $q$, which is a power of $2$; and so as $q \to \infty$ we get arbitrarily large Sylow $2$-subgroups.

Secondly, if the characteristic $p$ of $\mathbb{F}$ is odd and $\mathbb{F}$ is quadratically closed, write again $\mathbb{F} = \bigcup_{i \geq 1} \mathbb{F}_{p^{k_i}}$. Each element of $\mathbb{F}_{p^{k_i}}$ needs to have a square root in $\mathbb{F}$, but a finite field of odd characteristic cannot be quadratically closed (the squaring homomorphism has a non-trivial kernel), so $\mathbb{F}$ needs to contain $\mathbb{F}_{p^{2k_i}}$. Therefore, for $j > i$ large enough, we have $\mathbb{F}_{p^{k_i}} \subset \mathbb{F}_{p^{2k_i}} \subseteq \mathbb{F}_{p^{k_j}}$, and so $2$ divides $k_j/k_i$. This implies that $\nu_2(k_i) \to \infty$, and so the same proof as in the previous example shows that $\nu_2(|\operatorname{GL}_n(\mathbb{F}_{p^{k_i}})|) \to \infty$ as well.
\end{example}

\subsubsection{Simple examples}
\label{sss_simple}

We are able to give many examples of simple groups because of the following (\cite[Corollary 8.14, Theorem 8.23]{Rotman} - locally finite fields of characteristic $2$ are perfect):

\begin{theorem}[Jordan--Moore--Dickson]
\label{PSL}

If $\mathbb{F}$ is an infinite locally finite field, then $\operatorname{PSL}_n(\mathbb{F})$ is simple for all $n \geq 2$.
\end{theorem}

This gives examples of infinite simple groups which are $p$-adic amenable of norm $1$, for all $p > 3$:

\begin{example}
\label{PSL p > 3}

Let $p > 3$. By Example \ref{GL p > 3} there exist fields $\mathbb{F}$ such that $\| \operatorname{GL}_2(\mathbb{F}) \|_p = 1$. Then the simple group $\operatorname{PSL}_2(\mathbb{F})$ is $p$-adic amenable of norm $1$.
\end{example}

For the prime $p = 3$, we can do norm $3$:

\begin{example}
\label{PSL p 3}

By Example \ref{GL p 3} there are fields $\mathbb{F}$ such that $\| \operatorname{GL}_2(\mathbb{F}) \|_3 = 3$. Then the simple group $\operatorname{PSL}_2(\mathbb{F})$ is $3$-adic amenable of norm $3$: the reader can review the construction to check that the factor of $3$ remains after restricting to determinant $1$ and factoring out the center.
\end{example}

For the prime $3$ we can even get to norm $1$, but the construction is more involved. The reason why it is harder is that almost all finite simple groups have order a multiple of $3$. In fact, it follows from the Classification of Finite Simple Groups that the only family of finite simple groups whose order is not divisible by $3$ is the one of \emph{Suzuki groups} \cite[4.10]{Wilson}. Let us summarize the relevant facts about this family in the following theorem:

\begin{theorem}[Suzuki]
For any odd integer $k \geq 1$ and $q = 2^k$, there exists a subgroup of $\operatorname{Sp}_4(\mathbb{F}_q)$, denoted $\operatorname{Sz}(q)$, such that:
\begin{enumerate}
\item $\operatorname{Sz}(q)$ is simple.
\item $|\operatorname{Sz}(q)| = (q^2 + 1)q^2(q - 1)$.
\item $\operatorname{Sz}(q) \leq \operatorname{Sz}(q^r)$ for any odd integer $r \geq 1$.
\end{enumerate}
\end{theorem}

It is easy to check that $3$ cannot divide the order of $\operatorname{Sz}(q)$.

\begin{example}
\label{Suzuki}

Let $\mathbb{F}$ be the directed union of the fields $\mathbb{F}_q$, where $q = 2^{a^k}$, $k \geq 1$ and $a$ is odd. Define $\operatorname{Sz}(\mathbb{F})$ to be the directed union of the finite groups $\operatorname{Sz}(\mathbb{F}_q)$, which are totally ordered by inclusion. Then $\operatorname{Sz}(\mathbb{F})$ is simple, being a directed union of finite simple groups. Since each $\operatorname{Sz}(q)$ has no element of order $3$, it follows that $\| \operatorname{Sz}(\mathbb{F}) \|_3 = 1$.
\end{example}

As for $p = 2$, we can do norm $4$:

\begin{example}
\label{PSL p 2}

By Example \ref{GL p 2} there are fields $\mathbb{F}$ such that $\| \operatorname{GL}_2(\mathbb{F}) \|_2 = 16$. More precisely, they are directed unions of $\operatorname{GL}_2(\mathbb{F}_q)$, where $q \equiv 3 \mod 8$ for all $q$. Then the largest power of $2$ dividing $|\operatorname{SL}_2(\mathbb{F}_q)| = q(q^2 - 1)$ is $8$. Finally, dividing by the order of the center, we get $\| \operatorname{PSL}_2(\mathbb{F}) \|_2 = 4$.
\end{example}

We will see in Proposition \ref{2simple} that there exist no examples of norm $2$. Example \ref{F quad cl} gives many more examples of normed $2$-adic amenable groups of the form $\operatorname{GL}_n(\mathbb{F})$, and it follows that $\operatorname{PSL}_n(\mathbb{F})$ is normed $2$-adic amenable as well.

\subsection{Profinite groups}

As in the discrete case, we can give an easier-to-check characterization:

\begin{proposition}
\label{prof}

Let $(G_i)_{i \in I}$ be an inverse system of finite groups, and let $G$ be the inverse limit of the $G_i$. Then $G$ is normed $p$-adic amenable if and only if $\| G_i \|_p$ is uniformly bounded, and $\| G \|_p = \max \{ \| G_i \|_p \mid i \in I \}$.
\end{proposition}

\begin{proof}
Theorem \ref{main1} (or Theorem \ref{comp}) states that $G$ is normed $p$-adic amenable if and only if it is $p^\mathbb{N}$-free. By Proposition \ref{free lpc}, the $p^\mathbb{N}$-freeness of $G$ is determined by descending sequences of compact open subgroups, since ascending ones are necessarily eventually constant. Up to finite index subgroups, it is determined by descending sequences of compact open \emph{normal} subgroups, and so by the order of finite discrete quotients. These are precisely the $(G_i)_{i \in I}$ in the directed system defining $G$.
\end{proof}

So in a way discrete and profinite normed $p$-adic amenable groups behave dually to each other: to check $p^\mathbb{N}$-freness, we only need to look at finite subgroups in the first case, and at finite discrete quotients in the second.

\begin{example}
\label{ex Qq}

Let $p \neq \ell$ be prime numbers. Then $\mathbb{Z}_p$ is not normed $p$-adic amenable. It follows from Proposition \ref{opsub} that $\mathbb{Q}_p$ is not $p$-adic amenable either, since it contains $\mathbb{Z}_p$ as an open subgroup.

On the other hand, $\mathbb{Z}_\ell$ is $p$-adic amenable of norm 1. Also $\mathbb{Q}_\ell$ is $p$-adic amenable of norm $1$, being the directed union of $\ell^{-k} \mathbb{Z}_\ell$, for $k \geq 1$. Another way to see this is that $\mathbb{Q}_\ell / \mathbb{Z}_\ell \cong \mathbb{Z}(\ell^\infty)$ with the discrete topology, which we already saw is $p$-adic amenable of norm 1. Since $\mathbb{Z}_\ell \leq \mathbb{Q}_\ell$ is open, $\mathbb{Q}_\ell$ is $p$-adic amenable of norm 1 by Proposition \ref{ext}.
\end{example}

\begin{example}
\label{ex loc field}

We can generalize the previous example to additive groups of local fields. Let $\mathbb{K}$ be a local field, $\mathfrak{o}$ its ring of integers, which is a compact open subgroup of $\mathbb{K}$. The residue field is finite of characteristic $p > 0$ and order $q$. There exists a uniformizer $\pi \in \mathfrak{o}$, so that $\mathbb{K} = \bigcup_{k \geq 1} \pi^k \mathfrak{o}$.

The additive group $\mathfrak{o}$ is a pro-$p$ group, so it is not $p$-adic amenable, while it is $\ell$-adic amenable of norm $1$ for all $\ell \neq p$. The additive group of $\mathbb{K}$ has the same property, being a directed union of $\pi^{-k} \mathfrak{o}$, each of which is topologically isomorphic to $\mathfrak{o}$.
\end{example}

Proposition \ref{prof} allows us to look also at infinite direct products of profinite groups. Note that an infinite direct product of non-compact spaces is never locally compact, so it makes sense to restrict to profinite groups.

\begin{proposition}
Let $(G_i)_{i \in I}$ be profinite groups that are $p$-adic amenable of norm 1. Then $\prod\limits_{i \in I} G_i$ is $p$-adic amenable of norm 1.
\end{proposition}

\begin{proof}
Since the $G_i$ are profinite, so is their direct product $G$. If $U \leq G$ is an open subgroup, then
$$U \geq \prod\limits_{j \notin J} G_j \times \prod_{j \in J} \{ 1 \},$$
for some finite set $J \subset I$. Therefore $G/U$ is also a quotient of $\prod\limits_{j \in J} G_j$, which is $p$-adic amenable of norm 1 by Corollary \ref{extt}, so its order cannot be divisible by $p$. So $G$ has no finite discrete quotient of order a multiple of $p$, which implies that it is $p$-adic amenable of norm 1 by Proposition \ref{prof}.
\end{proof}

\begin{corollary}
Let $(G_i)_{i \in I}$ be a family of normed $p$-adic amenable groups, all but finitely many of which are profinite and of norm 1. Then $\prod\limits_{i \in I} G_i$ is $p$-adic amenable of norm $\prod\limits_{i \in I} \|G_i\|_p$.
\end{corollary}

The following example is an analogue of what we established in the discrete case (Example \ref{ex disc ap}):

\begin{example}
Let $(a_p)_p$ be a sequence indexed on the primes, with $a_p \in p^\mathbb{N} \cup \{ \infty \}$. Then there exists a profinite group $G$ such that $\| G \|_p = a_p$ for all primes $p$. An example is $\prod_p G_p$, where $G_p = \mathbb{Z}_p$ if $a_p = \infty$, and $G_p = \mathbb{Z} / a_p \mathbb{Z}$ otherwise.

Notice that if a profinite group is $p$-adic amenable of norm 1 for every $p$, it is automatically trivial. Indeed, a profinite group $G$ is either finite discrete (and we already looked at that case), or admits a non-trivial finite discrete quotient $F$. Then $G$ cannot be $p$-adic amenable for any prime $p$ dividing $|F|$.
\end{example}

An interesting source of examples comes from Galois theory.

\begin{example}
Let $L/K$ be an infinite Galois extension. Then $Gal(L/K)$ is the inverse limit of the finite groups $Gal(M/K)$ for all intermediate finite Galois extension $M/K$. Since $M/K$ is Galois, there exists a primitive element $x_M \in M$ and we have $|Gal(M/K)| = [M:K] = deg(x_M)$. Therefore the Galois group $Gal(L/K)$ is normed $p$-adic amenable if and only if there exists some $k \geq 1$ such that $p^k$ does not divide the degree of any element of $L$.

An extremal example of this is the maximal $p$-extension $K(p)$ of $K$. This is the composite of all Galois extensions of $K$ of degree a power of $p$. Then $G_K(p) := Gal(K(p)/K)$ is a pro-$p$ group, and any finite extension $L/K(p)$ has degree coprime to $p$. So on the one hand we obtain that $Gal(K^{sep}/K(p))$ is $p$-adic amenable of norm $1$, and on the other that $G_K(p)$ is normed $p$-adic amenable if and only if it is finite, and $\ell$-adic amenable of norm 1 for all primes $\ell \neq p$. In fact, $G_K(p)$ is infinite in whenever $K$ is a local field or a number field. The rich theory of Galois $p$-extensions is the subject of the book \cite{Koch}.
\end{example}

\begin{example}
Thanks to the explicit determination of the absolute Galois group of a non-Archimedean local field, we can say much more about normed $p$-adic amenability in this case. Namely, we can completely determine the primes $\ell$, and the corresponding norms, for which the following groups are normed $\ell$-adic amenable: the absolute Galois group, the inertia subgroup, the ramification subgroup, and all intermediate quotients. See \cite[VII.5]{Neuk} for the relevant definitions and the proofs of the facts used in this example.

Let $K$ be a non-Archimedean local field. Its residue field $\mathfrak{r}$ is finite, let $p$ denote its characteristic. Let $K^{sep}|K$ be the separable closure of $K$, and consider the following subextensions: $\tilde{K}|K$, the maximal \emph{unramified} extension, and $K_{tr}|K$, the maximal \emph{tamely ramified} extension. Denote by $G_K := Gal(K^{sep}|K)$ the \emph{absolute Galois group}, which contains the normal subgroups $V_K := Gal(K^{sep}|K_{tr})$, the \emph{ramification group}, and $T_K := Gal(K^{sep}|\tilde{K})$, the \emph{inertia subgroup}; we have $V_K \leq T_K$. 

First, $G_K/T_K \cong Gal(\tilde{K}|K)$ is canonically isomorphic to the absolute Galois group of the finite residue field $\mathfrak{r}$, which in turn is isomorphic to the profinite completion of $\mathbb{Z}$; that is, $\hat{\mathbb{Z}} \cong \prod \mathbb{Z}_{\ell}$. In particular this group is not normed $\ell$-adic amenable for any $\ell$. Secondly, $V_K$ is a free pro-$p$ group, in particular it is $\ell$-adic amenable of norm 1 if $\ell \neq p$, and not normed $p$-adic amenable. It remains to determine the amenability status of the quotient $T_K/V_K \cong Gal(K_{tr}|\tilde{K})$. This is isomorphic to $\hat{\mathbb{Z}}^{(p')} \cong \prod_{\ell \neq p} \mathbb{Z}_{\ell}$. Therefore it is $p$-adic amenable of norm 1, but not normed $\ell$-adic amenable for any $\ell \neq p$.

All in all, we have shown that $G_K, G_K/T_K$ and $T_K$ are not normed $\ell$-adic amenable for any $\ell$; that $V_K$ is $\ell$-adic amenable of norm 1 for $\ell \neq p$ and not normed $p$-adic amenable; and that $T_K/V_K$ is $p$-adic amenable of norm 1 and not normed $\ell$-adic amenable for $\ell \neq p$.
\end{example}

Further interesting examples of profinite groups, namely integral matrices over local fields, and the vertex stabilizer in $Aut(\mathcal{T}_d)$, will be treated in the next two subsections.

\subsection{Matrix groups over non-Archimedean local fields}

Let $\mathbb{K}$ be a non-Archimedean local field, $\mathfrak{o}$ its ring of integers, $\mathfrak{m}$ the maximal ideal, and $\mathfrak{r}$ the residue field. The norm $| \cdot |_\mathbb{K}$ takes values in $r^\mathbb{Z}$, for some $0 < r < 1$, and there exists a uniformizer $\pi \in \mathfrak{o}$, which satisfies $|\pi|_\mathbb{L} = r$. The residue field is finite, of characteristic $p > 0$, let $q$ be its order, so $\mathfrak{r} \cong \mathbb{F}_q$. Note that since the topology on $\mathbb{K}$ is t.d.l.c., this guarantees that matrix groups over $\mathbb{K}$ are t.d.l.c. as well. \\

We start by looking at the multiplicative group of $\mathbb{K}$. The additive one was treated together with profinite groups, in Example \ref{ex loc field}. We start with the following general result:

\begin{lemma}
\label{det npa}

Let $G \leq \operatorname{GL}_n(\mathbb{K})$ be a closed subgroup which does not surject continuously onto $\mathbb{Z}$. Then $\det(G) \subseteq \mathfrak{o}^\times$.
\end{lemma}

\begin{remark}
A locally elliptic group cannot surject continuously onto $\mathbb{Z}$.
\end{remark}

\begin{proof}
Consider the continuous homomorphisms:
$$G \xrightarrow{\det} \mathbb{K}^\times \xrightarrow{|\cdot|_\mathbb{K}} r^\mathbb{Z} \xrightarrow{\log_r} \mathbb{Z}.$$
By hypothesis, the image of this homomorphism is trivial. In other words, $\det(G) \subseteq \{ x \in \mathbb{K}^\times \mid |x|_\mathbb{K} = 1 \} = \mathfrak{o}^\times$.
\end{proof}

\begin{example}
The multiplicative group $\mathbb{K}^\times$ is not normed $\ell$-adic amenable for any prime $\ell$. Indeed, this coincides with $\operatorname{GL}_1(\mathbb{K})$, and the determinant takes values of arbitrary norm.
\end{example}

Using this example, we can easily establish the following:

\begin{example}
$\operatorname{GL}_n(\mathbb{K})$ and $\operatorname{SL}_n(\mathbb{K})$ for $n \geq 2$ are not normed $\ell$-adic amenable for any prime $\ell$. By the previous example, it suffices to identify a closed subgroup isomorphic to $\mathbb{K}^\times$. This can be achieved by the map $x \mapsto diag(x, x^{-1}, 1, \ldots, 1)$ for $n \geq 2$.
\end{example}

The situation with integral matrices is more interesting. Fix $n \geq 1$. We start by identifying explicitly a basis of compact open subgroups of $\operatorname{GL}_n(\mathfrak{o})$. A basis of compact open neighbourhoods of 1 in the space $\mathfrak{o}$ is $\{ 1 + \pi^k \mathfrak{o} \mid k \geq 1 \}$. It follows that a basis of compact open neighbourhoods of $I_n$ in the space $\operatorname{M}_{n \times n}(\mathfrak{o})$ is $\{ E_k := I_n + \pi^k \operatorname{M}_{n \times n}(\mathfrak{o}) \mid k \geq 1 \}$. Given any matrix $M \in \operatorname{M}_{n \times n}(\mathfrak{o})$, the expression $(I_n + \pi^k M + \pi^{2k} M + \cdots)$ converges absolutely, and is the inverse to $(I_n - \pi^k M)$, so every matrix of $E_k$ is invertible and its inverse is in $E_k$. Moreover, $(I_n + \pi^k M)(I_n + \pi^k N) = I_n + \pi^k(M + N) + \pi^{2k}MN$. So the $E_k$ are subgroups. We conclude that $\{ E_k \mid k \geq 1 \}$ is a basis of compact open subgroups of $\operatorname{GL}_n(\mathfrak{o})$. \\

Next we show that $E_1$ is a pro-$p$ group. Since $E_1$ is compact and $\{ E_k \mid k \geq 2 \}$ is a neighbourhood basis of the identity, it follows that $E_1$ is the inverse limit of $E_1 / E_k$. Therefore we only need to show that $E_1 / E_k$ is a $p$-group, and by induction this is equivalent to $E_k / E_{k+1}$ being a $p$-group. To see this, consider the map $f : E_k \to \operatorname{M}_{n \times n}(\mathfrak{r}) : I_n + \pi^k M \mapsto M \mod \mathfrak{m}$. By the calculation in the previous paragraph, $f$ is a homomorphism, where $\operatorname{M}_{n \times n}(\mathfrak{r})$ is seen as an additive group. It is clearly surjective with kernel $E_{k + 1}$. Therefore $E_k / E_{k + 1} \cong M_{n \times n}(\mathfrak{r}) \cong \mathbb{F}_q^{n^2}$. \\

Finally, the quotient $\operatorname{GL}_n(\mathfrak{o}) / E_1$ is isomorphic to $\operatorname{GL}_n(\mathfrak{r})$. Indeed, the map $f : \operatorname{GL}_n(\mathfrak{o}) \to \operatorname{GL}_n(\mathfrak{r}) : A \mapsto A \mod \mathfrak{m}$ is a continuous surjective homomorphism, and the kernel is precisely $E_1$.

\begin{example}
\label{ex:glno}

$\operatorname{GL}_n(\mathfrak{o})$ for $n \geq 1$ is not normed $p$-adic amenable. For all primes $\ell \neq p$, it is normed $\ell$-adic amenable of norm $|(q^n - 1) (q^{n-1} - 1) \cdots (q^2 - 1) (q - 1)|_{\ell}^{-1}$. In particular, it is $\ell$-adic amenable of norm 1 for all large enough $\ell$.

First, $\operatorname{GL}_n(\mathfrak{o})$ contains the infinite pro-$p$ group $E_1$ as an open subgroup, so it cannot be normed $p$-adic amenable by Proposition \ref{opsub}. Now fix $\ell \neq p$. By Proposition \ref{prof} the pro-$p$ group $E_1$ is $\ell$-adic amenable of norm 1, and by Theorem \ref{finite} the finite group $\operatorname{GL}_n(\mathfrak{r}) \cong \operatorname{GL}_n(\mathbb{F}_q)$ is $\ell$-adic amenable of norm $|(q^n - 1)(q^n - q) \cdots (q^n - q^{n-1})|_{\ell}^{-1} = |(q^n - 1)(q^{n-1} - 1) \cdots (q-1)|_{\ell}^{-1}$. We conclude by Corollary \ref{extt} that $\operatorname{GL}_n(\mathfrak{o})$ is $\ell$-adic amenable of the same norm.
\end{example}

\begin{example}
$\operatorname{SL}_n(\mathfrak{o})$ for $n \geq 2$ is not normed $p$-adic amenable. For all $\ell \neq p$, it is normed $\ell$-adic amenable of norm $|(q^n - 1) (q^{n-1} - 1) \cdots (q^2 - 1)|_{\ell}^{-1}$. In particular, it is $\ell$-adic amenable of norm 1 for all $\ell$ large enough.

The only difference with the above is that now $f : \operatorname{SL}_n(\mathfrak{o}) \to \operatorname{SL}_n(\mathfrak{r}) : A \mapsto A \mod \mathfrak{m}$ has a smaller image. We have $|\operatorname{SL}_n(\mathbb{F}_q)| = |\operatorname{GL}_n(\mathbb{F}_q)| / |\mathbb{F}_q^\times|$.
\end{example}

We have seen non-examples in the non-locally elliptic groups $\operatorname{GL}_n(\mathbb{K}), \operatorname{SL}_n(\mathbb{K})$, and examples in the compact groups $\operatorname{GL}_n(\mathfrak{o}), \operatorname{SL}_n(\mathfrak{o})$. Next, we look at an example which is locally elliptic but not compact.

\begin{example}
\label{upp un tr}

Let $U \leq \operatorname{SL}_n(\mathbb{K})$ denote the group of upper unitriangular matrices; that is, upper triangular matrices with $1$'s on the diagonal. Let $N_i$ denote the set of upper triangular matrices whose first $i$ superdiagonals are zero, so $U = I_n + N_1$. Clearly $N_{i+1} \subset N_i$ and $N_n = \{ 0_n \}$. One checks that $N_1^i \subset N_i$. It follows that $I_n + N_i$ is a closed subgroup of $U$ for all $i \geq 1$. Moreover $(I_n + N_i) / (I_n + N_{i+1}) \cong \mathbb{K}^{n-i}$. We conclude that $U$ is an iterated extension of the additive group of $\mathbb{K}$, which is not normed $p$-adic amenable, but it is $\ell$-adic amenable of norm $1$ for all primes $\ell \neq p$ by Example \ref{ex loc field}. By Corollary \ref{extt}, the same is true of $U$.
\end{example}

Next, we look at examples of discrete subgroups of $\operatorname{GL}_n(\mathbb{K})$. Recall that for such a group to be normed $\ell$-adic amenable for some $\ell$ it needs to be locally finite. In what follows, we will rely on the strong results about locally finite linear groups from \cite[Chapter 9, Chapter 12]{lin}. For starters, it is enough for such a group to be periodic \cite[Corollary 4.9]{lin}:

\begin{theorem}[Schur]
A periodic linear group is locally finite.
\end{theorem}

In the case of local fields of characteristic $0$ we can say even more. Recall that a non-Archimedean local field $\mathbb{K}$ has characteristic $0$ if and only if it is a finite extension of $\mathbb{Q}_p$. Then rephrasing \cite[Theorem 9.33 (i) and (iv)]{lin} we obtain:

\begin{theorem}[Wehrfritz]
\label{lin char 0}

If $\chi(\mathbb{K}) = 0$, then any periodic linear group over $\mathbb{K}$ is finite.
\end{theorem}

It follows that the only candidates for infinite examples of discrete subgroups of $\operatorname{GL}_n(\mathbb{K})$ arise when $\mathbb{K}$ has characteristic $p$; that is, when $\mathbb{K}$ is a finite extension of $\mathbb{F}_q((1/X))$ for some prime power $q$. In this case it is very natural to look at $\mathbb{F}_q[X]$, which is a discrete subring of $\mathbb{F}_q((1/X))$, so $\operatorname{GL}_n(\mathbb{F}_q[X])$ is a discrete subgroup of $\operatorname{GL}_n(\mathbb{K})$. Notice that $(\mathbb{F}_q[X])^\times = \mathbb{F}_q^\times$.

\begin{example}
$\operatorname{GL}_n(\mathbb{F}_q[X])$ and $\operatorname{SL}_n(\mathbb{F}_q[X])$ for $n \geq 2$ are not normed $\ell$-adic amenable for any prime $\ell$. For this it suffices to show that $\operatorname{SL}_2(\mathbb{F}_q[X])$ is not periodic: by Proposition \ref{disc} (or even Example \ref{npa period}) it follows that it is not normed $\ell$-adic amenable for any $\ell$, and by Proposition \ref{opsub} we conclude. This follows from the fact that this group contains a copy of $(\mathbb{Z}/p\mathbb{Z} * \mathbb{Z}/p\mathbb{Z})$ \cite[Theorem 4]{Nagao}.
\end{example}

\begin{example}
The group of upper triangular matrices in $\operatorname{GL}_n(\mathbb{F}_q[X])$ is normed $\ell$-adic amenable of norm $|(q - 1)|_\ell^{-n}$ for all primes $\ell \neq p$. Indeed, the diagonal entries of an upper triangular matrix in $\operatorname{GL}_n(\mathbb{F}_q[X])$ must be in $(\mathbb{F}_q[X])^\times = \mathbb{F}_q^\times$. Therefore the subgroup of upper unitriangular matrices has index $|\mathbb{F}_q^\times|^n = (q-1)^n$. This is a closed subgroup of the group of upper unitriangular matrices from Example \ref{upp un tr}, so it is $\ell$-adic amenable of norm $1$ for all $\ell \neq p$ by Corollary \ref{clsub}.
\end{example}

The previous example generalizes:

\begin{lemma}
\label{disc lin}

Let $\mathbb{K}$ be a local field of characteristic $p$ and let $G \leq \operatorname{GL}_n(\mathbb{K})$ be a discrete subgroup which is virtually solvable and does not surject onto $\mathbb{Z}$. Then $G$ is normed $\ell$-adic amenable for all $\ell \neq p$, and normed $p$-adic amenable if and only if it is finite.
\end{lemma}

\begin{remark}
A periodic group does not surject onto $\mathbb{Z}$.
\end{remark}

\begin{proof}
By a Theorem of Mal'cev \cite[Theorem 3.6]{lin}, $G$ is virtually triangularizable, so we may assume that $G$ is contained in the upper triangular subgroup of $\operatorname{GL}_n(\mathbb{K})$. Since $G$ is discrete, it is closed. Isolating the diagonal gives a homomorphism to $(\mathbb{K}^\times)^n$, and so $n$ homomorphisms to $\mathbb{K}^\times$. By the same argument as in Lemma \ref{det npa}, each of these has image in a closed subgroup of $\mathfrak{o}^\times$. Now by Example \ref{ex:glno} and Corollary \ref{clsub}, each closed subgroup of $\mathfrak{o}^\times$ is normed $\ell$-adic amenable for all $\ell \neq p$, and normed $p$-adic amenable if and only if it is finite. Thus, we may reduce to the case in which $G$ is a subgroup of the upper unitriangular group. Once again, by Example \ref{upp un tr} and Corollary \ref{clsub}, it follows that $G$ is $\ell$-adic amenable of norm $1$ for all $\ell \neq p$, and it is normed $p$-adic amenable if and only if it is finite.
\end{proof}

All examples we have seen are not normed $p$-adic amenable. This is not a coincidence:

\begin{lemma}
A closed subgroup of $\operatorname{GL}_n(\mathbb{K})$ is normed $p$-adic amenable if and only if it is finite.
\end{lemma}

\begin{proof}
Let $G \leq \operatorname{GL}_n(\mathbb{K})$ be closed and normed $p$-adic amenable. We have seen that $\{ E_k \mid k \geq 1 \}$ is a neighbourhood basis of the identity of $\operatorname{GL}_n(\mathbb{K})$, so $\{ G \cap E_k \mid k \geq 1 \}$ is a neighbourhood basis of the identity of $G$. By Proposition \ref{opsub}, the pro-$p$ group $G \cap E_1$ is normed $p$-adic amenable, so it must be finite. This implies that $G \cap E_k = \{ I_n \}$ for $k \geq 1$ large enough. Therefore $G$ is discrete.

If $\chi(\mathbb{K}) = 0$, it follows from Theorem \ref{lin char 0} that $G$ is finite and we are done. Instead suppose that $\chi(\mathbb{K}) = p$. Since $G$ is normed $p$-adic amenable, it has a finite Sylow $p$-subgroup, so by \cite[Corollary 9.7]{lin} it is virtually abelian. By Lemma \ref{disc lin}, it is virtually a locally finite $p$-group, so it must be finite.
\end{proof}

\subsection{Groups of tree automorphisms}
\label{ss_aut(T)}

Let $d \geq 3$ and consider the group $Aut(\mathcal{T}_d)$ of automorphisms of the $d$-valent tree. It admits as a basis of compact open subgroups the groups $Fix(T) := \{ g \in Aut(\mathcal{T}_d) \mid g(v) = v \text{ for all } v \in T \}$, where $T$ is a finite subtree of $\mathcal{T}_d$. Of particular relevance is the case in which $T$ is a single vertex.

\begin{example}
\label{fixv0}
Let $v_0 \in V(\mathcal{T}_d)$. Then the group $Fix(v_0)$ is isomorphic to the inverse limit of $(Aut(B_n(v_0)))_{n \geq 1}$. Now $|Aut(B_1(v_0))| = d!$, since $Aut(B_1(v_0)) \cong Sym(\partial B_1(v_0)) \cong S_d$. The prime divisors of the other orders can be calculated inductively. Indeed, $Aut(B_{n+1}(v_0))$ also acts on $B_n(v_0)$, since tree automorphisms preserve distances. The map $Aut(B_{n+1}(v_0)) \to Aut(B_n(v_0))$ is surjective. The automorphisms in the kernel are precisely those that permute the $(d-1)$ neighbours of vertices at distance $n$ from $v_0$. There are $(d-1)!^{|\partial B_n(v_0)|}$ such automorphisms. Thus $|Aut(B_{n+1}(v_0))| = |Aut(B_n(v_0))| \cdot (d-1)!^{|\partial B_n(v_0)|}$. Inductively, this proves that $Fix(v_0)$ is $p$-adic amenable of norm 1 if $p > d$, of norm $p$ if $p = d$, and not normed $p$-adic amenable if $p < d$.
\end{example}

\begin{example}
\label{fixT}
In an analogous way we can prove that given a finite subtree $T$ with at least two vertices, the compact open subgroup $Fix(T)$ is $p$-adic amenable of norm 1 if $p \geq d$ and not normed $p$-adic amenable if $p < d$.
\end{example}

On the other hand:

\begin{example}
$Aut(\mathcal{T}_d)$ is not normed $p$-adic amenable for any $p$ by Corollary \ref{compgen}, because it is non-compact and compactly generated. 
\end{example}

Although $Aut(\mathcal{T}_d)$ is not locally elliptic, it admits interesting non-compact and locally elliptic subgroups. Here is an example. \\

Let $r = (v_0, v_1, v_2, \ldots)$ be an infinite ray in $\mathcal{T}_d$. For all $i \geq 0$, define $S_i = \{ w \in V(\mathcal{T}_d) \mid d(v_i, w) = i = d(v_{i+1}, w) - 1 \}$. So $S_i$ is the set of elements of the sphere of radius $i$ around $v_i$, except those for which the path from $v_i$ to $w$ passes through $v_{i+1}$. Notice that $S_i \subset S_{i+1}$ for all $i$. The \emph{horosphere around $r$} is $\mathcal{H}(r) := \bigcup\limits_{i \geq 0} S_i$. This allows to give another description of the sets $S_i$, namely: $S_i = B_i(v_i) \cap \mathcal{H}(r)$. In Figure 1, part of the horosphere around $r$ is drawn. More precisely, the vertices drawn are exactly those in $S_3$, in a 3-valent tree.

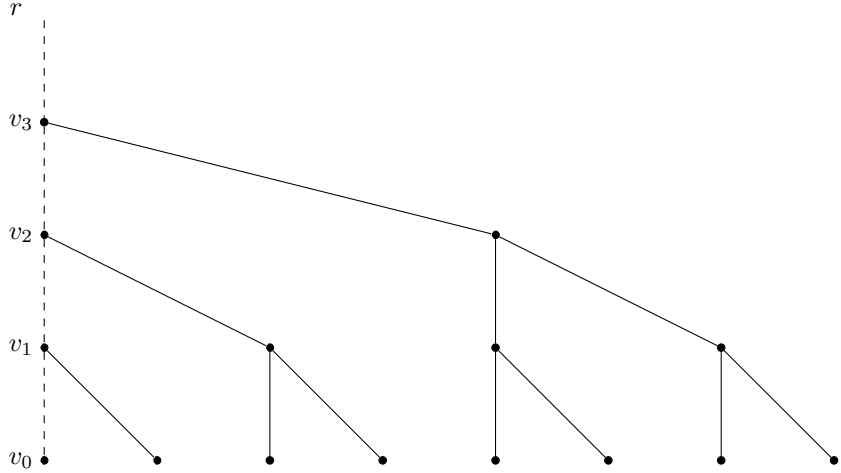
\begin{figure}
\centering
\captionsetup{justification=centering}

\begin{center}
\begin{tikzpicture}
[font=\footnotesize]
\node[label=left:{$r$}]{}
[every node/.style={circle,draw,inner sep=1,fill=black}]
child{ node[label=left:{$v_3$}]{} edge from parent[dashed]
	child[missing] child[missing] child[missing] child[missing]
    child { node[label=left:{$v_2$}]{}
    	child[missing] child[missing]
    	child { node[label=left:{$v_1$}]{}
    		child[missing]
    		child { node[label=left:{$v_0$}]{} }
    		child { node{} edge from parent[solid]}
    	}
    	child[missing]
    	child { node{} edge from parent[solid]
    		child[missing]
    		child { node{} }
    		child { node{} }
    	}
    }
    child[missing] child[missing] child[missing]
    child { node{} edge from parent[solid]
    	child[missing] child[missing]
    	child { node{}
    		child[missing]
    		child { node{} }
    		child { node{} }
    	}
    	child[missing]
    	child { node{}
    		child[missing]
    		child { node{} }
    		child { node{} }
    	}
    }
};
\end{tikzpicture}
\end{center}

\caption{The vertices in the bottom row form part of the horosphere around $r$. \\ In this example, $d = 3$.
\label{fig_horo}}

\end{figure}

\begin{example}
Let $G := Aut(\mathcal{T}_d)_{\mathcal{H}(r)}$ be the stabilizer of this horosphere. We remark two key facts. First, that any element of $G$ must fix some $v_i$; and secondly that any element of $G$ fixing $v_i$ must also fix $v_{i+1}$.

\begin{proof}
Let $\varphi \in G$, and let $w_0 := \varphi(v_0) \in \mathcal{H}(r)$. Let $i \geq 1$ be such that $w_0 \in S_i$. Since $w_0 \in S_i$ we have $d(w_0, v_i) = i$, so we can denote $(w_0, w_1, \ldots, w_i = v_i)$ the path from $w_0$ to $v_i$. We will prove by induction that $\varphi(v_j) = w_j$, and it will follow that $\varphi(v_i) = w_i = v_i$. \\

For $j = 0$, this is by definition.

Let us look at $j = 1$. Among the neighbours of $w_0$, $w_1$ is the only one which has $(d-1) \geq 2$ neighbours in $\mathcal{H}(r)$, while all others only have 1 (which is $w_0$). Now $v_1$ has $(d-1)$ neighbours in $\mathcal{H}(r)$, and so $\varphi(v_1)$ has $(d-1)$ neighbours in $\varphi(\mathcal{H}(r)) = \mathcal{H}(r)$. Moreover $\varphi(v_1)$ is a neighbour of $\varphi(v_0) = w_0$, and so $\varphi(v_1) = w_1$.

Now let $j \geq 2$, and suppose that $\varphi(v_{j-1}) = w_{j-1}$. Among the neighbours of $w_{j-1}$, $w_j$ is the only one at distance $j$ from $\mathcal{H}(r)$, while all others are at distance $j-2 \geq 0$. Moreover $\varphi(v_j)$ is a neighbour of $\varphi(v_{j-1}) = w_{j-1}$ satisfying $d(\varphi(v_j), \mathcal{H}(r)) = d(\varphi(v_j), \varphi(\mathcal{H}(r))) = d(v_j, \mathcal{H}(r)) = j$. Therefore $\varphi(v_j) = w_j$. \\

This proves that any element of $G$ must fix some $v_i$. Now suppose that $G$ fixes $v_i$. Then it must also preserve $S_i$, and so $w_0 = \varphi(v_0) \in S_i \subset S_{i+1}$. By the argument above, $\varphi$ fixes $v_{i+1}$ too.
\end{proof}

This allows us to write $G = \bigcup\limits_{i \geq 1} G_i$, where $G_i := G \cap Fix(v_i) \leq G_{i+1}$. This inclusion implies that if $e$ denotes the edge from $v_i$ to $v_{i+1}$, then $G_i$ is a subgroup of $Fix(e)$. Recall from Example \ref{fixT} that $Fix(e)$ is $p$-adic amenable of norm 1 for all $p \geq d$, so we deduce that $G$ has the same property by Proposition \ref{dirun}. On the other hand, any permutation of $S_i$ fixing the rest of $\mathcal{H}(r)$ is in $G_i$, which implies that $\|G_i\|_p \to \infty$ for $p < d$, by the same argument as in Example \ref{fixv0}; so $G$ is not normed $p$-adic amenable for $p < d$.
\end{example}

The last examples of this subection are groups acting on trees with an (almost) prescribed local action. These come from a construction of Burger and Mozes \cite{BM} which was later generalized by Le Boudec \cite{LeB}. We refer the reader to \cite[Section 3]{LeB} for details.
 
Let $\Omega$ be a set of cardinality $d \geq 3$. Given a vertex $v \in V(\mathcal{T}_d)$, denote by $E(v)$ the set of edges adjacent to $v$. Fix a proper coloring $c$ of the set of edges $E(\mathcal{T}_d)$; namely, a map $c : E(\mathcal{T}_d) \to \Omega$ such that for all $v \in V(\mathcal{T}_d)$ the restriction $c_v : E(v) \to \Omega$ is a bijection. Given an automorphism $g \in Aut(\mathcal{T}_d)$, the bijection $g_v : E(v) \to E(gv)$ induces a permutation $\sigma(g, v) := c_{gv} \circ g_v \circ c_v^{-1} \in Sym(\Omega)$. For a group $F \leq Sym(\Omega)$, the \emph{Burger-.Mozes universal group} $U(F)$ is the subgroup of $Aut(\mathcal{T}_d)$ whose local action is prescribed by $F$; namely $U(F) := \{ g \in Aut(\mathcal{T}_d) \mid \sigma(g, v) \in F \text{ for all } v \in V(\mathcal{T}_d) \}$. This is a closed subgroup of $Aut(\mathcal{T}_d)$, which is vertex-transitive, and discrete if and only if $F$ acts freely on $\Omega$. We start by looking at the vertex stabilizer $U(F)_{v_0} \leq Fix(v_0)$:

\begin{example}
Being a closed subgroup of $Fix(v_0)$, the group $U(F)_{v_0}$ is $p$-adic amenable of norm $1$ if $p > d$, and of norm at most $p$ if $p = d$, by Example \ref{fixv0}. We can calculate the norm for the primes $p \leq d$ as we did for $Fix(v_0)$. Indeed, $U(F)_{v_0}$ is the inverse limit of the finite groups $(Aut(B_n(v_0))^*)_{n \geq 1}$, where $Aut(B_n(v_0))^*$ is the image of the action of $U(F)_{v_0}$ on $B_n(v_0)$. First, by definition $Aut(B_1(v_0))^* \cong F$. Then $Aut(B_{n+1}(v_0))^*$ acts on $B_n(v_0)$ preserving the action of $U(F)_{v_0}$, and the corresponding homomorphism $Aut(B_{n+1}(v_0))^* \to Aut(B_n(v_0))^*$ is surjective. The automorphisms in the kernel are precisely those that permute the $(d - 1)$ neighbours of vertices at distance $n$ from $v_0$, according to the local action prescribed by $F$. Therefore this group is isomorphic to a product of $|\partial B_n(v_0)|$ copies of point stabilizers of $F$. A point stabilizer in $F$ has index at most $d$, with equality precisely when $F$ is transitive. Note that when $d = p$ is a prime, $F$ is transitive if and only if it has order divisible by $p$. Therefore in this case, $U(F)_{v_0}$ is $p$-adic amenable of norm $p$, if $F$ is transitive, and of norm $1$ otherwise. For the primes $p < d$ it depends on the possible orders of point stabilizers of $F$; but at any rate $U(F)_{v_0}$ is $p$-adic amenable of norm $1$ for any prime $p$ which does not divide the order of $F$, and of norm $p$ for any prime $p$ which divides the order of $F$ but not the order of any point stabilizer of $F$.
\end{example}

This construction can be relaxed to consider groups in which the local action is prescribed \emph{almost everywhere}. Define $G(F) := \{ g \in Aut(\mathcal{T}_d) \mid \sigma(g, v) \in F \text{ for all but finitely many } v \in V(\mathcal{T}_d) \}$. For $g \in G(F)$, the set of \emph{singularities} of $g$ is the finite set $S(g) := \{ v \in V(\mathcal{T}_d) \mid \sigma(g, v) \notin F \}$. Since $S(gh) \subseteq S(g) \cup S(h)$, it follows that $G(F)$ is a subgroup of $Aut(\mathcal{T}_d)$, which contains $U(F)$. It turns out \cite[Lemma 3.3]{LeB} that given $g \in G(F)$, the permutation $\sigma(g, v)$ preserves the orbits of $F$, even when $v \in S(g)$. We denote by $\hat{F}$ the subgroup of $Sym(\Omega)$ that preserves the orbits of $F$, so we have $U(F) \leq G(F) \leq U(\hat{F})$. There exists a group topology on $G(F)$ such that the inclusion $U(F) \to G(F)$ is continuous and open; however this is \emph{not} the subspace topology induced by $Aut(\mathcal{T}_d)$: indeed $G(F)$ is closed in $Aut(\mathcal{T}_d)$ if and only if $F = \hat{F}$ \cite[Corollary 3.6]{LeB}. We now consider the vertex stabilizer $G(F)_{v_0} \leq Fix(v_0)$:

\begin{example}
Let $K_n(v_0) := \{ g \in G(F)_{v_0} \mid S(g) \subseteq B_n(v_0) \}$ for all $n \geq 0$. We have $U(F)_{v_0} \leq K_n(v_0)$ for all $n$, and since $B_n(v_0)$ is finite, the index is also finite. Therefore $K_n(v_0)$ is a compact open subgroup of $G(F)_{v_0}$. By definition $G(F)_{v_0}$ is the directed union of the groups $K_n(v_0)$ as $n \geq 0$ grows. Again by \cite[Corollary 3.6]{LeB}, the inclusion $K_n(v_0) \subset K_{n+1}(v_0)$ is strict unless $F = \hat{F}$. Therefore unless $F = \hat{F}$, the group $G(F)_{v_0}$ is non-compact and locally elliptic. This holds in particular when $F$ is transitive and not the whole of $Sym(\Omega)$. Since, as before, $K_n(v_0) \leq Fix(v_0)$ as an open subgroup, by Example \ref{fixv0} and Proposition \ref{dirun} it follows that $G(F)_{v_0}$ is $p$-adic amenable of norm $1$ if $p > d$, and of norm at most $p$ if $p = d$.

Now to calculate the norms of $K_n(v_0)$ for $n \geq 1$, it suffices to notice that the action of $K_n(v_0)$ on $B_n(v_0)$ coincides with that of $U(\hat{F})$. Therefore $\| G(F)_{v_0} \|_p = \sup \| K_n(v_0) \|_p = \| U(\hat{F}) \|_p$, by the previous example. In particular when $d = p$ is prime, $G(F)$ is $p$-adic amenable of norm $p$ if $\hat{F}$ is transitive (which is equivalent to $F$ being transitive), and of norm $1$ otherwise.
\end{example}

%

\subsection{Non-simplicity of $2$-adic amenable groups}

In Examples \ref{PSL p > 3} and \ref{Suzuki} we have provided, for each odd prime $p$, an infinite discrete simple $p$-adic amenable group of norm $1$. Moreover, in Example \ref{PSL p 2} we have provided an infinite discrete simple $2$-adic amenable group of norm $4$. Our goal in this subsection is to show that $2$-adic amenability of smaller norm is an obstruction to simplicity. More precisely, we will show:

\begin{proposition}
\label{2simple}

Let $G$ be a topologically simple non-abelian t.d.l.c. group. Then $\| G \|_2 \geq 4$ and there exist odd primes $p \neq \ell$ such that $\| G \|_p, \| G \|_\ell > 1$.
\end{proposition}

\begin{remark}
The Classification of Finite Simple Groups implies that the statement holds with $p = 3$ and $\ell = 5$: see Example \ref{12 20}.
\end{remark}

For conciseness, let us say that a t.d.l.c. group $G$ satisfies $(*)$ if $\| G \|_2 \leq 2$, or if $\| G \|_p = 1$ for all but one odd prime. Therefore Proposition \ref{2simple} can be rephrased as: if $G$ is non-abelian and satisfies property $(*)$, then $G$ is not topologically simple. The heavy lifting is done by the finite case, which is a rephrasing of the Odd Order Theorem and Burnside's Theorem:

\begin{lemma}[Feit-Thompson, Burnside]
\label{2simple fin}

Let $G$ be a finite group satisfying property $(*)$. Then $G$ is solvable.
\end{lemma}

\begin{proof}
The condition $\| G \|_2 \leq 2$ is equivalent to $G$ being of order $m$ or $2m$, for some odd number $m$. In the first case, by the Odd Order Theorem, $G$ is solvable. The second case is a well-known argument in finite group theory, which we recall for completeness. Let $G$ act on itself by left translation, so we see $G$ as a subgroup of $S_{2m}$. Let $g \in G$ have order $2$. Then $g$ decomposes as a product of $m$ transpositions. Since $m$ is odd, $g$ is an odd permutation. It follows that $G \cap A_{2m} \neq G$, and so $G$ admits a normal subgroup of index $2$ which has odd order, and so is solvable by the previous case; so $G$ itself is solvable.

If $\| G \|_p = 1$ for all but one odd prime $\ell$, then $G$ has order $2^a \ell^b$. Such a group is solvable by Burnside's Theorem.
\end{proof}

This implies that groups satisfying property $(*)$ have a local structure that imitates that of solvable groups.

\begin{definition}
Let the profinite group $G$ be the inverse limit of the finite groups $(G_i)_{i \in I}$. If each $G_i$ is solvable, we say that $G$ is \emph{prosolvable}. A group is \emph{inductively prosolvable} if it is a directed union of open prosolvable groups.
\end{definition}

\begin{lemma}
\label{ind pros}

Let $G$ be a group satisfying property $(*)$. Then $G$ is inductively prosolvable.
\end{lemma}

\begin{proof}
Theorem \ref{main1} states that $G$ is a directed union of compact open subgroups $(U_i)_{i \in I}$, each of which satisfies property $(*)$ by Proposition \ref{opsub}. Each $U_i$ is profinite, and by Proposition \ref{prof} it is an inverse limit of finite groups $F_j$ satisfying property $(*)$. By Lemma \ref{2simple fin}, each $F_j$ is solvable, and so by definition each $U_i$ is prosolvable.
\end{proof}

Inductively prosolvable groups have a particularly nice Sylow structure, which imitates the structure of Hall subgroups of finite solvable groups \cite[Theorem 2.63]{period}. Proposition \ref{2simple} follows directly from Lemma \ref{ind pros} and the following theorem \cite{Protasov}:

\begin{theorem}[Protasov]
An inductively prosolvable topologically simple group is cyclic of prime order.
\end{theorem}

This theorem follows from a more general study of how properties of normal series of local subgroups pass to normal series of the ambient group. It is a continuous version of similar results by Mal'cev for discrete groups \cite{Malcev}, which imply:

\begin{theorem}[Mal'cev]
A locally solvable discrete simple group is cyclic of prime order.
\end{theorem}

An elementary proof of this theorem was later given by Robinson \cite[Theorem 5.27]{loc_sol}. We adapt his argument to give an elementary proof that an inductively prosolvable \emph{algebraically} simple group is cyclic of prime order, which is a weaker version of Protasov's Theorem and a stronger version of Mal'cev's Theorem, in the locally elliptic setting. We will need the following fact:

\begin{lemma}
\label{cl pros}

A closed subgroup of a prosolvable group is prosolvable.
\end{lemma}

\begin{proof}
Let $G$ be prosolvable, and let $H \leq G$ be a closed subgroup. We need to show that for any open normal subgroup $V \leq H$ the quotient $H/V$ is solvable. By Lemma \ref{opinco}, there exists an open subgroup $U \leq G$ such that $V = H \cap U$. Up to finite-index, $U$ is normal in $G$. Then $H/V = H/(H \cap U) \cong HU/U \leq G/U$ is solvable.
\end{proof}

\begin{proposition}
An inductively prosolvable algebraically simple group is cyclic of prime order.
\end{proposition}

\begin{proof}
Let $G$ be an algebraically simple non-abelian topological group, and suppose by contradiction that $G$ is inductively prosolvable. Since $G$ is non-abelian, there exist elements $a, b \in G$ with $c := [a, b] \neq 1$. Now $G$ is algebraically simple, so $\langle c^G \rangle = G$, where $c^G$ is the conjugacy class of $c$. Thus there exists a finite set $A \subset G$ such that $a, b \in \langle c^A \rangle$: the group generated by the $A$-conjugates of $c$. Let $U$ be a prosolvable open subgroup containing $c$ and $A$, which exists because $G$ is inductively prosolvable. We have $c = [a, b] \in \overline{ \langle c^U \rangle } '$, the commutator subgroup of $\overline{ \langle c^U \rangle }$. Therefore $\overline{\langle c^U \rangle}' = \overline{ \langle c^U \rangle}$. In particular $\overline{ \langle c^U \rangle }$ is a non-trivial closed subgroup of $U$ without non-trivial abelian quotients. But $U$ is prosolvable and so $\overline{\langle c^U \rangle}$ is as well, by Lemma \ref{cl pros}. A finite solvable group admits a non-trivial abelian quotient, and a prosolvable group admits a non-trivial finite solvable quotient, so we reach a contradiction.
\end{proof}

The property in Proposition \ref{2simple} can be replaced by any other property on the order of a finite non-abelian group that obstructs its simplicity.

\begin{example}
\label{12 20}
Instead of using the Burnside Theorem, we could use a more precise fact that follows from the Classification of Finite Simple Groups \cite{Wilson}: a finite non-abelian simple group is either of order a multiple of $3$, or it is a Suzuki group, in which case its order is a multiple of $5$. Therefore in the statement of Proposition \ref{2simple} we may choose $p = 3$ and $\ell = 5$.
\end{example}

We were able to give many examples of infinite \emph{discrete} simple normed $p$-adic amenable groups, which are optimal in terms of the norm for each prime $p$ by Proposition \ref{2simple}. However, we were not able to find non-discrete examples.

A natural candidate for examples of locally elliptic topologically simple groups is the construction presented independently by Willis \cite[Section 3]{lpc_char} and Akin-Glasner-Weiss \cite[Section 4]{Willis2}. This takes as input an increasing family of finite groups $G_n$ containing subgroups $H_n$ such that $H_{n+1}$ contains $H_n$ as a direct factor, in such a way that $G_n$ commutes with its complement; and outputs a non-discrete locally elliptic group $G$ which contains the directed union of the $G_n$ as a dense subgroup. In order for $G$ to be topologically simple, the $G_n$ should be simple, or close to being so. This is easily achieved with groups such as the symmetric groups, or matrix groups of growing degree over a fixed finite field, where in both cases the $H_n$ are the stabilizers of partitions with a growing number of blocks. But these groups have order eventually divisible by arbitrarily large powers of any given prime, and so the resulting $G$ will not be normed $p$-adic amenable for any $p$. On the other hand, choosing as $G_n$ matrix groups of a fixed degree over increasingly larger fields, which is what we did in the discrete case to avoid this divisibility issue, it does not seem possible to choose subgroups $H_n$ with the right property.

We also tried to adapt the constructions in \cite[Section 6]{Osimple}, which failed for similar reasons. Therefore we ask:

\begin{question}
Does there exist, for some prime $p$, a non-discrete normed $p$-adic amenable topologically simple group?
\end{question}

\pagebreak

\section{Characterization over other fields}
\label{other fields}

We move to general non-Archimedean valued fields. In the non-spherically complete case, the notion of (normed) $\mathbb{K}$-amenability is less interesting, just as in the case of norm 1:

\begin{theorem}
\label{main nsc}

Let $G$ be a t.d.l.c. group, $\mathbb{K}$ a non-Archimedean non-spherically complete valued field with residue field $\mathfrak{r}$.
\begin{enumerate}
\item If $\chi(\mathbb{K}) = p > 0$, then $G$ is normed $\mathbb{K}$-amenable if and only if it is compact and $p$-free, in which case it is $\mathbb{K}$-amenable of norm 1.
\item If $\chi(\mathbb{K}) = \chi(\mathfrak{r}) = 0$, then $G$ is normed $\mathbb{K}$-amenable if and only if it is compact, in which case it is $\mathbb{K}$-amenable of norm 1.
\item If $\chi(\mathbb{K}) = 0$ and $\chi(\mathfrak{r}) = p > 0$, then $G$ is normed $\mathbb{K}$-amenable if and only if it is compact and $p^\mathbb{N}$-free, in which case $\| G \|_\mathbb{K} = \min \{ p^k \mid G \text{ is } p^{k+1}\text{-free}\}$.
\end{enumerate}
\end{theorem}

In the spherically complete case, we will show that normed amenability over $\mathbb{K}$ is equivalent to normed amenability over the closure $\mathbb{K}_0$ of its prime field. This allows to reduce the question of normed $\mathbb{K}$-amenability to the cases of $\mathbb{Q}_p$, or to the cases of the trivially valued fields $\mathbb{F}_p$ and $\mathbb{Q}$, which are covered by Schikhof's Theorem \cite{Schik}:

\begin{theorem}
\label{main sc}

Let $G$ be a t.d.l.c. group, $\mathbb{K}$ a non-Archimedean spherically complete valued field, $\mathbb{K}_0$ the closure of its prime field. Then $G$ is normed $\mathbb{K}$-amenable if and only if $G$ is normed $\mathbb{K}_0$-amenable, and the norms coincide. More precisely:
\begin{enumerate}
\item If $\mathbb{K}_0 \cong \mathbb{F}_p$, then $G$ is normed $\mathbb{K}$-amenable, if and only if it is $\mathbb{K}$-amenable of norm 1, if and only if it is locally elliptic and $p$-free.
\item If $\mathbb{K}_0 \cong \mathbb{Q}$, then $G$ is normed $\mathbb{K}$-amenable, if and only if it is $\mathbb{K}$-amenable of norm 1, if and only if it is locally elliptic.
\item If $\mathbb{K}_0 \cong \mathbb{Q}_p$, then $G$ is normed $\mathbb{K}$-amenable if and only if it is locally elliptic and $p^\mathbb{N}$-free, in which case $\| G \|_\mathbb{K} = \min \{ p^k \geq 1 \mid G \text{ is } p^{k+1} \text{-free}\}$.
\end{enumerate}
\end{theorem}

These results show that the only case in which the norm plays an additional role is one which reduces to the $p$-adic case. This is why we only discuss the general case after having thoroughly worked out the $p$-adic one. \\

\begin{remark}
\label{examples general}

Another way to state these theorems is: if $\mathbb{K}$ is spherically complete, then in case $1.$ $G$ is normed $\mathbb{K}$-adic amenable if and only if $\| G \|_p = 1$; in case $2.$ if and only if it is locally elliptic; and in case $3.$ if and only if it is normed $p$-adic amenable, and the norms coincide. Similarly for non-spherically complete fields, where moreover $G$ needs to be compact. Therefore we will not spend any time on examples, since they are already contained in the previous section. Also we will not prove the analogues of Corollaries \ref{clsub}, \ref{dirunn} and \ref{extt}, since they immediately follow from the $p$-adic case using this observation.
\end{remark}

Another implication of these results is a strong restriction on the possible values that $\| G \|_\mathbb{K}$ can take. We will show:

\begin{corollary}
\label{norm attain}
In all cases, $\| G \|_\mathbb{K}$ is attained.
\end{corollary}

This has already been proved in Theorem \ref{comp} in the non-spherically complete case, since if $\| G \|_\mathbb{K} < \infty$ then $G$ is compact by Theorem \ref{main nsc}. So it is really a result about spherically complete fields: it follows directly from Theorem \ref{main sc} and Corollary \ref{K-dirun}.

If $\chi(\mathbb{K}) = 0$ and $\chi(\mathfrak{r}) = p$, then $\| \cdot \|_\mathbb{K}$ takes values in $p^\mathbb{Z}$, even though $| \cdot |_\mathbb{K}$ does not necessarily. By Corollary \ref{norm attain}, this implies that if there exists a normed left-invariant $\mathbb{K}$-mean of norm $\alpha$, then there must also exist a normed left-invariant $\mathbb{K}$-mean of norm at most $p^k \leq \alpha$. Similarly, if $\chi(\mathbb{K}) = \chi(\mathfrak{r})$, then $\| \cdot \|_\mathbb{K}$ only takes the value $1$, even though $\mathbb{K}$ is not necessarily trivially valued. So again by Corollary \ref{norm attain}, if there exists a normed left-invariant $\mathbb{K}$-mean, then there must also exist a left-invariant $\mathbb{K}$-mean of norm $1$.

\subsection{The non-spherically complete case}

In light of Theorem \ref{comp}, in order to prove Theorem \ref{main nsc} we only need to show that if $G$ is normed $\mathbb{K}$-amenable, then $G$ is compact. This will be done just as in \cite[Theorem 2.1]{Schik}, by applying general results on harmonic analysis over non-spherically complete fields taken from \cite{vR}. For the rest of this subsection, $\mathbb{K}$ is a non-Archimedean non-spherically complete valued field. \\

Let $G$ be a \emph{$\sigma$-compact} t.d.l.c. group; that is, a countable increasing union of compact sets. The $G$ is \emph{$\mathbb{N}$-compact} in the sense of Van Rooij: it is homeomorphic to a direct product of countable discrete spaces \cite[Corollary 2.8]{vR}. By \cite[Theorem 7.24, Theorem 7.22]{vR}, every element $\Phi$ of the dual space $C_b(G, \mathbb{K})^*$ satisfies the following property: for every $\varepsilon > 0$ there exists a compact subset $Y \subseteq G$ such that
$$|\Phi(f)|_\mathbb{K} \leq \max \{ \| \Phi \|_{op} \cdot \| f|_Y \|_\infty, \, \varepsilon \cdot \| f \|_\infty \}$$
for all $f \in C_b(G, \mathbb{K})$. We may assume that $Y$ is compact and open, up to replacing it by a compact open subset of $G$ that contains it: this may be achieved by taking a finite cover consisting of cosets of a compact open subgroup. Now given $g \in G$, we have $g Y \cap Y \neq \emptyset$ if and only if $g \in YY^{-1}$, which is a compact subset of $G$. Therefore if $G$ is not compact, there exists $g \in G$ such that $gY \cap Y = \emptyset$.

\begin{proof}[Proof of Theorem \ref{main nsc}]
Let $G$ be normed $\mathbb{K}$-amenable, and suppose by contradiction that $G$ is not compact. Any non-compact t.d.l.c. group contains an open non-compact $\sigma$-compact subgroup (see e.g. \cite[Lemma 2.2]{Schik}), so by Proposition \ref{opsub} we may assume that $G$ is $\sigma$-compact.

Let $m$ be a normed $\mathbb{K}$-mean. By the discussion above, there exists a compact open subset $Y \subseteq G$ such that $|m(f)|_\mathbb{K} \leq \max \{ \| m \|_{op} \cdot \| f|_Y \|_\infty, \, 1/2 \cdot \| f \|_\infty \}$. Since $G$ is not compact, there exists $g \in G$ such that $gY \cap Y = \emptyset$. Let $U \subseteq G$ be any clopen subset of $G$ that is disjoint from $Y$, and consider $\mathbbm{1}_U \in C_b(G, \mathbb{K})$. Then $|m(\mathbbm{1}_U)|_\mathbb{K} \leq \max \{ \| m \|_{op} \cdot 0, \, 1/2 \cdot 1 \} = 1/2.$

This implies on the one hand, using left-invariance, that
$$|m(\mathbbm{1}_Y)|_\mathbb{K} = |m(g \cdot \mathbbm{1}_Y)|_\mathbb{K} = |m(\mathbbm{1}_{gY})|_\mathbb{K} \leq 1/2.$$
On the other hand, using Lemma \ref{ultrametric cor}, that
$$|m(\mathbbm{1}_Y)|_\mathbb{K} = |m(\mathbbm{1}_G) - m(\mathbbm{1}_{G \, \backslash \, Y})|_\mathbb{K} = |1 - m(\mathbbm{1}_{G \, \backslash \, Y})|_\mathbb{K} = 1,$$
which is a contradiction.
\end{proof}

\subsection{The spherically complete case}

Let now $\mathbb{K}$ be spherically complete, so the Hahn--Banach Theorem holds. In this case, Schikhof's Theorem \cite[Theorem 3.6]{Schik} states that $G$ admits a left-invariant $\mathbb{K}$-mean of norm 1 if and only if it is locally elliptic and $\chi(\mathfrak{r})$-free. \\

We want to show that normed $\mathbb{K}$-amenability is equivalent to normed $\mathbb{K}_0$-amenability. However, with our definition it is not clear how one should adapt a normed $\mathbb{K}$-mean to obtain a normed $\mathbb{K}_0$-mean. So we will give an equivalent characterization for normed $\mathbb{K}$-amenability in which the choice of a mean is implicit, thanks to the Hahn--Banach Theorem. The following result is a normed analogue of \cite[Theorem 3.1]{Schik}:

\begin{proposition}
\label{K-mean free}

Let $\mathcal{H}$ denote the closure of the linear span of $\{ (g \cdot f - f) \mid f \in C_b(G, \mathbb{K}), \, g \in G \} \subseteq C_b(G, \mathbb{K})$. Let $\alpha := \inf \{ \| \mathbbm{1}_G - h \|_\infty \mid h \in \mathcal{H} \} \leq 1$. Then $G$ is normed $\mathbb{K}$-amenable if and only if $\alpha > 0$, in which case $\| G \|_\mathbb{K} = \alpha^{-1}$.
\end{proposition}

\begin{proof}
It is enough to show that there exists a left-invariant normed $\mathbb{K}$-mean of norm at most $M$ if and only if $\| \mathbbm{1}_G - h \|_\infty \geq M^{-1}$ for all $h \in \mathcal{H}$. Notice that a normed mean is left-invariant if and only if it vanishes identically on $\mathcal{H}$. \\

Let $m$ be a normed left-invariant mean with $\| m \|_{op} \leq M$, and let $h \in \mathcal{H}$. Then
$$\| \mathbbm{1}_G - h \|_\infty \geq M^{-1} |m(\mathbbm{1}_G - h) |_\mathbb{K} = M^{-1} |1 - 0|_\mathbb{K} = M^{-1}.$$

Now suppose that for all $h \in \mathcal{H}$ it holds $\| \mathbbm{1}_G - h \|_\infty \geq \alpha > 0$. In particular $(\mathbb{K} \cdot \mathbbm{1}_G) \cap \mathcal{H} = \{ 0 \}$. Define $m : (\mathbb{K} \cdot \mathbbm{1}_G) \bigoplus \mathcal{H} \to \mathbb{K} : (\lambda \cdot \mathbbm{1}_G + h) \mapsto \lambda$. Now
$$\alpha^{-1} \| \lambda \mathbbm{1}_G + h \|_\infty = \alpha^{-1} \cdot \| \mathbbm{1}_G + h/ \lambda \|_\infty \cdot |\lambda|_\mathbb{K} \geq |\lambda|_\mathbb{K} = |m(\lambda \mathbbm{1}_G + h)|_\mathbb{K}.$$
By Hahn--Banach, we can extend $m$ to a linear map $m : C_b(G, \mathbb{K}) \to \mathbb{K}$ of norm at most $\alpha^{-1}$. Moreover, by construction, $m(\mathbbm{1}_G) = 1$, and $m$ vanishes on $\mathcal{H}$, so it is left-invariant.
\end{proof}

Recall that Proposition \ref{dirun} was only stated for local fields, since it makes use of the Banach--Alaoglu Theorem. This proposition gives us a new approach to prove the generalization of Proposition \ref{dirun}, at least in the locally elliptic case. This is an analogue of \cite[Lemma 3.3]{Schik}:

\begin{corollary}
\label{K-dirun}

If $G$ is locally elliptic, then $\| G \|_\mathbb{K} = \sup \{ \| U \|_\mathbb{K} \mid U \leq G \text{ is compact and open} \}$. Moreover, $\| G \|_\mathbb{K}$ is attained.
\end{corollary}

\begin{proof}
By Proposition \ref{opsub}, for any $U \leq G$ compact and open we have $\| U \|_\mathbb{K} \leq \| G \|_\mathbb{K}$. Therefore we only need to show that $\| G \|_\mathbb{K} \leq \sup \{ \| U \|_\mathbb{K} \mid U \leq G \text{ is compact and open} \}$ in order to obtain the equality. \\

Let $\mathcal{H} \subseteq C_b(G, \mathbb{K})$ be as in Proposition \ref{K-mean free}. For a compact open subgroup $U \leq G$, let $\mathcal{H}_U \subseteq C_b(U, \mathbb{K})$ be the analogous space for the group $U$. By Proposition \ref{K-mean free}, it suffices to show that
$$\inf \{ \| \mathbbm{1}_G - h \|_\infty \mid h \in \mathcal{H} \} \geq \inf \{ \| \mathbbm{1}_U - h \|_\infty \mid h \in \mathcal{H}_U, \, U \leq G \text{ is compact and open} \}.$$

So let $g_1, \ldots, g_n \in G$ and $f_1, \ldots, f_n \in C_b(G, \mathbb{K})$, so that $h := \sum (g_i \cdot f_i - f_i)$ is a typical element of $\mathcal{H}$. Since $G$ is locally elliptic, there exists a compact open subgroup $U \leq G$ containing all of the $g_i$. Let $h_U := \sum (g_i \cdot f_i|_U - f_i|_U) \in \mathcal{H}_U$. Then $\| \mathbbm{1}_U - h_U \|_\infty \leq \| \mathbbm{1}_G - h \|_\infty$. Passing to the limit, this implies the desired inequality. \\

By Theorem \ref{comp}, since $U$ is compact, $\| U \|_\mathbb{K}$ is attained and takes values in a discrete set. Therefore
$$\inf \{ \| \mathbbm{1}_G - h \|_\infty \mid h \in \mathcal{H} \} = \inf \{ \| \mathbbm{1}_U - h \|_\infty \mid h \in \mathcal{H}_U, \, U \leq G \text{ is compact and open} \}$$
is attained. Since the proof of Proposition \ref{K-mean free} is constructive, it follows that $\| G \|_\mathbb{K}$ is also attained.
\end{proof}

We are ready to prove Theorem \ref{main sc}:

\begin{proof}[Proof of Theorem \ref{main sc}]
Let $G$ be normed $\mathbb{K}$-amenable. Let $\mathcal{H}$ be as in Proposition \ref{K-mean free}, and let $\mathcal{H}_0$ be the corresponding space for $\mathbb{K}_0$. Then $\mathcal{H}_0 \subset \mathcal{H}$, so Proposition \ref{K-mean free} implies that $G$ is $\mathbb{K}_0$-amenable and $\| G \|_{\mathbb{K}_0} \leq \| G \|_\mathbb{K}$. \\

Now let $G$ be normed $\mathbb{K}_0$-amenable. We split three cases.

Suppose that $\mathbb{K}_0 \cong \mathbb{F}_p$, equivalently $\chi(\mathbb{K}) = p > 0$. Since $\mathbb{F}_p$ is trivially valued, if $G$ is normed $\mathbb{F}_p$-amenable, then it is $\mathbb{F}_p$-amenable of norm $1$ and $\| G \|_{\mathbb{F}_p}$ is attained. By Schikhof's Theorem, $G$ is locally elliptic and $p$-free. Therefore $G$ is a directed union of compact open subgroups $U$, each of which is $p$-free. By Theorem \ref{comp} each $U$ is $\mathbb{K}$-amenable of norm 1, and so $G$ is also $\mathbb{K}$-amenable of norm $1$ by Corollary \ref{K-dirun}.

Suppose that $\mathbb{K}_0 \cong \mathbb{Q}$, equivalently $\chi(\mathbb{K}) = \chi(\mathfrak{r}) = 0$. One proves analogously that if $G$ is $\mathbb{K}_0$-amenable, then it is $\mathbb{K}$-amenable and $\| G \|_\mathbb{K} = \| G \|_{\mathbb{K}_0} = 1$.

Finally, suppose that $\mathbb{K} \cong \mathbb{Q}_p$, equivalently $\chi(\mathbb{K}) = 0$ and $\chi(\mathfrak{r}) = p > 0$. By Theorem \ref{main1}, if $G$ is normed $p$-adic amenable, then $G$ is locally elliptic and $p^\mathbb{N}$-free. Let $\| G \|_p =: p^k$. Then $G$ is a directed union of compact open subgroups $U$, each of which is $p^{k+1}$-free. By Theorem \ref{comp} each $U$ is $\mathbb{K}$-amenable of norm at most $p^k$, and so $G$ is also $\mathbb{K}$-amenable of norm at most $p^k$ by Corollary \ref{K-dirun}.
\end{proof}

%
%

\subsection{Ordered fields}
\label{ss:ordered}

An \emph{ordered field} is a field $\mathbb{K}$ equipped with a total ordering $\leq$ which is translation invariant and such that the product of positive elements is positive. It follows from these axioms that $1$ is positive. If $\mathbb{K}$ admits an ordering $\leq$ and a norm $| \cdot |_\mathbb{K}$, the ordering and the norm are said to be \emph{compatible} if $|x|_\mathbb{K} \leq |y|_\mathbb{K}$ whenever $0 \leq x \leq y$. We then say that $\mathbb{K}$ is an \emph{ordered valued field}.

\begin{example}
Let $\mathbb{F}$ be an ordered field, such as $\mathbb{Q}$ or $\mathbb{R}$, and let $\mathbb{F}(X)$ be the field of rational functions over $\mathbb{F}$ with the the norm $|P/Q| = e^{deg(P) - deg(Q)}$. If $P(X) = a_nX^n + \cdots$ and $Q(X) = b_mX^m + \cdots$, we declare $P/Q > 0$ if and only if $a_n/b_m > 0$. This makes $\mathbb{F}(X)$ into an ordered field.

We check that the order and the norm are compatible. We first check on $\mathbb{F}[X]$, so let $P_i(X) = a^i_{n_i}X^{n_i} + \cdots$ for $i = 1, 2$, and assume that $|P_1| < |P_2|$. This means that $deg(P_2) > deg(P_1)$ and so $P_2 - P_1 = a_2^{n_2}X^{n_2} + \cdots$. If now $P_1, P_2 > 0$, this implies that $P_2 - P_1 > 0$ and so $P_2 > P_1$. For the general case let $|P_1/Q_1| < |P_2/Q_2|$ and suppose that $P_1/Q_1, P_2/Q_2 > 0$. Up to taking negatives we may assume that $P_i, Q_i > 0$. Then $|P_1Q_2| < |P_2 Q_1|$ and so the previous case implies that $P_1Q_2 < P_2Q_1$ and so $P_1/Q_1 < P_2/Q_2$.
\end{example}

Let $\mathbb{K}$ be a non-Archimedean ordered valued field. In this case it is natural to impose a further restriction on a mean, namely that if $f \in C_b(G, \mathbb{K})$ is positive everywhere, then $m(f) \geq 0$, just as is done in the real case. We prove that this additional requirement does not add anything, at least when the norm is non-trivial:

\begin{proposition}
\label{ordered}

Let $(\mathbb{K}, |\cdot|_\mathbb{K}, \leq)$ be a non-Archimedean ordered non-trivially valued field. Suppose that $G$ is normed $\mathbb{K}$-amenable. Then a mean $m$ may be chosen such that moreover $\| m \|_{op} = 1$ and $m(f) \geq 0$ whenever $f$ is positive everywhere.
\end{proposition}

\begin{proof}
Let $G$ be a normed $\mathbb{K}$-amenable group. By the axioms we have $0 < 1 < 1+1 < \cdots$, and so $1 = |1|_\mathbb{K} \leq |n|_\mathbb{K}$ for all positive integers $n$. This implies that $\chi(\mathbb{K}) = \chi(\mathfrak{r}) = 0$, and so by Theorems \ref{main nsc} and \ref{main sc} and Corollary \ref{norm attain}, there exists a mean $m$ of norm 1. \\

Then the last condition follows by a standard argument (see \cite[Theorem 3.2]{Pier} for the real case). Indeed, let us start by assuming that $0 \leq f(x) \leq 1$ for all $x \in G$. Suppose by contradiction that $m(f) < 0$. Consider $g = \mathbbm{1}_G - f$. Then $1 = m(\mathbbm{1}_G) = m(f) + m(g) < m(g)$. Since $0 \leq g(x) \leq 1$ for all $x \in G$, it holds $\| g \|_\infty \leq 1$, which implies that $|m(g)|_\mathbb{K} \leq 1$ as $\| m \|_{op} = 1$. But then $1 < m(g) \leq 1$, a contradiction. \\

Now let $f \geq 0$ be an arbitrary continuous bounded function. Since the norm is assumed to be non-trivial, there exists a positive element $\alpha \in \mathbb{K}$ such that $|\alpha|_\mathbb{K} > \| f \|_\infty$, from which it follows that $\alpha > f(x)$ for all $x \in G$ and so $0 \leq \alpha^{-1} f(x) < 1$ for all $x \in G$. By the previous step $m(\alpha^{-1} f) \geq 0$, and so $m(f) = \alpha m(\alpha^{-1} f) \geq 0$ as well.
\end{proof}

The same proof shows that if $f \geq 0$ has an upper bound, then $m(f) \geq 0$. On a non-trivially valued field any norm-bounded function admits an upper bound, but in the trivially valued case every function is bounded, and so we do not obtain information on the ordering from the norm.

\begin{question}
Does the previous proposition hold in the trivially valued case?
\end{question}

\pagebreak

\section{Bounded cohomology of normed $\mathbb{K}$-amenable groups}
\label{smain2}

After going through the formal definitions, we will prove a general vanishing theorem that will imply the ``only if'' part of Theorem \ref{main2} in the cases of compact and discrete groups. The ``if'' part will only be proven for discrete groups, and we will explain why this approach fails in the continuous setting. The hope is that a functorial characterization of continuous bounded cohomology over non-Archimedean valued fields, analogous to \cite[Chapter 7]{Monod}, could be developed. This would allow to compute continuous bounded cohomology by means of other resolutions, for instance using left uniformly continuous or measurable functions. We will explain how this would imply Theorem \ref{main2} in full generality.

\subsection{Continuous bounded cohomology}

We start with the formal definition of continuous bounded cohomology, which was hinted at in the introduction. \\

Let $E$ be a normed $\mathbb{K}[G]$-module. Define
$$C^n_b(G, E) := \{ f : G^{n+1} \to E \mid f \text{ is continuous and } \|f\|_\infty < \infty \}.$$
The diagonal action of $G$ on $G^{n+1}$ induces a linear isometric action of $G$ on $C^n_b(G, E)$:
$$(g \cdot f)(g_0, \ldots, g_n) := g \cdot f(g^{-1}g_0, \ldots, g^{-1}g_n).$$
Define further the coboundary map $\delta^n : C^n_b(G, E) \to C^{n+1}_b(G, E)$ by
$$\delta^n f (g_0, \ldots, g_{n+1}) := \sum\limits_{i = 0}^{n+1} (-1)^i f(g_0, \ldots, \hat{g_i}, \ldots, g_{n+1}).$$
Then $(C^\bullet_b(G, E), \delta^\bullet)$ is a cochain complex of normed $\mathbb{K}[G]$-modules. We consider the subcomplex of invariants $(C^\bullet_b(G, E)^G, \delta^\bullet)$, denote by $Z^\bullet_b(G, E)$ its cocycles, by $B^\bullet_b(G, E)$ its coboundaries, and by $H^\bullet_{cb}(G, E) := Z^\bullet_b(G, E)/B^\bullet_b(G, E)$ its homology.

\begin{definition}
\label{Hcb def}

We call $H^\bullet_{cb}(G, E)$ the \emph{continuous bounded cohomology of $G$ with coefficients in $E$}.  When $G$ is discrete, and so all functions on $G$ are continuous, we simply call it the \emph{bounded cohomology of $G$ with coefficients in $E$}, and denote it by $H^\bullet_b(G, E)$.
\end{definition}

\subsection{Left uniformly continuous bounded cohomology}

Our goal is to prove the vanishing result in Theorem \ref{main2} by the same method as the classical proof in the discrete case \cite[Theorem 3.6]{Frig}. This proof adapts to the continuous setting, but continuity is not enough to make it work: we need a stronger regularity condition.

\begin{definition}
\label{luc def}

Let $f : G \to X$ be a function from the topological group $G$ to the metric space $X$. We say that $f$ is \emph{left uniformly continuous} if for all $\varepsilon > 0$ there exists a neighbourhood $U$ of the identity in $G$ such that $d(f(g), f(ug)) < \varepsilon$ for every $g \in G$ and every $u \in U$.

A function $f : G^n \to X$ is left uniformly continuous if it is when we see $G^n$ as a topological group with the product topology. This corresponds to the product uniform structure induced by the one on $G$.
\end{definition}

For a normed $\mathbb{K}[G]$-module $E$, we denote:
$$LUC_b^n(G, E) := \{ f : G^{n+1} \to E \mid f \text{ is left uniformly continuous and } \| f \|_\infty < \infty \} \subset C^n_b(G, E).$$
This defines the subcomplex $(LUC^\bullet_b(G, E), \delta^\bullet) \subseteq (C^\bullet_b(G, E), \delta^\bullet)$, since sums of left uniformly continuous functions are left uniformly continuous. Moreover:

\begin{lemma}
\label{luc_inv}

Let $E$ be a normed $\mathbb{K}[G]$-module. Then $LUC^n_b(G, E)$ is $G$-invariant.
\end{lemma}

\begin{proof}
Let $f \in LUC^n_b(G, E)$ and $g \in G$. We need to show that $g \cdot f \in LUC^n_b(G, E)$. Denote by $\overline{g} = (g, g, \ldots, g) \in G^{n+1}$. Let $\varepsilon > 0$, and let $U$ be a neighbourhood of the identity in $G^{n+1}$ such that for any $x \in G^{n+1}$ and any $u \in U$ we have $\|f(x) - f(ux)\|_E < \varepsilon$. Let $V$ be a neighbourhood of the identity in $G^{n+1}$ such that $\overline{g}^{-1}V\overline{g} \subseteq U$. Then for all $x \in G^{n+1}$ and all $v \in V$, denoting $u = \overline{g}^{-1}v\overline{g} \in U$:
$$\|(g \cdot f)(x) - (g \cdot f)(vx)\|_E = \|g \cdot (f(\overline{g}^{-1}x) - f(\overline{g}^{-1}vx))\|_E = $$
$$ = \|f(\overline{g}^{-1}x) - f((\overline{g}^{-1}v\overline{g})\overline{g}^{-1}x)\|_E = \|f(\overline{g}^{-1}x) - f(u(\overline{g}^{-1}x))\|_E < \varepsilon.$$
\end{proof}

This allows to define the cohomology of the complex $LUC^\bullet_b(G, E)^G$, denoted $H^\bullet_{lucb}(G, E)$. \\

For compact and discrete groups every continuous map is uniformly continuous, so the two cohomology theories coincide. In the real case, this is true for all locally compact groups, although the proof is more involved \cite[Chapter 7]{Monod}. For non-Archimedean coefficients, a functorial characterization is not (yet) available, so we will have to treat these as two distinct cohomology theories.

\begin{theorem}
\label{luc 0}

Let $G$ be a normed $\mathbb{K}$-amenable group, and let $E$ be a dual normed $\mathbb{K}[G]$-module. Then $H^n_{lucb}(G, E) = 0$ for any $n \geq 1$.
\end{theorem}

\begin{proof}
Let $F$ be a normed $\mathbb{K}[G]$-module such that (up to isometric $G$-isomorphism) $E = F^*$ and the action on $E$ is the dual of the action on $F$. Let $m$ be a left-invariant normed $\mathbb{K}$-mean. We will use $m$ to define, for $n \geq 0$, a collection of maps $j^{n+1} : LUC^{n+1}_b(G, E) \to LUC^n_b(G, E)$, which provides a $G$-equivariant homotopy between the identity and the zero map of $LUC^n_b(G, E)^G$, for $n \geq 1$. \\

Let $f \in LUC^{n+1}_b(G, E), g \in G, \overline{g} \in G^{n+1}, x \in F$. Then $f(g, \overline{g}) \in E = F^*$. So we can define
$$f^{\overline{g}}_x : G \to \mathbb{K} : g \mapsto f(g, \overline{g})(x).$$
Since evaluation at $x$ is continuous as a map from $E$ to $\mathbb{K}$, and $f$ is bounded and continuous, so is $f^{\overline{g}}_x$; that is, $f^{\overline{g}}_x \in C_b(G, \mathbb{K})$. This allows us to define
$$j^{n+1}(f) \in C^n_b(G, E), \, \, \, \, \, \, (j^{n+1}(f)(\overline{g}))(x) = m(f^{\overline{g}}_x) \in \mathbb{K}.$$
Since all objects involved are bounded, $j^{n+1}(f)$ is also bounded. We show that it is left uniformly continuous. We need to show that for any $\varepsilon > 0$, there exists a neighbourhood $U$ of the identity in $G^{n+1}$ such that for all $\overline{g} \in G^{n + 1}$ and all $u \in U$:
$$\varepsilon > \|j^{n+1}(f)(\overline{g}) - j^{n+1}(f)(u\overline{g})\|_E \sim \sup\limits_{\|x\|_F \leq 1} |(j^{n+1}(f)(\overline{g}))(x) - (j^{n+1}(f)(u\overline{g}))(x)|_\mathbb{K} = $$
$$ = \sup\limits_{\|x\|_F \leq 1} |m(f^{\overline{g}}_x) - m(f^{u\overline{g}}_x)|_\mathbb{K} = \sup\limits_{\|x\|_F \leq 1} |m(f^{\overline{g}}_x - f^{u\overline{g}}_x)|_\mathbb{K}.$$
The $\sim$ in the first line indicates that we switched form the operator norm to an equivalent one: see the discussion in Subsection \ref{preli Banach}. Now $m$ is bounded, so up to scaling $\varepsilon$ it is enough to show that:
$$\varepsilon > \sup\limits_{\|x\|_F \leq 1} \|f^{\overline{g}}_x - f^{u\overline{g}}_x\|_\infty = \sup\limits_{\|x\|_F \leq 1} \sup\limits_{g \in G} |f^{\overline{g}}_x(g) - f^{u\overline{g}}_x(g)|_\mathbb{K} = $$
$$ = \sup\limits_{g \in G} \sup\limits_{\|x\|_F \leq 1} |f^{\overline{g}}_x(g) - f^{u\overline{g}}_x(g)|_\mathbb{K} \sim \sup\limits_{g \in G} \|f(g, \overline{g}) - f(g, u\overline{g})\|_E.$$

Since $f$ is left uniformly continuous, there exists a neighbourhood $U'$ of the identity in $G^{n+2}$ such that for all $(g, \overline{g}) \in G^{n+2}$ and all $u' \in U'$ we have $\|f(g, \overline{g}) - f(u'(g, \overline{g}))\|_E < \varepsilon$. Given such a $U'$, there exists a neighbourhood $U$ of the identity in $G^{n+1}$ such that $\{ 1 \} \times U \subset U'$. Choosing such a $U$, we are done. \\

Thus, the map $j^{n+1} : LUC^{n+1}_b(G, E) \to LUC^n_b(G, E)$ is well-defined. That it is linear is clear. It is also bounded:
$$|j^{n+1}(f)(\overline{g})(x)|_\mathbb{K} = |m(f^{\overline{g}}_x)|_\mathbb{K} \leq \|m\|_{op} \|f^{\overline{g}}_x\|_\infty \leq \|m\|_{op} \|f\|_\infty \|x\|_F,$$
so $\|j^{n+1}\|_{op} \leq \|m\|_{op}$. \\

We show that $j^{n+1}$ is $G$-equivariant. First,
$$(g \cdot f)^{\overline{g}}_x(h) = (g \cdot f)(h, \overline{g})(x) = f(g^{-1} h, g^{-1} \overline{g})(x) = (g \cdot f^{g^{-1} \overline{g}}_x)(h).$$
So, using that $m$ is left-invariant:
$$(j^{n+1}(g \cdot f)(\overline{g}))(x) = m((g \cdot f)^{\overline{g}}_x) = m(g \cdot f^{g^{-1} \overline{g}}_x) = m(f^{g^{-1} \overline{g}}_x) = ((g \cdot j^{n + 1}(f))(\overline{g}))(x).$$

Finally, we show that $j^\bullet$ is a homotopy between the identity and the zero map of $LUC^\bullet_b(G, E)^G$. Let $f \in LUC^n_b(G, E)^G$. If we denote by $\overline{g}^i := (g_0, \ldots, \hat{g}_i, \ldots, g_n)$, then:
$$(\delta^n f)^{\overline{g}}_x(g) = (\delta^n f)(g, \overline{g})(x) = f(\overline{g})(x) - \sum\limits_{i = 0}^n (-1)^i f^{\overline{g}^i}_x(g).$$
Therefore, using that $m$ is linear and $m(\mathbbm{1}_G) = 1$:
$$(\delta^{n-1} j^n + j^{n+1} \delta^n)(f)(\overline{g})(x) = \left( \sum\limits_{i = 0}^n (-1)^i m(f^{\overline{g}^i}_x) \right) + $$
$$ + \left( m(f(\overline{g})(x)) - \sum\limits_{i = 0}^n (-1)^i m(f^{\overline{g}^i}_x) \right) = m(f(\overline{g})(x)) = f(\overline{g})(x).$$
This shows that, for all $n \geq 1$, we have $\delta^{n-1} j^n + j^{n + 1} \delta^n = id_n$. So $j^\bullet$ is the desired homotopy. \\

We have found a bounded $G$-equivariant (partial) chain homotopy between the identity and the zero map of $LUC^n_b(G, E)$, for $n \geq 1$. We conclude that $H^n_{lucb}(G, E) = 0$ for all $n \geq 1$.
\end{proof}

If $G$ is compact or discrete, then all continuous functions are automatically left uniformly continuous. Therefore:

\begin{corollary}
Let $G$ be compact or discrete. Then $H^n_{cb}(G, E) = 0$ for any dual normed $\mathbb{K}[G]$-module $E$ and all $n \geq 1$.
\end{corollary}

In the general case a continuous function need not be left uniformly continuous. It could still be possible, however, that the two cohomology theories are isomorphic, as in the real case.

\begin{question}
\label{q main2}
Can this vanishing result be extended to $H^\bullet_{cb}(G, E)$? Are the two cohomology theories isomorphic?
\end{question}

\subsection{Characterization for discrete groups}

We will show that the vanishing of a specific bounded cohomology class implies normed $\mathbb{K}$-amenability. Let $G$ be a discrete group, and consider the space $C_b(G, \mathbb{K}) / \mathbb{K}$, where $\mathbb{K}$ is identified with the subspace of constant functions. Since this is a closed invariant subspace, $C_b(G, \mathbb{K}) / \mathbb{K}$ is a well-defined normed $\mathbb{K}$-vector space with a linear isometric action of $G$. So we can consider the dual space $(C_b(G, \mathbb{K}) / \mathbb{K} )^*$, which is naturally identified with the subspace of elements of $C_b(G, \mathbb{K})^*$ that vanish on constant functions.

For any $g \in G$, let $\delta_g \in C_b(G, \mathbb{K})^*$ be the evaluation map at $g$. Define $J_\mathbb{K}(g_0, g_1) := (\delta_{g_1} - \delta_{g_0}) \in (C_b(G, \mathbb{K}) / \mathbb{K} )^*$ for any $g_0, g_1 \in G$. It is easy to check that $J_\mathbb{K}$ is $G$-equivariant, that it is bounded, and that $\delta^1 J_\mathbb{K} = 0$. Therefore $J_\mathbb{K}$ defines a bounded cohomology class $[J_\mathbb{K}] \in H^1_b(G, (C_b(G, \mathbb{K}) / \mathbb{K} )^*)$. We call it the \emph{$\mathbb{K}$-Johnson class} of $G$ (compare with \cite[Section 3.4]{Frig}).

\begin{theorem}
\label{johnson}

Let $G$ be a discrete group, and $[J_\mathbb{K}]$ its $\mathbb{K}$-Johnson class. If $[J_\mathbb{K}]$ vanishes, then $G$ is normed $\mathbb{K}$-amenable.
\end{theorem}

\begin{proof}
Let $\psi \in C^0_b(G, (C_b(G, \mathbb{K}) / \mathbb{K} )^*)^G$ be such that $\delta^0 \psi = J_\mathbb{K}$. Let $\hat{\psi} \in C^0_b(G, C_b(G, \mathbb{K})^*)^G$ be its pullback. Define $m := e_1 - \hat{\psi}(1) \in C_b(G, \mathbb{K})^*$. Since $\hat{\psi}(1)$ vanishes on constant functions, $m(\mathbbm{1}_G) = 1$, so $m$ is a normed $\mathbb{K}$-mean on $G$. Since $\delta^0 \psi = J_\mathbb{K}$, we also have $m = e_g - \hat{\psi}(g)$ for any $g \in G$. This easily implies that $m$ is left-invariant.
\end{proof}

\begin{corollary}[Theorem \ref{main2}]
Let $G$ be a discrete group. Then the following are equivalent:
\begin{enumerate}
\item $G$ is normed $\mathbb{K}$-amenable.
\item For any dual normed $\mathbb{K}[G]$-module $E$, and all $n \geq 1$, it holds $H^n_b(G, E) = 0$.
\item The $\mathbb{K}$-Johnson class $[J_\mathbb{K}] \in H^1_b(G, (C_b(G, \mathbb{K})/\mathbb{K})^*)$ of $G$ vanishes.
\end{enumerate}
\end{corollary}

Note that the function $J_{\mathbb{K}} : G^2 \to (C_b(G, \mathbb{K})/\mathbb{K})^*$ can also be defined on a topological group. However, it cannot define a continuous bounded cohomology class, because of the following:

\begin{lemma}
\label{fail J}

Suppose that $G$ is non-discrete. Then $J_\mathbb{K}$ is discontinuous everywhere.
\end{lemma}

\begin{proof}
We will show that there exists a $c > 0$ such that whenever $g_0, g_1, h_0, h_1$ are all distinct, we have $\|J_\mathbb{K}(g_0, g_1) - J_\mathbb{K}(h_0, h_1)\|_{op} \geq c$, where we are identifying $(C_b(G, \mathbb{K})/\mathbb{K})^*$ with a subspace of $C_b(G, \mathbb{K})^*$. We choose this $c$ to be such that for all $T \in C_b(G, \mathbb{K})^*$
$$\|T\|_{op} \geq c \sup\limits_{\|f\|_\infty \leq 1} |T(f)|_\mathbb{K},$$
see the discussion in Subection \ref{preli Banach}. Then:
$$\|J_\mathbb{K}(g_0, g_1) - J_\mathbb{K}(h_0, h_1)\|_{op} \geq c \sup\limits_{\|f\|_\infty \leq 1} |(f(g_0) - f(g_1)) - (f(h_0) - f(h_1))|_\mathbb{K}.$$
By the ultrametric inequality, this supremum is at most 1. Since $G$ is t.d.l.c., we can find disjoint clopen sets $U_i, V_i$ containing the $g_i, h_i$ whose union is $G$. Define $f : G \to \mathbb{K}$ to be the characteristic function of $U_0$. Then $f$ is continuous, bounded by 1, and $(f(g_0) - f(g_1)) - (f(h_0) - f(h_1)) = 1$.
\end{proof}

Despite Lemma \ref{fail J}, the map $J_\mathbb{K}$ is still measurable. So if the resolution $L^\infty(G^{\bullet + 1}, E)$ computed $H^\bullet_{cb}(G, E)$, as in the real case \cite[Section 7]{Monod}, this would prove that if $H^\bullet_{cb}(G, E) = 0$ in positive degrees for all dual normed $\mathbb{K}[G]$-modules $E$, then there exists an invariant element of $L^\infty(G, \mathbb{K})^*$ sending $\mathbbm{1}_G$ to $1$, which restricts to a normed invariant $\mathbb{K}$-mean.

\begin{question}
If $H^n_{cb}(G, E) = 0$ for every dual normed $\mathbb{K}[G]$-module $E$ and all $n \geq 1$, is $G$ necessarily normed $\mathbb{K}$-amenable? Does the resolution $L^\infty(G^{\bullet + 1}, E)$ compute $H^\bullet_{cb}(G, E)$? 
\end{question}

We remark that if the resolution $L^\infty(G^{\bullet + 1}, E)$ computes $H^\bullet_{cb}(G, E)$, this would be fruitful to solve question \ref{q main2} as well. Indeed, the proof of Theorem \ref{luc 0} adapts to this complex, if one is equipped with a mean on $L^\infty(G, \mathbb{K})$. The existence of such a mean for a normed $\mathbb{K}$-amenable group could be proven, at least for spherically complete fields, similarly to \cite[Proposition 4.17]{Pier}: normed $\mathbb{K}$-amenable groups are $p^\mathbb{N}$-free for $p = \chi(\mathfrak{r})$, which makes tools from non-Archimedean harmonic analysis available \cite{NHA}.

\subsection{Profinite representations}

Let $\Gamma$ be a discrete group. In the study of the geometry of $\Gamma$ via its finite quotients, it is very useful to consider the \emph{profinite topology}. This is the group topology on $\Gamma$ defined by declaring the collection of finite-index subgroups of $\Gamma$ to be a neighbourhood basis of the identity. The subgroup $K \leq \Gamma$ defined as the intersection of all finite-index subgroups of $\Gamma$ is equal to the closure of the identity element. Thus, the profinite topology on $\Gamma$ is Hausdorff if and only if $K = \{ 1 \}$, that is, if and only if $\Gamma$ is \emph{residually finite}. If $\Gamma$ is countably infinite and residually finite, then it cannot be locally compact in this topology by a Baire category argument.

Although this topology does not necessarily enjoy the good topological properties that we have assumed so far, we will still be able to apply our previous result to prove a vanishing result for the left uniformly continuous bounded cohomology of $\Gamma$ with dual coefficients. This is done by considering the \emph{profinite completion} $\hat{\Gamma}$, which is the inverse limit of the finite quotients of $\Gamma$. In particular $\hat{\Gamma}$ is profinite, so compact and totally disconnected. There is a natural continuous homomorphism $\Gamma \to \hat{\Gamma}$ with kernel $K$ and dense image.

\begin{definition}
A \emph{profinite representation} of $\Gamma$ is a linear isometric action $\Gamma \times E \to E$ on a normed $\mathbb{K}$-vector space $E$, such that for all $x \in E$ the orbit map $\Gamma \to E : \gamma \mapsto \gamma \cdot x$ is continuous with respect to the profinite topology on $\Gamma$.
\end{definition}

Notice that if we are interested in trivial coefficients, the continuity requirement is automatically satisfied.

\begin{lemma}
Let $\Gamma \times E \to E$ be a linear isometric action, and let $x \in E$. Consider the profinite topology on $\Gamma$. Then the following are equivalent:
\begin{enumerate}
\item The orbit map $\Gamma \to E : \gamma \mapsto \gamma \cdot x$ is continuous.
\item The orbit map $\Gamma \to E : \gamma \mapsto \gamma \cdot x$ is right uniformly continuous (defined in the same way as left uniform continuity).
\end{enumerate}
Thus, any profinite representation $\Gamma \times E \to E$ extends to a representation $\hat{\Gamma} \times E \to E$.
\end{lemma}

\begin{proof}
Suppose that the orbit map is continuous and let $\varepsilon > 0$. We need to show that there exists a finite-index subgroup $U \leq \Gamma$ such that for all $\gamma \in \Gamma$ we have $\| \gamma \cdot x - \gamma u \cdot x \|_E < \varepsilon$. We define $U := \{ u \in \Gamma \mid \| u \cdot x - x \|_E \} < \varepsilon$. This is a subgroup of $\Gamma$ by the ultrametric inequality, and being open by continuity it has finite index in $\Gamma$. Then for all $\gamma \in \Gamma$, using that the action is linear and isometric, we have $\| \gamma \cdot x - \gamma u \cdot x \|_E = \| x - u \cdot x \| < \varepsilon$, which concludes the proves the equivalence.

Now since $E$ is Hausdorff and $K$ is the closure of the identity element in $\Gamma$, the orbit map factors through $\Gamma / K$, which is the dense image of $\Gamma$ in the profinite completion $\hat{\Gamma}$. By density and uniform continuity, the orbit map extends uniquely to all of $\hat{\Gamma}$, giving a well-defined action $\hat{\Gamma} \times E \to E$.
\end{proof}

Any left uniformly continuous map $\Gamma^n \to E$ factors through $(\Gamma / K)^n$, and the corresponding map $(\Gamma /K)^n \to E$ extends uniquely to a continuous map $\hat{\Gamma}^n \to E$. Conversely, if $\hat{\Gamma}^n \to E$ is continuous, then the restriction to $(\Gamma/K)^n \leq \hat{\Gamma}^n$ gives a left uniformly continuous map $\Gamma^n \to E$, because $\hat{\Gamma}$ is compact. It follows that $LUC^\bullet_b(\Gamma, E) \cong C^\bullet_b(\hat{\Gamma}, E)$. Moreover, since the action of $\hat{\Gamma}$ on $E$ is a continuous extension of the action of $\Gamma$ on $E$, this isomorphism preserves invariant elements, and so $H^\bullet_{lucb}(\Gamma, E) \cong H^\bullet_{cb}(\hat{\Gamma}, E)$, where the left uniformly continuous bounded cohomology is defined as usual, even if $\Gamma$ is not locally compact, or even not Hausdorff with respect to the profinite topology.

Recall from Theorem \ref{comp} that the normed $\mathbb{K}$-amenability of a profinite group defined by an inverse system of finite groups is:
\begin{enumerate}
\item Equivalent to $p$ not dividing the order of these finite groups, if $\chi(\mathbb{K}) = p$.
\item Automatic if $\chi(\mathbb{K}) = \chi(\mathfrak{r}) = 0$.
\item Equivalent to a bound on the $p$-adic valuation of the orders of the finite groups in this inverse system, if $\chi(\mathbb{K}) = 0$ and $\chi(\mathfrak{r}) = p$.
\end{enumerate}
By definition of $\hat{\Gamma}$, the corresponding finite groups are precisely the finite quotients of $\Gamma$. Now the norm of a finite quotient of $\Gamma$ equals the norm of a Sylow subgroup thereof, which is a quotient of a finite-index subgroup of $\Gamma$, that is, a \emph{virtual quotient} of $\Gamma$. From Theorem \ref{main2}, we conclude:

\begin{corollary}
\label{prof rep}

Let $\Gamma$ be a group equipped with the profinite topology. Suppose that:
\begin{enumerate}
\item $\Gamma$ admits no virtual $p$-quotients, if $\chi(\mathbb{K}) = p$.
\item Make no further assumptions if $\chi(\mathbb{K}) = \chi(\mathfrak{r}) = 0$.
\item $\Gamma$ admits only finitely many virtual $p$-quotients, if $\chi(\mathbb{K}) = 0$ and $\chi(\mathfrak{r}) = p$.
\end{enumerate}
Then $H^n_{lucb}(\Gamma, E) = 0$ for any dual profinite representation $E$ over $\mathbb{K}$, and all $n \geq 1$.
\end{corollary}

\begin{example}
This vanishing result applies to a periodic group in which no element has order $p$. For instance, it applies to Grigorchuk's group, for all $p > 2$. Notice that Grigorchuk's group is residually finite.
\end{example}

\pagebreak

\section{More results on bounded cohomology}
\label{BC}

By now the reader should be convinced that normed $\mathbb{K}$-amenability is a very restrictive condition. In light of the results proven in Section \ref{smain2}, this suggests that there should be plenty of groups with interesting bounded cohomology with coefficients in a normed $\mathbb{K}$-vector space. In this section and the next we try to verify this intuition, focusing on trivial $\mathbb{K}$ coefficients, except in Subsection \ref{ss_QZ} where we can say something about non-trivial coefficients as well. In degree $2$, the \emph{exact} bounded cohomology, which is described by \emph{quasimorphisms}, can be dealt with in much more detail, which will be done in the next section. \\

Our analysis of bounded cohomology will be done mainly in terms of the \emph{comparison map}, from bounded to ordinary continuous cohomology, which will be defined shortly. An injectivity criterion for this map in degree $2$ will be proven in Subsection \ref{ss_QZ}. The main result of this section is Proposition \ref{surj_comp}, giving a criterion for surjectivity in the discrete case for any degree.

\subsection{Bar resolution, comparison map and duality}
\label{preli BC}

Let $E$ be a normed $\mathbb{K}[G]$-module. In our definition of bounded cohomology, we considered invariant functions $f \in C^n_b(G, E)^G$. Any such function is determined by the values it takes on $(n+1)$-tuples with first coordinate 1, and any such element can be uniquely written as $(1, g_1, g_1 g_2, \ldots, g_1 \cdots g_n)$. Therefore defining $\overline{C}^0_b(G, E) := E$ and $\overline{C}^n_b(G, E) := C_b(G^n, E)$ for $n \geq 1$, we have an isometric isomorphism
$$C^n_b(G, E)^G \to \overline{C}^n_b(G, E) : f \mapsto \left[ (g_1, \ldots, g_n) \mapsto f(1, g_1, g_1 g_2, \ldots, g_1 \cdots g_n) \right].$$
Under this isomorphism, the coboundary map becomes: $\delta^n : \overline{C}^n_b(G, E) \to \overline{C}^{n+1}_b(G, E)$ defined by $\delta^0 v(g) = g \cdot v - v$, and for $n \geq 1$:
$$\delta^n f (g_1, \ldots, g_{n+1}) = g_1 \cdot f(g_2, \ldots, g_{n+1}) + $$
$$+ \left( \sum\limits_{i = 1}^n (-1)^i f(g_1, \ldots, g_i g_{i+1}, \ldots, g_{n+1}) \right) + (-1)^{n+1} f(g_1, \ldots, g_n).$$
For instance, in degree two, $\delta^2 f (g, h) = g \cdot f(h) - f(gh) + f(g)$.

\begin{definition}
The cochain complex $(\overline{C}^\bullet_b(G, E), \delta^\bullet)$, where $\delta^\bullet$ is as above, is called the \emph{bar resolution} of $G$.
\end{definition}

We denote by $\overline{Z}^\bullet_b(G, E)$ the cocycles of this cochain complex and by $\overline{B}^\bullet_b(G, E)$ its coboudaries. By the discussion above, the cohomology $\overline{Z}^\bullet_b(G, E)/ \overline{B}^\bullet_b(G, E)$ is canonically isomorphic to $H^\bullet_{cb}(G, E)$. The same can be done in ordinary cohomology, defining $\overline{C}^\bullet(G, E) = \{ f : G^n \to E \mid f \text{ is continuous}\}$ with the same boundary map, which computes the ordinary \emph{continuous cohomology} $H_c^\bullet(G, E)$. \\

The inclusion $\overline{C}^n_b(G, E) \to \overline{C}^n(G, E)$ is a chain map, thus it induces a natural map map $c^\bullet : H^\bullet_{cb}(G, E) \to H^\bullet_c(G, E)$ from continuous bounded cohomology to ordinary continuous cohomology.

\begin{definition}
The map $c^\bullet$ is called the \emph{comparison map}. Its kernel is the \emph{exact continuous bounded cohomology} and is denoted by $EH^\bullet_{cb}(G, E)$.
\end{definition}

Representatives of exact classes in degree $n$ are bounded cocycles that can be expressed as the coboundary of a (not necessarily bounded) cochain in degree $(n - 1)$. We call $f \in \overline{C}^{n-1}(G, E)$ a \emph{quasicocycle} if $\delta^{n-1} f$ is bounded, in which case $\delta^{n-1} f$ represents a class in $EH^n_{cb}(G, E)$. The space of quasicocycles is denoted by $QZ^{n-1}(G, E)$. The coboundary of a quasicocycle represents the trivial class in bounded cohomology if and only if the quasicocycle is \emph{trivial}, meaning it can be expressed as a sum of a cocycle and a bounded cochain. Therefore $\delta^{n-1}$ induces a vector space isomorphism
$$QZ^{n-1}(G, E)/ \left( \overline{Z}^{n-1}(G, E) + \overline{C}^{n-1}_b(G, E) \right) \cong EH^n_{cb}(G, E).$$
When the coefficient module is $\mathbb{K}$ equipped with the trivial $G$-action, we call 1-quasicocycles \emph{quasimorphisms}, and denote their space simply by $Q(G, \mathbb{K})$. The quantity $\|\delta^1 f \|_\infty < \infty$ is called the \emph{defect} of the quasimorphism $f$, and is denoted by $D(f)$. The isomorphism above becomes:
$$Q(G, \mathbb{K}) / \left( Hom(G, \mathbb{K}) + \overline{C}^1_b(G, \mathbb{K}) \right) \cong EH^2_{cb}(G, \mathbb{K}).$$

Here is another interpretation of bounded cohomology from the point of view of duality. We will use this point of view only in Subsection \ref{ss_surj}, so let us restrict to the case of a \emph{discrete} group $G$ and trivial $\mathbb{K}$ coefficients. Denote by $\overline{C}_n(G, \mathbb{K})$ the free $\mathbb{K}$-vector space with basis $G^n$, and define $\partial_n$ by
$$\partial_n : \overline{C}_n(G, \mathbb{K}) \to \overline{C}_{n-1}(G, \mathbb{K}) : (g_1, \ldots, g_n) \mapsto (g_2, \ldots, g_n) +$$
$$+ \left( \sum\limits_{i = 1}^{n-1} (-1)^i (g_1, \ldots, g_i g_{i+1}, \ldots, g_n) \right) + (-1)^n (g_1, \ldots, g_{n-1}).$$
This defines a chain complex $(\overline{C}_\bullet(G, \mathbb{K}), \partial_\bullet)$, whose algebraic dual is precisely $(\overline{C}^\bullet(G, \mathbb{K}), \delta^\bullet)$. We denote the cycles by $Z_n(G, \mathbb{K})$ and the boundaries by $B_n(G, \mathbb{K})$.

Moreover, let us consider the natural $\ell^\infty$ norm on $\overline{C}_n(G, \mathbb{K})$ (compare to the real case \cite[Chapter 6]{Frig}, in which the natural norm is the $\ell^1$-norm). This way it becomes a normed $\mathbb{K}$-vector space, and $\partial_n$ is bounded with respect to this norm. Then $\overline{C}^n_b(G, \mathbb{K})$ is precisely the topological dual of $\overline{C}_n(G, \mathbb{K})$ with respect to this norm.

\subsection{Degrees 0 and 1}

Let $G$ be a topological group (not necessarily t.d.l.c., or even locally compact). By definition of the bar resolution, $H^0_{cb}(G, E) = \{ x \in E \mid g \cdot x = x \text{ for all } g \in G \} = E^G$ for any group $G$ and any normed $p$-adic vector space $E$. \\

In degree 1 we see the first striking difference with the real setting. Let us restrict to trivial $\mathbb{K}$ coefficients, so $H^1_{cb}(G, \mathbb{K}) \cong \overline{Z}^1_b(G, \mathbb{K}) = \{ f : G \to E \text{ bounded} \mid \delta^1 f = 0 \}$ is the space of $\mathbb{K}$-valued bounded homomorphisms of $G$. In the real case, it is easy to see that no homomorphism onto a normed vector space can be bounded, so $H^1_{cb}(G, \mathbb{R}) = 0$ for any group $G$. However, in the non-Archimedean setting, there can be plenty of bounded homomorphisms. \\

First, we make some reductions. Any homomorphism $f : G \to \mathbb{K}$ factors through the abelianization $Ab(G) = G/[G, G]$. If $f$ is continuous, then $f(\overline{[G, G]}) \subseteq \overline{f([G, G])} = \{ 0 \}$, and so $f$ factors through $G/\overline{[G, G]}$, the largest abelian Hausdorff quotient of $G$. Now let $G$ be abelian. If moreover $\chi(\mathbb{K}) = 0$, then the additive group of $\mathbb{K}$ is torsion-free, and a homomorphism $f : G \to \mathbb{K}$ factors through $G/Tor(G)$, where $Tor(G)$ is the torsion subgroup of $G$, and by the same argument as before it also factors through $G/\overline{Tor(G)}$, the largest torsion-free Hausdorff quotient of $G$. Finally, up to scalar, any bounded homomorphism $G \to \mathbb{K}$ takes values in the ring of integers $\mathfrak{o}$. All together, the study of $H^1_{cb}(G, \mathbb{K})$ for arbitrary groups $G$ reduces to the study of continuous homomorphisms $G \to \mathfrak{o}$ where $G$ is an abelian group; if $\chi(\mathbb{K}) = 0$ we may even reduce to $G$ being torsion-free. The property of being t.d.l.c. is preserved under this reduction. \\

If $G$ is discrete and $\mathbb{K} = \mathbb{Q}_p$, so $\mathfrak{o} = \mathbb{Z}_p$, the space of homomorphisms $G \to \mathbb{Z}_p$ is called the space of \emph{$p$-adic functionals on $G$}, and is studied extensively in \cite{p-funct}. For a general locally compact abelian group, some theory on continuous homomorphisms $G \to \mathbb{Z}_p$ is developed in \cite{char}. \\

Here are some simple cases: 

\begin{example}
Let $G = \mathbb{Z}^n$, then $Hom(\mathbb{Z}^n, \mathfrak{o}) \cong Hom(\mathbb{Z}, \mathfrak{o})^n \cong \mathfrak{o}^n$. In light of the previous discussion, if $\chi(\mathbb{K}) = 0$ this simple observation gives a complete description of $H^1_{cb}(G, \mathbb{K})$ for any discrete group $G$ with finitely generated abelianization; namely if $Ab(G)$ has rank $n$, then $H^1_b(G, \mathbb{K}) \cong \mathbb{K}^n$.
\end{example}

\begin{example}
If $\chi(\mathbb{K}) = p > 0$, then we should also take into account the space of $\mathfrak{o}$-valued homomorphism of a finite abelian group. Notice that each element in the additive group $\mathfrak{o}$ has order $p$, because $\mathfrak{o}$ is an $\mathbb{F}_p$-vector space. So if $G$ is a finite $\ell$-group, where $\ell \neq p$, then $Hom(G, \mathfrak{o}) = \{ 0 \}$. Therefore it suffices to consider $\mathfrak{o}$-valued homomorphisms of a finite abelian $p$-group. Any such group may be decomposed as a sum of cyclic $p$-groups, and if $P$ is a cyclic $p$-group, then $Hom(P, \mathfrak{o}) \cong \mathfrak{o}$: any homomorphism is determined by the image of the generator, and any choice of an element of $\mathfrak{o}$ gives a well-defined homomorphism.

This discussion now gives a complete description of $H^1_b(G, \mathbb{K})$ for any discrete group $G$ with finitely generated abelianization. Namely, if $Ab(G)$ has rank $n$, and the $p$-Sylow subgroup of its torsion subgroup is a sum of $k$ cyclic $p$-groups, then $H^1_b(G, \mathbb{K}) \cong \mathbb{K}^{n + k}$.
\end{example}


\begin{example}
Let $\mathbb{K}$ be such that $\chi(\mathfrak{r}) = p > 0$, so $|p|_\mathbb{K} < 1$. Let $G$ be a \emph{$p$-divisible group}, that is an abelian discrete group $G$ such that for any $g \in G$ there exists $h \in G$ such that $pg = h$. Then $H^1_b(G, \mathbb{K}) = 0$. Indeed, repeatedly dividing by $p$ the image of a non-trivial homomorphism gives elements of larger and larger norm, so any non-trivial quasimorphism is unbounded.

Any divisible group is $p$-divisible (see Subsection \ref{ss_div}), so for instance the Pr\"ufer groups and the additive groups of $\mathbb{Q}$-vector spaces are $p$-divisible. Other examples include, for a prime $\ell \neq p$, the groups $\mathbb{Z}_\ell$, and the additive groups of vector spaces over fields of characteristic $\ell$.
\end{example}

\subsection{1-quasicocycles of compactly generated groups}
\label{ss_QZ}

In the real setting, the quasimorphisms of a finitely generated group contain plenty of information on the geometry of the group, the most celebrated example being Bavard Duality, which relates it to stable commutator length \cite{scl}. Large classes of finitely generated groups admit plenty of non-trivial real-valued quasimorphisms, for instance a non-abelian free group \cite{Brooks} \cite{Rolli}. More generally, an acylindrically hyperbolic group $G$ has uncountably-dimensional second exact bounded cohomology with coefficients in any $\mathbb{R}[G]$-module \cite{Osin}. When the coefficient is non-Archimedean, however, the situation is very different. Once again $G$ is a topological group, not necessarily t.d.l.c., or even locally compact.

\begin{proposition}
\label{QZ comp gen}

Let $G$ be compactly generated, $E$ a normed $\mathbb{K}[G]$-module. Then any 1-quasicocycle $f : G \to E$ is bounded.
\end{proposition}

\begin{proof}
Let $f$ be a 1-quasicocycle, so that $\delta^1 f : G^2 \to E$ is bounded. We start with the inequality:
$$\|f(gh)\|_E \leq \max \{ \|g \cdot f(h) - f(gh) + f(g)\|_E, \|g \cdot f(h)\|_E, \|f(g)\|_E \} \leq $$
$$ \leq \max \{ \|\delta^1 f\|_\infty, \|f(g)\|_E, \|f(h)\|_E \}.$$
By induction it follows that
$$\|f(g_1 \cdots g_n)\|_E \leq \max \{ \|\delta^1 f\|_\infty, \|f(g_i)\|_E \mid 1 \leq i \leq n \}.$$
Let $K$ be a compact symmetric generating set for $G$. By continuity $f|_K$ is bounded. Therefore for all $g \in G$, if $g = g_1 \cdots g_n$ for $g_i \in K$, by the previous inequality we conclude
$$\|f(g)\|_E \leq \max\{ \|\delta^1 f\|_\infty, \|f|_K\|_\infty \}.$$
Since the bound is uniform, $f$ is bounded.
\end{proof}

\begin{remark}
The key insight in this proof is that the quasicocycle inequality holds for arbitrary sums with a uniform bound, thanks to the ultrametric inequality. In the real case this is false, and in fact it is what gives \emph{homogeneous} quasimorphisms such an important role, and allows the connection with stable commutator length \cite{scl}.
\end{remark}

In particular, any finitely generated group (seen as a discrete group) admits only bounded 1-quasicocycles.

\begin{corollary}
\label{inj comp}

Let $G$ be a compactly generated group. Then for any normed $\mathbb{K}[G]$-module $E$, the comparison map $c^2: H^2_{cb}(G, E) \to H^2_c(G, E)$ is injective.
\end{corollary}

\begin{proof}
The kernel of the comparison map is the space of 1-quasicocycles modulo true cocycles and bounded functions.
\end{proof}

We will look more closely at quasimorphisms in the next section, where we will also prove triviality results for non-compactly generated groups.

\subsection{Surjectivity of the comparison map}
\label{ss_surj}

Let $G$ be a discrete group, $\mathbb{K}$ a non-Archimedean valued field. In this subsection we investigate the surjectivity of the comparison map $c^n : H^n_b(G, \mathbb{K}) \to H^n(G, \mathbb{K})$. We will prove that a finiteness condition on $G$ (depending on the field $\mathbb{K}$) implies that the map is surjective. This applies to a wide class of examples. For instance, it implies that $c^2$ is an isomorphism for any finitely presented group. For comparison, in the real case, the surjectivity of the comparison map in degree 2 with any normed module characterizes hyperbolic groups among finitely presented groups \cite{min1} \cite{min2}. This result will also apply to more general trivial coefficients: this will be useful when investigating the injectivity of the comparison map in the next section. We refer to \cite{Brown} for the basics of group cohomology that will be used. \\

We start by assuming that $\chi(\mathbb{K}) = 0$, in which case $\mathbb{Z}$ is contained in $\mathfrak{o}$. Since $\mathbb{K}$ has characteristic $0$ its additive group is divisible, so by Corollary \ref{UCT div} the Universal Coefficient Theorem in Cohomology gives an isomorphism
$$H^n(G, \mathbb{K}) \cong Hom_{\mathbb{Z}}(H_n(G, \mathbb{Z}), \mathbb{K}).$$
Explicitly, the map $H^n(G, \mathbb{K}) \to Hom_{\mathbb{Z}}(H_n(G, \mathbb{Z}), \mathbb{K})$ is defined in terms of the duality point of view introduced at the end of Subsection \ref{preli BC}. Given $f \in \overline{C}^n(G, \mathbb{K})$, we can see it as a linear map $f : \overline{C}_n(G, \mathbb{K}) \to \mathbb{K}$, which restricts to a homomorphism $f : \overline{Z}_n(G, \mathbb{Z}) \to \mathbb{K}$. If we started with $f \in \overline{Z}^n(G, \mathbb{K})$, then $f(\overline{B}_n(G, \mathbb{Z})) \subseteq f(\overline{B}_n(G, \mathbb{K})) = 0$, which implies that $f$ descends to a well-defined map $H_n(G, \mathbb{Z}) \to \mathbb{K}$. Finally, two cocycles that differ by a coboundary define the same map.

\begin{lemma}
\label{UCTb Z}
Let $\mathbb{K}$ be a non-Archimedean valued field of characteristic $0$. Then the isomorphism $H^n(G, \mathbb{K}) \cong Hom_{\mathbb{Z}}(H_n(G, \mathbb{Z}), \mathbb{K})$ restricts to an isomorphism
$$c^n(H^n_b(G, \mathbb{K})) \cong \{ f \in Hom_{\mathbb{Z}}(H_n(G, \mathbb{Z}), \mathbb{K}) \mid \| f \|_\infty < \infty \}.$$
\end{lemma}

\begin{proof}
If $f \in \overline{Z}^n_b(G, \mathbb{K})$, then the corresponding homomorphism $\overline{Z}_n(G, \mathbb{Z}) \to \mathbb{K}$ has bounded image by the ultrametric inequality and the boundedness of $\mathbb{Z}$ in $\mathbb{K}$. Thus the induced homomorphism $H_n(G, \mathbb{Z}) \to \mathbb{K}$ also has bounded image.

Conversely, consider a homomorphism $H_n(G, \mathbb{Z}) \to \mathbb{K}$ with bounded image. By lifting to $\overline{Z}_n(G, \mathbb{Z})$, we obtain a map $f : \overline{Z}_n(G, \mathbb{Z}) \to \mathbb{K}$ with bounded image that vanishes on $\overline{B}_n(G, \mathbb{Z})$. Now $\overline{B}_{n-1}(G, \mathbb{Z})$ is a subgroup of the free abelian group $\overline{C}_{n-1}(G, \mathbb{Z})$, so it is itself free. Therefore the exact sequence
$$0 \to \overline{Z}_n(G, \mathbb{Z}) \hookrightarrow \overline{C}_n(G, \mathbb{Z}) \xrightarrow{\partial_n} \overline{B}_{n-1}(G, \mathbb{Z}) \to 0$$
splits, and we obtain a decomposition $\overline{C}_n(G, \mathbb{Z}) = \overline{Z}_n(G, \mathbb{Z}) \bigoplus Y_n$, for some subgroup $Y_n \leq \overline{C}_n(G, \mathbb{Z})$ intersecting $\overline{Z}_n(G, \mathbb{Z})$ trivially. This allows to extend $f$ to a homomorphism $\overline{C}_n(G, \mathbb{Z}) \to \mathbb{K}$ by declaring $f(Y_n) = 0$. Then $f(\overline{C}_n(G, \mathbb{Z})) = f(\overline{Z}_n(G, \mathbb{Z}))$ is still bounded. In particular $f$ is bounded on $G^n$, so it defines a class in the image of the comparison map.
\end{proof}

If $\chi(\mathbb{K}) = \chi(\mathfrak{r}) = 0$, then we may replace $\mathbb{Z}$ by $\mathbb{Q}$, which is bounded in $\mathbb{K}$. Moreover $\mathbb{Q}$ is a field, hence all $\mathbb{Q}$-modules are free, and so the Universal Coefficient Theorem still gives an isomorphism. Therefore:

\begin{lemma}
\label{UCTb Q}
Let $\mathbb{K}$ be a non-Archimedean valued field such that $\chi(\mathbb{K}) = \chi(\mathfrak{r}) = 0$. Then the isomorphism $H^n(G, \mathbb{K}) \cong Hom_{\mathbb{Q}}(H_n(G, \mathbb{Q}), \mathbb{K})$ restricts to an isomorphism
$$c^n(H^n_b(G, \mathbb{K})) \cong \{ f \in Hom_{\mathbb{Q}}(H_n(G, \mathbb{Q}), \mathbb{K}) \mid \| f \|_\infty < \infty \}.$$
\end{lemma}

If $\chi(\mathbb{K}) = p$, then $p$ is not divisible in $\mathbb{K}$, and so $\mathbb{K}$ is not an injective $\mathbb{Z}$-module. But we can replace $\mathbb{Z}$ by $\mathbb{F}_p$, since $\mathbb{K}$ is an $\mathbb{F}_p$-vector space. Again, this is bounded in $\mathbb{K}$ and all $\mathbb{F}_p$-modules are free, so we obtain:

\begin{lemma}
\label{UCTb Fp}
Let $\mathbb{K}$ be a non-Archimedean valued field of characteristic $p > 0$. Then the isomorphism $H^n(G, \mathbb{K}) \cong Hom_{\mathbb{F}_p}(H_n(G, \mathbb{F}_p), \mathbb{K})$ restricts to an isomorphism
$$c^n(H^n_b(G, \mathbb{K})) \cong \{ f \in Hom_{\mathbb{F}_p}(H_n(G, \mathbb{F}_p), \mathbb{K}) \mid \| f \|_\infty < \infty \}.$$
\end{lemma}

We can now prove that finiteness conditions on $G$ imply surjectivity of the comparison map:

\begin{proposition}
\label{surj_comp}

Let $\mathbb{K}$ be a non-Archimedean valued field, $G$ a discrete group. Denote:
\begin{enumerate}
\item $R = \mathbb{F}_p$ if $\chi(\mathbb{K}) = p > 0$;
\item $R = \mathbb{Q}$ if $\chi(\mathbb{K}) = \chi(\mathfrak{r}) = 0$;
\item $R = \mathbb{Z}$ if $\chi(\mathbb{K}) = 0$ and $\chi(\mathfrak{r}) = p > 0$.
\end{enumerate}
If $H_n(G, R)$ is a finitely generated $R$-module, then the comparison map $c^n : H^n_b(G, \mathbb{K}) \to H^n(G, \mathbb{K})$ is surjective.
\end{proposition}

\begin{proof}
By the previous lemmas, the comparison map is surjective if all $R$-linear maps $H_n(G, R) \to \mathbb{K}$ have bounded image. This holds by the ultrametric inequality if $H_n(G, R)$ is finitely generated, since $R$ is bounded in $\mathbb{K}$.
\end{proof}

\begin{remark}
If $\chi(\mathbb{K}) = \chi(\mathfrak{r}) = 0$, then one could also choose $R = \mathbb{Z}$, since Lemma \ref{UCTb Z} also covers this case. However the statement of the proposition is stronger: by the Universal Coefficient Theorem in Homology $H_n(G, \mathbb{Q})$ is a finite-dimensional $\mathbb{Q}$-vector space whenever $H_n(G, \mathbb{Z})$ is a finitely generated abelian group.
\end{remark}

As a sanity check, let us consider the case $n = 1$ and $\mathbb{K} = \mathbb{Q}_p$. If $Ab(G) = H_1(G, \mathbb{Z})$ is a finitely generated abelian group, then Proposition \ref{surj_comp} tells us that the comparison map $c^1 : H^1_b(G, \mathbb{Q}_p) \to H^1(G, \mathbb{Q}_p)$ is surjective. That is, any $\mathbb{Q}_p$-valued homomorphism from $G$ is bounded. This can be seen directly: any homomorphism $G \to \mathbb{Q}_p$ must factor through $Ab(G)$, and any homomorphism from a finitely generated group to $\mathbb{Q}_p$ is necessarily bounded by the ultrametric inequality (this is also a special case of Proposition \ref{QZ comp gen}). \\

Here is a large class of examples to which we can apply Proposition \ref{surj_comp}. A group is of \emph{type $F_n$} if it admits a classifying space with compact $n$-skeleton. In particular $G$ is of type $F_1$ if and only if it is finitely generated, and it is of type $F_2$ if and only if it is finitely presented. We say that $G$ is of \emph{type $F_\infty$} if it is of type $F_n$ for all $n$.

\begin{corollary}
\label{Fn surj}

Let $\mathbb{K}$ be a non-Archimedean valued field, $G$ a discrete group of type $F_n$. Then $c^n : H^n_b(G, \mathbb{K}) \to H^n(G, \mathbb{K})$ is surjective.
\end{corollary}

\begin{proof}
If $G$ is of type $F_n$, then $H_n(G, R)$ is a finitely generated $R$-module for any ring $R$. The result then follows from Proposition \ref{surj_comp}.
\end{proof}

In degree $2$ we obtain the following:

\begin{corollary}
Let $\mathbb{K}$ be a non-Archimedean valued field, $G$ a finitely presented group. Then $c^2 : H^2_b(G, \mathbb{K}) \to H^2(G, \mathbb{K})$ is a vector space isomorphism.
\end{corollary}

\begin{proof}
Injectivity holds for any finitely generated group by Corollary \ref{inj comp}, and surjectivity holds for any finitely presented group by Proposition \ref{Fn surj}.
\end{proof}

These results also apply to other trivial coefficients. Let $\mathbb{K}$ be a non-Archimedean valued field, and let $R$ be as in Proposition \ref{surj_comp}. Let $M$ be an injective $R$-module; that is, $M$ is a divisible abelian group if $R = \mathbb{Z}$, and otherwise there are no constraints. Suppose that $M$ is endowed with a non-Archimedean norm which is compatible with the norm on $R$ coming from $\mathbb{K}$. Explicitly, suppose that there is a map $\| \cdot \| : M \to \mathbb{R}_{\geq 0}$ such that

\begin{enumerate}
\item $\| m \| = 0$ if and only if $m = 0$;
\item $\| m + n \| \leq \max \{ \| m \|, \| n \| \}$;
\item $\| \alpha m \| = |\alpha|_\mathbb{K} \| m \|$ for all $\alpha \in \mathbb{K}$.
\end{enumerate}

This corresponds to Definition \ref{def_nvs} when $R = \mathbb{Q}$ or $\mathbb{F}_p$. We can then define analogously the bounded cohomology with coefficients in $M$ and the corresponding comparison map. The proofs of Lemmas \ref{UCTb Z}, \ref{UCTb Q} and \ref{UCTb Fp} carry over, as $M$ is an injective $R$-module, and so the Universal Coefficient Theorem in Cohomology gives the desired isomorphisms (see Corollary \ref{UCT div} for $R = \mathbb{Z}$, the other two cases are dealt with just as before, since $M$ is a vector space). We conclude:

\begin{proposition}
\label{surj_comp 2}

Let $\mathbb{K}$ be a non-Archimedean valued field, $G$ a discrete group, $R$ as in Proposition \ref{surj_comp}. Let $M$ be an injective $R$-module endowed with a compatible non-Archimedean norm. If $H_n(G, R)$ is a finitely generated $R$-module, then the comparison map $c^n : H^n_b(G, M) \to H^n(G, M)$ is surjective.
\end{proposition}

This proposition will be useful when we look at the injectivity of the comparison map in the next subsection (see Corollary \ref{inj Fn}). \\

The techniques in this subsection do not apply to non-trivial coefficients. Therefore we ask:

\begin{question}
Let $E$ be a normed $\mathbb{K}[G]$-module. When is the comparison map $c^n : H^n_b(G, E) \to H^n(G, E)$ surjective?
\end{question}

\pagebreak

\section{Quasimorphisms}
\label{s_qm}

In Proposition \ref{QZ comp gen}, we saw that $\mathbb{K}$-valued continuous quasimorphisms of compactly generated groups are bounded. We now look at quasimorphisms of groups which are not compactly generated. Quasimorphisms $\mathbb{Q}_p \to \mathbb{Q}_p$ and $\mathbb{Q} \to (\mathbb{Q}, |\cdot|_p)$ have been treated in \cite{Kionke}, where they were used to provide a construction of $\mathbb{Q}_p$. With little more work, one can use the results in \cite{Kionke} to prove that $EH^2_{cb}(\mathbb{Q}_p, \mathbb{Q}_p) = 0$ and $EH^2_b(\mathbb{Q}_p, \mathbb{Q}_p), EH^2_b(\mathbb{Q}, \mathbb{Q}_p) = 0$, where $\mathbb{Q}_p$ and $\mathbb{Q}$ are seen as discrete groups. We will obtain these results by a different method. \\

Once again we see the trichotomy that was present in Theorems \ref{main nsc} and \ref{main sc}. If $\chi(\mathbb{K}) = \chi(\mathfrak{r})$, then all $\mathbb{K}$-valued continuous quasimorphisms, from \emph{any} group, are trivial. This is not true in the case $\chi(\mathbb{K}) = 0, \chi(\mathfrak{r}) = p$; so also in this setting, the first interesting case is $\mathbb{K} = \mathbb{Q}_p$. We are able to classify $\mathbb{Q}_p$-valued quasimorphisms in a suitable sense, and it turns out that the existence of a non-trivial one imposes very strong conditions on the subgroup structure of the group. This generalizes to finite extensions of $\mathbb{Q}_p$. All of this is in stark constrast to the real case, in which quasimorphism spaces can be very complex. 

In the last subsection we investigate higher-dimensional quasicocycles, proving triviality results analogous to those in degree $2$ for fields satisfying $\chi(\mathbb{K}) = \chi(\mathfrak{r})$, and weaker ones for general fields. The latter will have stronger consequences for discrete groups, when combined with the surjectivity results from Subsection \ref{ss_surj}. \\

The following notation will simplify some of our statements:

\begin{definition}
\label{def dH}

Let $f : G \to \mathbb{K}$ be a continuous quasimorphism. We denote by
$$d_H(f) := \inf \{ \| f - h \|_\infty \mid h : G \to \mathbb{K} \text{ is a continuous homomorphism} \}.$$
\end{definition}

In particular $d_H(f) < \infty$ if and only if $f$ is a trivial quasimorphism. Notice that the subspace of homomorphisms is closed with respect to the norm $\| \cdot \|_\infty$ in the space of quasimorphisms. This is because the defect $D$ is continuous, and $f$ is a homomorphism if and only if $D(f) = 0$. This implies that the infimum in the definition of $d_H$ is always attained. In particular $d_H(f) = 0$ if and only if $f$ is a homomorphism. \\

If the norm on $\mathbb{K}$ is the trivial one, then any $\mathbb{K}$-valued map is clearly bounded, and this is true in particular of quasimorphisms. Therefore in this section \emph{all non-Archimedean valued fields are assumed to be non-trivially valued}. Groups will always be topological unless otherwise stated, and will not be assumed to be t.d.l.c., although any continuous map $G \to \mathbb{K}$ factors through the component group.

\subsection{The defect group}

We start with our first example of a non-trivial quasimorphism:

\begin{example}
\label{pruf}

Consider the Pr\"ufer $p$-group $\mathbb{Z}(p^\infty) = \mathbb{Q}_p / \mathbb{Z}_p$, equipped with the discrete topology. We have already seen that $\mathbb{Z}(p^\infty)$ is locally finite and not normed $p$-adic amenable. In particular it is periodic; it follows that any homomorphism $\mathbb{Z}(p^\infty) \to \mathbb{Q}_p$ is trivial, and so all trivial quasimorphisms are bounded. Therefore, to show that a given quasimorphism is non-trivial, it suffices to show that it is unbounded.

We choose a section of the canonical projection $\mathbb{Q}_p \to \mathbb{Z}(p^\infty)$:
$$\iota : \mathbb{Z}(p^\infty) \to \mathbb{Q}_p : \sum\limits_{i = -1}^{-n} a_i p^i \mod 1 \mapsto \sum\limits_{i = -1}^{-n} a_i p^i.$$
When calculating $\iota(x + y)$ and $\iota(x) + \iota(y)$, the only difference is that in $\iota(x) + \iota(y)$ there may be some terms that carry over to the non-negative powers of $p$. In other words, $\iota(x + y) - \iota(x) - \iota(y) \in \mathbb{Z}_p$. So $\iota$ is a quasimorphism of defect 1. Being unbounded, it is non-trivial. We will call this the \emph{standard section} $\iota : \mathbb{Z}(p^\infty) \to \mathbb{Q}_p$.
\end{example}

This example will turn out to be fundamental in the description of $\mathbb{Q}_p$-valued quasimorphisms. It can be generalized to any non-Archimedean valued field:

\begin{example}
Let $\mathbb{K}$ be a non-Archimedean valued field, and $\mathfrak{o}$ its ring of integers. The group $\mathbb{K}/\mathfrak{o}$ is a discrete additive group, and so any choice of a section $\iota : \mathbb{K}/\mathfrak{o} \to \mathbb{K}$ is continuous. Since the norm on $\mathbb{K}$ is supposed to be non-trivial, $\iota$ is moreover unbounded. The difference between $\iota(x + y)$ and $\iota(x) + \iota(y)$ is in $\mathfrak{o}$, so this is a continuous unbounded quasimorphism of defect at most $1$. It is not necessarily non-trivial, as we shall shortly see.
\end{example}

This motivates the following definition:

\begin{definition}
Let $\mathbb{K}$ be a non-Archimedean valued field, $\mathfrak{o}$ its ring of integers. The discrete additive group $\mathbb{K}/\mathfrak{o}$ is called the \emph{defect group} of $\mathbb{K}$, and is denoted by $D(\mathbb{K})$. This group is equipped with the quotient seminorm induced by $\mathbb{K}$, which is actually a norm: it follows from Lemma \ref{ultrametric cor} that $|x \mod \mathfrak{o} |_{D(\mathbb{K})} = |x|_\mathbb{K}$ for any $x \notin \mathfrak{o}$.
\end{definition}

\begin{remark}
The name \emph{defect group} has an established meaning in modular representation theory. Since that context is disjoint from ours, this overlap in terminology should not cause any confusion.
\end{remark}

We have seen that $D(\mathbb{Q}_p) \cong \mathbb{Z}(p^\infty)$. Let us see another example.

\begin{example}
\label{def grp Fqt}
Let $\mathbb{F}$ be a field and consider the field $\mathbb{F}((X))$ of Laurent series with the norm $|\sum a_i X^i| = e^{-n}$, where $n$ is the smallest integer such that $a_n \neq 0$. Then $\mathfrak{o} = \mathbb{F}[[X]]$. It follows that $D(\mathbb{F}((X))) \cong \bigoplus\limits_{n \geq 1} \mathbb{F}$. An element of $D(\mathbb{F}((X)))$ may be written uniquely as $\sum_{i \leq -1} a_i X^i \mod \mathbb{F}[[X]]$, which gives a canonical section $\iota : D(\mathbb{F}((X))) \to \mathbb{F}((X))$. Note that this section is actually a homomorphism.
\end{example}

The examples of $\mathbb{Q}_p$ and $\mathbb{F}((X))$ behave very differently, in that $\iota$ is a non-trivial quasimorphism in the first case, and it is a homomorphism in the second case. This is part of a general pattern:

\begin{proposition}
\label{class def}

Let $\mathbb{K}$ be a non-Archimedean valued field, and let $\mathfrak{r}$ be its residue field.
\begin{enumerate}
\item If $\chi(\mathbb{K}) = p > 0$, then $D(\mathbb{K})$ is isomorphic to the additive group of an $\mathbb{F}_p$-vector space, and there exists a homomorphic section $\iota : D(\mathbb{K}) \to \mathbb{K}$.
\item If $\chi(\mathbb{K}) = \chi(\mathfrak{r}) = 0$, then $D(\mathbb{K})$ is isomorphic to the additive group of a $\mathbb{Q}$-vector space, and there exists a homomorphic section $\iota : D(\mathbb{K}) \to \mathbb{K}$.
\item If $\chi(\mathbb{K}) = 0$ and $\chi(\mathfrak{r}) = p > 0$, then $D(\mathbb{K})$ is isomorphic to a direct sum of copies of $\mathbb{Z}(p^\infty)$, and any section $\iota : D(\mathbb{K}) \to \mathbb{K}$ is a non-trivial quasimorphism.
\end{enumerate}
\end{proposition}

\begin{proof}
In the first two cases, let $\mathbb{K}_0$ be the prime field, which is isomorphic to either $\mathbb{F}_p$ or $\mathbb{Q}$. The hypothesis implies that in both cases $\mathbb{K}_0$ is trivially valued. It follows that $\mathfrak{o}$ is a $\mathbb{K}_0$-vector subspace of $\mathbb{K}$. In particular, it admits a (not necessarily closed) complement $\mathfrak{o}^\perp$, and so there exists a homomorphic section $\iota : D(\mathbb{K}) \xrightarrow{\cong} \mathfrak{o}^\perp \hookrightarrow \mathbb{K}$. \\

In the third case $\chi(\mathbb{K}) = 0$, and so the additive group of $\mathbb{K}$ is divisible. It follows that $D(\mathbb{K})$ is divisible, being a quotient thereof. Moreover, for any $x \in \mathbb{K}$ and for $k$ large enough, $|p^k x|_\mathbb{K} < 1$, and so $D(\mathbb{K})$ is a divisible $p$-group. We conclude by the Classification Theorem from Subsection \ref{ss_div} that $D(\mathbb{K}) \cong \mathbb{Z}(p^\infty)^I$ for some index set $I$. Since $D(\mathbb{K})$ is periodic and $\mathbb{K}$ is torsion-free, there exists no non-trivial homomorphism $D(\mathbb{K}) \to \mathbb{K}$. So any section, being unbounded, is automatically a non-trivial quasimorphism.
\end{proof}

\begin{example}
Any finite extension of $\mathbb{Q}_p$ has as defect group a finite direct sum of copies of $\mathbb{Z}(p^\infty)$; and $D(\mathbb{C}_p)$ is an infinite direct sum of copies of $\mathbb{Z}(p^\infty)$.
\end{example}

\subsection{$\mathbb{K}$-valued quasimorphisms and $D(\mathbb{K})$-valued homomorphisms}

Let us fix a non-Archimedean valued field $\mathbb{K}$ with ring of integers $\mathfrak{o}$ and defect group $D(\mathbb{K}) = \mathbb{K}/\mathfrak{o}$. Denote by $\pi : \mathbb{K} \to D(\mathbb{K})$ the canonical projection and let $\iota : D(\mathbb{K}) \to \mathbb{K}$ be a section. If $\chi(\mathbb{K}) = \chi(\mathfrak{r})$, we will choose $\iota$ to be a homomorphism: this is possible by Proposition \ref{class def}. Let $G$ be a topological group. \\

The following simple lemma is the key observation at the basis of all results in this section.

\begin{lemma}
\label{qm key}

Let $f : G \to \mathbb{K}$ be a continuous quasimorphism with $D(f) \leq 1$. Then $\pi f : G \to D(\mathbb{K})$ is a continuous homomorphism. Conversely, let $h : G \to D(\mathbb{K})$ be a continuous homomorphism. Then $\iota h : G \to \mathbb{K}$ is a continuous quasimorphism and $D(\iota h) \leq 1$.

Moreover, $\| \iota \pi f - f \|_\infty \leq 1$, and $\pi \iota h = h$.
\end{lemma}

\begin{proof}
Let $f : G \to \mathbb{K}$ be a continuous quasimorphism with $D(f) \leq 1$. Then for all $x, y \in G$, we have $f(xy) - f(x) - f(y) \in \mathfrak{o}$. Therefore the map $\pi f : G \to \mathbb{K} \to D(\mathbb{K})$ is a continuous homomorphism. Conversely, let $h : G \to \mathbb{K}$ be a continuous homomorphism. Then $\iota$ is a continuous quasimorphism of defect at most $1$, so $\iota h$ is too.

For all $x \in \mathbb{K}$, we have $\pi (\iota \pi(x)) = \pi(x)$, since $\iota$ is a section of $\pi$. This implies that $\iota \pi (x) - x \in \mathfrak{o}$. So for all $x \in G$ it holds $|\iota \pi f(x) - f(x)|_\mathbb{K} \leq 1$. The fact that $\pi \iota h = h$ follows again from $\iota$ being a section of $\pi$.
\end{proof}

Since any quasimorphism has a multiple of defect at most $1$, this lemma creates a bijective correspondence between $\mathbb{K}$-valued quasimorphisms and $D(\mathbb{K})$-valued homomorphisms (up to scalars and bounded functions). When $\iota$ is a homomorphism, it implies that all quasimorphisms are trivial.

\begin{theorem}
\label{qm split}

Let $\mathbb{K}$ be a non-Archimedean valued field, and suppose that $\chi(\mathbb{K}) = \chi(\mathfrak{r})$. Then for any group $G$, all $\mathbb{K}$-valued continuous quasimorphisms are trivial.
\end{theorem}

\begin{proof}
By Proposition \ref{class def}, there exists a homomorphic section $\iota : D(\mathbb{K}) \to \mathbb{K}$. Let $f : G \to \mathbb{K}$ be a continuous quasimorphism. Up to scalar, suppose that $D(f) \leq 1$. Then by Lemma \ref{qm key}, $\iota \pi f : G \to \mathbb{K} \to D(\mathbb{K}) \to \mathbb{K}$ is a continuous quasimorphism at a bounded distance from $f$. Since $\iota$ and $\pi f$ are homomorphisms, so is $\iota \pi f$.
\end{proof}

\begin{remark}
This result will be generalized to higher degrees in Theorem \ref{qz split}.
\end{remark}

\begin{example}
The previous theorem applies to all fields of positive characteristic, all ordered fields (see the proof of Proposition \ref{ordered}), all fields of rational functions, Laurent series, and all extensions of such fields.
\end{example}

Lemma \ref{qm key} also allows to characterize unbounded $\mathbb{K}$-valued quasimorphisms:

\begin{proposition}
\label{class qm 0}

Let $G$ be a group, and let $f : G \to \mathbb{K}$ be a continuous quasimorphism with $D(f) \leq 1$. Let $U := f^{-1}(\mathfrak{o})$. Then $U$ is normal in $G$ and the following are equivalent:
\begin{enumerate}
\item $f$ is unbounded.
\item $\pi f : G \to D(\mathbb{K})$ is unbounded.
\item $\pi f$ induces an isomorphism between $G/U$ and an unbounded subgroup of $D(\mathbb{K})$.
\end{enumerate}
\end{proposition}

\begin{proof}
The implication $1. \Rightarrow 2.$ is clear.

The map $\pi f$ is a homomorphism by Lemma \ref{qm key}, and so $U = f^{-1}(\mathfrak{o}) = \ker(\pi f)$ is a normal subgroup of $G$. Moreover $U$ is open, and so $G/U$ is discrete. The implication $2. \Rightarrow 3.$ follows. 

Finally, let $A$ be an unbounded subgroup of $D(\mathbb{K})$ such that $\pi f$ induces an isomorphism $G/U \cong A$. By Lemma \ref{qm key}, the composition $G \to G/U \cong A \hookrightarrow D(\mathbb{K}) \xrightarrow{\iota} \mathbb{K}$, which is an unbounded quasimorphism, is at a bounded distance from $f$. This proves $2. \Rightarrow 3.$
\end{proof}

So when looking at unbounded quasimorphisms, it is enough to restrict to groups which admit unbounded subgroups of $D(\mathbb{K})$ as quotients. For example, if $G$ is normed $\mathbb{K}$-amenable, then by Theorem \ref{luc 0} every left uniformly continuous quasimorphism is trivial; but this observation implies a much stronger statement:

\begin{corollary}
Let $G$ be normed $\mathbb{K}$-amenable. Then any continuous $\mathbb{K}$-valued quasimorphism is bounded.
\end{corollary}

\begin{proof}
By Proposition \ref{class qm 0} and Proposition \ref{quot}, it suffices to show that an infinite subgroup of $D(\mathbb{K})$ is not normed $\mathbb{K}$-amenable. If $\chi(\mathfrak{r}) = p$, then Proposition \ref{class def} implies that $D(\mathbb{K})$ is a $p$-group, and so any infinite subgroup of $D(\mathbb{K})$ is not $p^\mathbb{N}$-free, so not normed $\mathbb{K}$-amenable by Theorems \ref{main nsc} and \ref{main sc}. Similarly, if $\chi(\mathfrak{r}) = 0$, then $D(\mathbb{K})$ is torsion-free, and so any infinite subgroup of $D(\mathbb{K})$ is not locally elliptic, so not normed $\mathbb{K}$-amenable.
\end{proof}

Since $D(\mathbb{K})$ is abelian, we may make the following restriction:

\begin{corollary}
\label{qm abelian quot}

Any continuous $\mathbb{K}$-valued quasimorphism of $G$ is at a bounded distance from one that factors through $G/\overline{[G, G]}$, the largest abelian Hausdorff quotient of $G$.
\end{corollary}

\begin{example}
Let $G$ be a group such that $G/\overline{[G, G]}$ is compact. Then for any $\mathbb{K}$, all quasimorphisms $G \to \mathbb{K}$ are bounded. This applies in particular to perfect groups or topologically simple groups.
\end{example}


This is already very surprising. For comparison, $EH^2_b(F_2, \mathbb{R})$ is uncountably dimensional, while $EH^2_b(Ab(F_2), \mathbb{R}) = EH^2_b(\mathbb{Z}^2, \mathbb{R}) = 0$ (see \cite[Chapter 2]{Frig}). More generally, the second bounded cohomology with trivial real coefficients vanishes for any abelian group, so the analogue of this result cannot hold for any non-abelian group admitting a non-trivial real-valued quasimorphism.

\begin{example}
\label{EH2b Zli}

Let $G$ be a group that does not admit $\mathbb{Z}(p^\infty)$ as a discrete quotient. Then all $\mathbb{Q}_p$-valued continuous quasimorphisms of $G$ are bounded; in particular $EH^2_{cb}(G, \mathbb{Q}_p) = 0$. This applies to all periodic groups with a bound on the order of their cyclic $p$-subgroups, so for instance $EH^2_b(\mathbb{Z}(\ell^\infty), \mathbb{Q}_p) = 0$ for all primes $\ell \neq p$.
\end{example}

\begin{example}
The previous example applies to $\mathbb{Q}_{\ell}$ for all primes $\ell \neq p$. Indeed, any open normal subgroup of $\mathbb{Q}_{\ell}$ contains $\ell^k \mathbb{Z}_{\ell}$ for some $k$, so any discrete quotient of $\mathbb{Q}_{\ell}$ is also a quotient of $\mathbb{Q}_{\ell} / \ell^k \mathbb{Z}_{\ell} \cong \mathbb{Z}(\ell^\infty)$, which is an $\ell$-group. Since $p \neq \ell$, the Pr\"ufer $p$-group cannot be a discrete quotient of $\mathbb{Q}_{\ell}$.
\end{example}

Corollary \ref{qm abelian quot} is potentially useful, because when reducing to locally compact abelian groups one can use the tools from duality theory. For instance, in \cite{char} it is proven that choosing $\mathbb{Z}(p^\infty)$ as a value group for a theory of characters, yields a duality analogous to Pontryagin's classical one for t.d.l.c. abelian \emph{topological $p$-groups}: those in which all elements $x$ satisfy $p^k x \xrightarrow{k \to \infty} 0$. This is relevant since continuous homomorphisms onto $\mathbb{Z}(p^\infty)$ are closely related to continuous quasimorphisms to $\mathbb{Q}_p$, by Lemma \ref{qm key}. We will not use this approach in what follows: we are able to prove strong results even without using local compactness; but we believe that some more results may be proven this way. \\

Here is a case where reducing to abelian groups one can prove more. The following result is reminiscent of the extension of quasicocycles from hyperbolically embedded subgroups (\cite{Osin} in degree $2$ and \cite{FPS} in higher degrees).

\begin{corollary}
Let $G$ be abelian, and let $U \leq G$ be an open subgroup. Then any non-trivial continuous quasimorphism of $U$ is at a bounded distance from one which extends to a non-trivial continuous quasimorphism of $G$. In particular, if $EH^2_{cb}(G, \mathbb{K}) = 0$, then $EH^2_{cb}(U, \mathbb{K}) = 0$.
\end{corollary}

\begin{proof}
If $\chi(\mathbb{K}) = \chi(\mathfrak{r})$, then the statement is empty by Theorem \ref{qm split}. So we may assume that $\chi(\mathbb{K}) = 0$ and $\chi(\mathfrak{r}) = p$, in which case $D(\mathbb{K})$ is divisible by Proposition \ref{class def}. By Lemma \ref{qm key}, it suffices to consider unbounded quasimorphisms of the form $\iota f$, where $\iota : D(\mathbb{K}) \to \mathbb{K}$ is a section and $f : U \to D(\mathbb{K})$ is a continuous homomorphism. By Corollary \ref{div ext}, we may extend $f$ to a continuous homomorphism $f^\# : G \to D(\mathbb{K})$. Then $\iota f^\#$ is a continuous quasimorphism of $G$ which extends $f$. If $\iota f^\#$ is trivial, then it is at a bounded distance from a continuous homomorphism $g : G \to \mathbb{K}$, so $\iota f$ is trivial too, being at a bounded distance from $g|_U$.
\end{proof}

For comparison, $EH^2_b(F_2, \mathbb{R})$ is uncountably dimensional, but there are many groups containing a non-abelian free group all of whose quasimorphisms are trivial, for instance $\operatorname{SL}_n(\mathbb{Z})$ for $n \geq 3$ \cite[Example 5.33]{scl}. \\

So far we have considered \emph{algebraic} properties of $D(\mathbb{K})$ that obstruct the existence of an unbounded continuous quasimorphism $G \to \mathbb{K}$. We can also consider a \emph{topological} property, namely the discreteness of $D(\mathbb{K})$. The kernel of a continuous homomorphism from $G$ to a discrete group is open. Let $R_0$ be the intersection of all open normal subgroups of $G$. This is called the \emph{discrete residual} of $G$: any continuous homomorphism from $G$ to a discrete group factors through $G/R_0$. The group $G/R_0$ is \emph{residually discrete}: for any non-identity element there exists a continuous homomorphism onto a discrete group which does not have it in the kernel.

\begin{corollary}
Any continuous $\mathbb{K}$-valued quasimorphism of $G$ is at a bounded distance from one that factors through $G/R_0$.
\end{corollary}

So in the study of $\mathbb{K}$-valued quasimorphisms up to bounded distance, one may reduce to totally disconnected abelian residually discrete groups. Of course one may take this even further by considering the quotient by the intersection of all open subgroups which are kernels of unbounded homomorphisms $G \to D(\mathbb{K})$, but we will stop here.

\subsection{$\mathbb{Q}_p$-valued quasimorphisms}

In light of Theorem \ref{qm split}, the only fields admitting non-trivial quasimorphisms are those satisfying $\chi(\mathbb{K}) = 0$ and $\chi(\mathfrak{r}) = p$. They all admit at least one non-trivial quasimorphism, namely any section $\iota : D(\mathbb{K}) \to \mathbb{K}$, by Proposition \ref{class def}. If $\mathbb{K}$ is complete, this condition is equivalent to $\mathbb{K}_0 \cong \mathbb{Q}_p$, where as usual $\mathbb{K}_0$ denotes the closure of the prime field. Thus we will focus on the case $\mathbb{K} = \mathbb{Q}_p$ in this subsection, for which non-trivial quasimorphisms can be classified in a suitable sense. \\

In this setting we can exploit the special structure of $\mathbb{Z}(p^\infty)$: all proper subgroups are of the form $\mathbb{Z}/p^k \mathbb{Z}$ for some $k \geq 0$, and all non-trivial quotients are isomorphic to $\mathbb{Z}(p^\infty)$ itself. Notice that in this case the defect $D(\iota)$ of the standard section is equal to $1$, and so the defect of a quasimorphism of the form $G \to \mathbb{Z}(p^\infty) \to \mathbb{Q}_p$ is also equal to $1$. Moreover, the induced norm of a non-zero element in $\mathbb{Z}(p^\infty)$ equals its order. Proposition \ref{class qm 0} is rephrased as follows:

\begin{proposition}
\label{class qm 0 Qp}

Let $G$ be a group, and let $f : G \to \mathbb{Q}_p$ be a continuous quasimorphism with $D(f) = 1$. Let $U := f^{-1}(\mathbb{Z}_p)$. Then $U$ is normal in $G$ and the following are equivalent:
\begin{enumerate}
\item $f$ is unbounded.
\item $\pi f : G \to \mathbb{Z}(p^\infty)$ is surjective.
\item $G/U \cong \mathbb{Z}(p^\infty)$.
\end{enumerate}
\end{proposition}

\begin{proof}
The only unbounded (and in fact, the only infinite) subgroup of $\mathbb{Z}(p^\infty)$ is $\mathbb{Z}(p^\infty)$ itself.
\end{proof}

We start by studying a stronger property than triviality for quasimorphisms, for which a cleaner characterization is possible. Recall that the function $d_H$ measures the distance of a function from the set of homomorphisms (Definition \ref{def dH}).

\begin{definition}
A continuous quasimorphism $f : G \to \mathbb{K}$ is \emph{strongly trivial} if $d_H(f) \leq D(f)$.
\end{definition}

So a quasimorphism $f : G \to \mathbb{Q}_p$ with defect $1$ is strongly trivial if and only if there exists a homomorphism $h : G \to \mathbb{Q}_p$ such that $(f - h)$ takes values in $\mathbb{Z}_p$.

\begin{proposition}
\label{class qm strong}

Let $G$ be a group and let $f : G \to \mathbb{Q}_p$ be a continuous unbounded quasimorphism with $D(f) = 1$. Let $U := f^{-1}(\mathbb{Z}_p)$. Then the following are equivalent:
\begin{enumerate}
\item $f$ is not strongly trivial.
\item $\pi f : G \to \mathbb{Z}(p^\infty)$ does not factor through a continuous homomorphism $G \to \mathbb{Q}_p$.
\item Any sequence $U = U_0 \geq U_1 \geq \cdots$ of open normal subgroups of $G$ such that $G/U_k \cong \mathbb{Z}(p^\infty)$ eventually stabilizes.
\end{enumerate}
\end{proposition}

\begin{proof}
$1. \Rightarrow 2.$ Suppose that $\pi f$ factors through a continuous homomorphism $h : G \to \mathbb{Q}_p$. Then $\pi f = \pi h$, which implies that $(f - h)$ takes values in $\mathbb{Z}_p$. So $f$ is strongly trivial. \\

$2. \Rightarrow 3.$ Assume that there exists a sequence $U = U_0 \geq U_1 \geq \cdots$ as in Item $3.$, which does not stabilize. Without loss of generality $U_k \neq U_{k+1}$ for all $k$. We need to show that $\pi f : G \to \mathbb{Z}(p^\infty)$ factors through $\mathbb{Q}_p$. By hypothesis $G/U_1 \cong \mathbb{Z}(p^\infty)$. Since $U_1 \neq U_0$ and $U_0$ has infinite index in $G$, the group $U_0/U_1$ is non-trivial and has infinite index in $G/U_1 \cong \mathbb{Z}(p^\infty)$, so it must correspond to $\mathbb{Z}/p^k\mathbb{Z} \leq \mathbb{Z}(p^\infty)$ for some $k \geq 1$. This allows to introduce intermediate open normal subgroups subgroups
$$U_1 = V^k \leq V^{k-1} \leq \cdots \leq V^0 = U_0$$
such that $G/V^i \cong \mathbb{Z}(p^\infty)$ and $[V^i : V^{i+1}] = p$. Repeating this process by induction we obtain a modified sequence
$$U = U_0 \geq U_1 \geq \cdots \geq U_k \geq \cdots$$
of open normal subgroups $U_k \leq G$ such that $[U_k : U_{k+1}] = p$ for all $k \geq 0$ and $G/U_k \cong \mathbb{Z}(p^\infty) \geq \mathbb{Z}/p^k\mathbb{Z} \cong U_0 / U_k$.

Let $Q_k := U_0 / U_k \cong \mathbb{Z}/p^k \mathbb{Z}$ as a discrete group, and for all $k \geq l$ let $q^k_l : Q_k \to Q_l$ be the natural quotient map. Then $\{Q_k \mid q^k_l \}$ is an inverse system  of finite groups, which, up to choosing inductively the isomorphisms $Q_k \cong \mathbb{Z}/p^k \mathbb{Z}$ in a compatible way, is isomorphic to the inverse system $\{ \mathbb{Z}/p^k \mathbb{Z} \mid q^k_l = (\text{reduction mod } p^l) \}$ defining $\mathbb{Z}_p$. By the universal property of the inverse limit, we obtain a continuous homomorphism $U_0 \to \mathbb{Z}_p$, such that $U_0 \to \mathbb{Z}_p \to \mathbb{Z}/p^k \mathbb{Z}$ has kernel $U_k$.

We are left to show that we can ``glue'' the continuous homomorphisms $U_0 \to \mathbb{Z}_p$ and $G/U_0 \to \mathbb{Z}(p^\infty)$ so as to obtain a continuous homomorphism $G \to \mathbb{Q}_p$. For all $k \geq 1$, let $U_0 \leq U_{-k} \leq G$ be the open subgroup such that $U_{-k}/U_0 \leq G/U_0$ corresponds to $\mathbb{Z}/ p^k \mathbb{Z} \leq \mathbb{Z}(p^\infty) \cong G/U_0$. Then the same construction we did for $U_0$ works for $U_{-k}$ as well, giving a continuous homomorphism $U_{-k} \to p^{-k}\mathbb{Z}_p$ which can be chosen so that it restricts to $U_{-(k-1)} \to p^{-(k-1)}\mathbb{Z}_p$. All these homomorphisms are compatible and continuous, the $U_{-k}$ are open in $G = \bigcup_{k \geq 0} U_{-k}$, and the $p^{-k} \mathbb{Z}_p$ are open in $\mathbb{Q}_p = \bigcup_{k \geq 0} p^{-k} \mathbb{Z}_p$. Thus we obtain a well-defined continuous homomorphism $G \to \mathbb{Q}_p$ mapping $U = U_0$ to $\mathbb{Z}_p$. By construction, $\pi f$ factors through this homomorphism. \\

$3. \Rightarrow 1.$ Suppose that $f$ is strongly trivial; we need to find a sequence as in Item $3.$ that does not stabilize. By hypothesis there exists a continuous homomorphism $h : G \to \mathbb{Q}_p$ such that $(f - h)$ takes values in $\mathbb{Z}_p$. Therefore $U = f^{-1}(\mathbb{Z}_p) = h^{-1}(\mathbb{Z}_p)$. Then $U_k := h^{-1}(p^k \mathbb{Z}_p)$ is an open normal subgroup of $G$, since $h$ is a continuous homomorphism. The surjective homomorphism
$$G \xrightarrow{h} \mathbb{Q}_p \to \mathbb{Q}_p / p^k \mathbb{Z}_p \cong \mathbb{Z}(p^\infty)$$
has kernel $U_k$. Since $h$ is unbounded, being close to $f$, it follows that $G/U_k \cong \mathbb{Z}(p^\infty)$ and that the $U_k$ are all distinct. Therefore $U = U_0 \geq U_1 \geq \cdots$ is a sequence as in Item $3.$, which does not stabilize.
\end{proof}

Item $3.$ readily implies the following restriction on the subgroup structure of a group admitting an unbounded quasimorphism which is not strongly trivial:

\begin{corollary}
\label{min pruf ker}

The group $G$ admits an unbounded $\mathbb{Q}_p$-valued quasimorphism which is not strongly trivial if and only it admits an open normal subgroup $U \leq G$ such that $G/U \cong \mathbb{Z}(p^\infty)$ and $U$ is minimal for this property.
\end{corollary}

Item $2.$ has the following immediate consequence:

\begin{example}
\label{EH2b Qp}

Any homomorphism $\mathbb{Q}_p \to \mathbb{Z}(p^\infty)$ factors through $\mathbb{Q}_p$, tautologically. Therefore every unbounded continuous quasimorphism $\mathbb{Q}_p \to \mathbb{Q}_p$ is strongly trivial.
\end{example}

This example allows to show that the standard section is the unique non-trivial quasimorphism (up to scalars and bounded functions) for $\mathbb{Z}(p^\infty)$:

\begin{example}
\label{EH2b Zpi}

The space $EH^2_{cb}(\mathbb{Z}(p^\infty), \mathbb{Q}_p)$ is one-dimensional: any quasimorphism is at a bounded distance from a scalar multiple of $\iota$. Indeed, let $f : \mathbb{Z}(p^\infty) \to \mathbb{Q}_p$ be any quasimorphism. Define $f \pi : \mathbb{Q}_p \to \mathbb{Z}(p^\infty) \to \mathbb{Q}_p$. This is a continuous quasimorphism of $\mathbb{Q}_p$, so since $EH^2_{cb}(\mathbb{Q}_p, \mathbb{Q}_p) = 0$, it is at a bounded distance from a continuous homomorphism $\mathbb{Q}_p \to \mathbb{Q}_p$. This must be of the form $x \mapsto \lambda x$ for some $\lambda \in \mathbb{Q}_p$. Therefore $|f \pi(x) - \lambda x|_p \leq C$ for some constant $C \geq 0$. In particular, since $\pi \iota \pi = \pi$, we have:
$$|f \pi(x) - (\lambda \cdot \iota) \pi(x)|_p = |f \pi( \iota \pi (x)) - \lambda \cdot \iota \pi(x)|_p \leq C.$$
Since $\pi$ is surjective, this shows that $f$ is at a bounded distance from $\lambda \cdot \iota$.
\end{example}

Corollary \ref{min pruf ker} imposes strong restrictions on the existence of subgroups of a group $G$ admitting an unbounded quasimorphism which is not strongly trivial. We apply this to subgroups of $\mathbb{Q}$:

\begin{example}
\label{EH2b Q}

Let $G$ be a subgroup of $\mathbb{Q}$ with the discrete topology. We show that all $\mathbb{Q}_p$-valued unbounded quasimorphisms of $G$ are strongly trivial, using Corollary \ref{min pruf ker}. Let $U \leq G$ be a subgroup such that $G/U \cong \mathbb{Z}(p^\infty)$; clearly $U \neq \{ 0 \}$. Up to precomposing by a scalar multiplication, which is an automorphism of $\mathbb{Q}$, we may assume that $1 \in U$, and so $\mathbb{Z} \leq U$. Therefore $G/U$ is a quotient of the torsion group $G/\mathbb{Z}$. If $a/b \in G \cap [0, 1)$ is a reduced fraction, and $p$ does not divide $b$, then $a/b \mod \mathbb{Z}$ has order coprime to $p$ in $G/\mathbb{Z}$; so $a/b$ is in the kernel of any homomorphism $G \to G / \mathbb{Z} \to \mathbb{Z}(p^\infty)$. In particular it is in $U$. Therefore $U$ contains all elements of $G$ which can be written as a fraction with denominator not divisible by $p$. It follows that the inclusion $G \cap \mathbb{Z}[1/p] \to G$ induces an isomorphism $(G \cap \mathbb{Z}[1/p]) / (U \cap \mathbb{Z}[1/p]) \to G/U$.

Therefore we may assume without loss of generality that $\mathbb{Z} \leq U \leq G \leq \mathbb{Z}[1/p]$. This gives injective homomorphisms $\mathbb{Z}(p^\infty) \cong G/U \to \mathbb{Z}[1/p]/U$. The last group is an infinite quotient of $\mathbb{Z}[1/p] / \mathbb{Z} \cong \mathbb{Z}(p^\infty)$, so it is itself isomorphic to $\mathbb{Z}(p^\infty)$. Since all proper subgroups of $\mathbb{Z}(p^\infty)$ are finite, the inclusion $G/U \to \mathbb{Z}[1/p]/U$ is really an isomorphism, so $G = \mathbb{Z}[1/p]$. At this point it is easy to find a proper subgroup of $U$ which also yields a quotient isomorphic to $\mathbb{Z}(p^\infty)$, for instance $p \mathbb{Z}$. We conclude that any such $U$ is not minimal, and thus all unbounded quasimorphisms of $G$ are strongly trivial.
\end{example}

The following theorem generalizes Proposition \ref{class qm strong}, and the classification of non-trivial quasimorphisms is a direct corollary:

\begin{theorem}
\label{class qm}
Let $G$ be a group, and let $f : G \to \mathbb{Q}_p$ be a continuous unbounded quasimorphism with $D(f) = 1$. Let $U := f^{-1}(\mathbb{Z}_p)$ and $k \geq 0$. Then the following are equivalent:
\begin{enumerate}
\item $d_H(f) > p^k$.
\item For any open subgroup $K \leq G$ of index at most $p^k$, the continuous homomorphism $\pi f|_K : K \to \mathbb{Z}(p^\infty)$ does not factor through a continuous homomorphism $K \to \mathbb{Q}_p$.
\item For any open subgroup $K \leq G$ of index at most $p^k$, any sequence $U \cap K =: U_0 \geq U_1 \geq \cdots$ of open normal subgroups of $K$ such that $K/U_k \cong \mathbb{Z}(p^\infty)$, eventually stabilizes.
\end{enumerate}
Moreover, in conditions $2.$ and $3.$, ``at most $p^k$'' may be replaced by ``dividing $p^k$'', and $K$ may be assumed to be normal.
\end{theorem}

\begin{remark}
Proposition \ref{class qm strong} is the case $k = 0$.
\end{remark}

\begin{proof}
Let $K$ be a finite-index subgroup of $G$. Since $\pi f$ is surjective, $\pi f(K)$ is a finite-index subgroup of $\mathbb{Z}(p^\infty)$, so $\pi f|_K$ is also surjective. Similarly, $UK/U$ is a finite-index subgroup of $G/U$, so $\mathbb{Z}(p^\infty) \cong G/U = UK/U \cong K/(U \cap K)$. Since $D(f) = 1$, it follows that $D(f|_K) = 1$ as well.

We claim that $d_H(f) \leq p^k$ if and only if there exists an open subgroup $K \leq G$ of index at most $p^k$ such that $f|_K$ is strongly trivial; that is, such that $d_H(f|_K) \leq 1$. Since $U \cap K = f|_K^{-1}(\mathbb{Z}_p)$, the theorem now follows from Proposition \ref{class qm strong} and the previous argument. More precisely, we will show that if $d_H(f) \leq p^k$, then there exists an open \emph{normal} subgroup $K \leq G$ of index \emph{dividing} $p^k$ such that $f|_K$ is strongly trivial; and that if there exists an open subgroup $K \leq G$ of index at most $p^k$ such that $f|_K$ is strongly trivial, then $d_H(f) \leq p^k$. This implies the last sentence of the theorem as well. \\

We prove the claim. Suppose that $d_H(f) \leq p^k$. Then there exists a continuous homomorphism $h : G \to \mathbb{Q}_p$ such that $b := (f - h)$ takes values in $p^{-k} \mathbb{Z}_p$. Since $h$ is a homomorphism, $D(b) = D(f - h) = D(f) = 1$. So $\pi b : G \to \mathbb{Z}(p^\infty)$ is a continuous homomorphism taking values in $\mathbb{Z}/p^k \mathbb{Z}$. Let $K := b^{-1}(\mathbb{Z}_p) = \ker(\pi b)$. This is an open normal subgroup of $G$, and $[G : K] = |G/K|$ divides $|\mathbb{Z}/p^k\mathbb{Z}| = p^k$. Since $\pi b|_K = 0$ by definition of $K$, we have $\pi f|_K = \pi h|_K$. So $(f|_K - h|_K)$ takes values in $\mathbb{Z}_p$, which implies that $d_H(f|_K) \leq 1$.

Suppose that there exists $K \leq G$ open of index at most $p^k$ such that $f|_K$ is strongly trivial. Then there exists a continuous homomorphism $h : K \to \mathbb{Q}_p$ such that $b := (f|_K - h)$ takes values in $\mathbb{Z}_p$. Since $\mathbb{Q}_p$ is divisible as a discrete group, Corollary \ref{div ext} provides a continuous homomorphism $h^\# : G \to \mathbb{Q}_p$ extending $h$. Let $b^\# := (f - h^\#)$. This is continuous, and since $h^\#$ is a homomorphism $D(b^\#) = D(f - h^\#) = D(f) = 1$. So $\pi b^\# : G \to \mathbb{Z}(p^\infty)$ is a homomorphism. But $\pi b^\#|_K = \pi f|_K - \pi h = 0$, so $K \leq \ker(\pi b^\#)$. Since $[G : K] \leq p^k$, the image of $\pi b^\#$ has order at most $p^k$, so it must be contained in $\mathbb{Z}/p^k\mathbb{Z} \leq \mathbb{Z}(p^\infty)$. It follows that $(f - h^\#)$ takes values in $p^{-k}\mathbb{Z}_p$, and so $d_H(f) \leq p^k$.
\end{proof}

\begin{corollary}
\label{class qm cor}

Let $G$ be a group, and let $f : G \to \mathbb{Q}_p$ be a continuous unbounded quasimorphism with $D(f) = 1$. Let $U := f^{-1}(\mathbb{Z}_p)$. Then the following are equivalent:
\begin{enumerate}
\item $f$ is non-trivial.
\item For any open finite-index subgroup $K \leq G$, the continuous homomorphism $\pi f|_K : K \to \mathbb{Z}(p^\infty)$ does not factor through a continuous homomorphism $K \to \mathbb{Q}_p$.
\item For any open finite-index subgroup $K \leq G$, any sequence $U \cap K =: U_0 \geq U_1 \geq \cdots$ of open normal subgroups of $K$ such that $K/U_k \cong \mathbb{Z}(p^\infty)$, eventually stabilizes.
\end{enumerate}
Moreover, in conditions $2.$ and $3.$, ``finite-index'' may be replaced by ``index a power of $p$'', and $K$ may be assumed to be normal.
\end{corollary}

\begin{proof}
This follows from Theorem \ref{class qm}, noticing that $d_H(f) > p^k$ for all $k \geq 0$ if and only if $f$ is non-trivial.
\end{proof}

\subsection{Additivity}

In this subsection we prove that the functor $EH^2_{cb}$ is additive, in a suitable sense. This will allow to compute $EH^2_{cb}(G, \mathbb{Q}_p)$ for a divisible group $G$, and it will determine the exact bounded cohomology over a finite extension of $\mathbb{Q}_p$ in terms of that over $\mathbb{Q}_p$. We start by looking at the first coordinate.

\begin{proposition}
Let $G_1, G_2$ be groups, $\mathbb{K}$ a (not necessarily non-Archimedean) valued field. Then the restriction of quasimorphisms induces a vector-space isomorphism
$$EH^2_{cb}(G_1 \times G_2, \mathbb{K}) \cong EH^2_{cb}(G_1, \mathbb{K}) \times EH^2_{cb}(G_2, \mathbb{K}).$$
\end{proposition}

\begin{proof}
Denote $G := G_1 \times G_2$. The map is clearly well-defined and linear at the level of quasimorphisms. Given a trivial quasimorphism $f = (h + b) : G \to \mathbb{K}$, where $h$ is a continuous homomorphism and $b$ is a continuous bounded function, $f|_{G_i} = (h|_{G_i} + b|_{G_i})$ is also a trivial quasimorphism. So the map is well-defined and linear at the level of exact bounded cohomology, too.

Given quasimorphisms $f_i : G_i \to \mathbb{K}$, they represent the same class as $f_i - f_i(1)$, so we may assume that $f_i(1) = 0$. Define
$$f : G \to \mathbb{K} : (g_1, g_2) \mapsto f_1(g_1) + f_2(g_2).$$
Then $f$ is a continuous quasimorphism of defect $D(f) \leq \max\{D(f_1), D(f_2)\}$. Moreover, $f(g_1, 1) = f_1(g_1) + f_2(1) = f_1(g_1)$, since $f_i(1) = 0$. Similarly $f(1, g_2) = f_2(g_2)$ so $f|_{G_i} = f_i$. This proves that the map is surjective.

It remains to show that the map is injective. Let $f : G \to \mathbb{K}$ be a quasimorphism whose class is in the kernel. It represents the same class as $f - f(1, 1)$, so we may assume that $f(1, 1) = 0$. Define $\tilde{f}(g_1, g_2) := f(g_1, 1) + f(1, g_2)$, which is continuous. Then
$$\| f - \tilde{f} \|_\infty = \sup |f(g_1, g_2) - f(g_1, 1) - f(1, g_2)|_\mathbb{K} \leq D(f),$$
so $f$ and $\tilde{f}$ represent the same class. Let $f_i := \tilde{f}|_{G_i}$. Notice that $\tilde{f}(g_1, 1) = f(g_1, 1)$ and $\tilde{f}(1, g_2) = f(1, g_2)$, because $f(1, 1) = 0$; and so $\tilde{f}(g_1, g_2) = f_1(g_1) + f_2(g_2)$. By hypothesis $f|_{G_i}$ is trivial, and so $f_i$ is too: there exist continuous homomorphisms $h_i$ and continuous bounded functions $b_i$ such that $f_i = (h_i + b_i)$. Define the continuous homomorphism $h : G \to \mathbb{K} : (g_1, g_2) \mapsto h_1(g_1) + h_2(g_2)$ and the continuous bounded function $b : G \to \mathbb{K}$ similarly. Then $\tilde{f}(g_1, g_2) = f_1(g_1) + f_2(g_2) = (h + b)(g_1, g_2)$. Therefore $\tilde{f}$ is trivial, and so $f$ is too.
\end{proof}

Assuming that the field is non-Archimedean and complete, we can also consider infinite direct sums of discrete groups. This requires a technical lemma:

\begin{lemma}
\label{qm add}

Let $G$ be a group, $\mathbb{K}$ a non-Archimedean valued field. For any trivial continuous quasimorphism $f : G \to \mathbb{K}$ we have $D(f) \leq d_H(f)$. If moreover $\mathbb{K}$ is complete, then there exists a constant $C \geq 0$ such that any trivial continuous quasimorphism $f : G \to \mathbb{K}$ satisfies $d_H(f) \leq C D(f)$.
\end{lemma}

\begin{remark}
This lemma is true in the real case, although the first inequality only holds with a factor of $3$, without the ultrametric inequality.
\end{remark}

\begin{proof}
The first inequality is a straightforward computation.

Since $\mathbb{K}$ is complete, $\overline{C}^1_b(G, \mathbb{K})$ and $\overline{C}^2_b(G, \mathbb{K})$ are Banach spaces. The map $\delta^1 : \overline{C}^1_b(G, \mathbb{K}) \to \overline{C}^2_b(G, \mathbb{K})$ is linear and bounded, so it is open by the Open Mapping Theorem. This map induces a vector space isomorphism $\overline{C}^1_b(G, \mathbb{K})/\overline{Z}^1_b(G, \mathbb{K}) \to \overline{B}^2_b(G, \mathbb{K})$, which is also $C$-bi-Lipschitz for some constant $C > 0$, since $\delta^1$ is open. Now let $f$ be a trivial continuous quasimorphism. Then $f = (h + b)$ for a continuous homomorphism $h$ and a bounded continuous function $b$.
$$ \inf\limits_{g \in \overline{Z}^1_b(G, \mathbb{K})} \| b + g \|_\infty = \| b + \overline{Z}^1_b(G, \mathbb{K}) \|_\infty \leq C \| \delta^1 b \|_\infty = C \| \delta^1(b + h) \|_\infty = C D(f).$$
So there exists a bounded continuous homomorphism $g : G \to \mathbb{K}$ such that $\| b + g \|_\infty \leq CD(f)$. We conclude by rewriting $f$ as $f = (h - g) + (b + g)$.
\end{proof}

Notice that this lemma implies that, in theory, given a quasimorphism $f$ of defect $1$, one would only need to check $d_H(f) > p^k$ for a finite number of integers (depending on the group) in order to conclude that $f$ is non-trivial. This sheds some new light on Theorem \ref{class qm} and Corollary \ref{class qm cor}.

\begin{proposition}
\label{EH add}

Let $(G_i)_{i \in I}$ be discrete groups, and let $\mathbb{K}$ be a non-Archimedean complete valued field. Then the restriction of quasimorphisms induces an injective linear map
$$EH^2_b(\bigoplus_i G_i, \mathbb{K}) \to \prod_i EH^2_b(G_i, \mathbb{K})$$ 
whose image is the subspace of sequences admitting representatives of uniformly bounded defect.
\end{proposition}

\begin{remark}
This weaker version of additivity will appear in a topological context in the next subsection: see the statement of Theorem \ref{main_top}.
\end{remark}

\begin{proof}
Let $G := \bigoplus_i G_i$. We denote elements of $G$ by $\underline{g} = (g_i)_{i \in I}$, and the identity of $G$ by $\underline{1}$. The map is clearly well-defined and linear at the level of quasimorphisms. Given a trivial quasimorphism $f = (h + b) : G \to \mathbb{K}$, where $h$ is a homomorphism and $b$ is a bounded function, $f|_{G_i} = (h|_{G_i} + b|_{G_i})$ is also a trivial quasimorphism. So the map is well-defined and linear at the level of exact bounded cohomology, too.

Let $f : G \to \mathbb{K}$ be a quasimorphism. Since $D(f|_{G_i}) \leq D(f)$ for all $i$, the image of $f$ is a sequence with bounded defect. Given quasimorphisms $f_i : G_i \to \mathbb{K}$ of uniformly bounded defect, they represent the same class as $f_i - f_i(1)$. Since $f_i(1)$ is bounded by $D(f_i)$, this allows to assume that $f_i(1) = 0$ and still that $D(f_i)$ is uniformly bounded. The map
$$f : G \to \mathbb{K} : \underline{g} \mapsto \sum f_i(g_i)$$
is well-defined, because $f_i(g_i) = f_i(1) = 0$ for all but finitely many coordinates. By the ultrametric inequality $f$ is a quasimorphism of defect $D(f) \leq \sup D(f_i) < \infty$. Finally, $f|_{G_i} = f_i$ since for all $j \neq i$ the $j$-th coordinate of an element of $G_i$ is $1$, and $f_j(1) = 0$. This proves that the map is surjective.

It remains to show that the map is injective. Let $f : G \to \mathbb{K}$ be a quasimorphism whose class is in the kernel. Up to replacing it by $f - f(\underline{1})$, which represents the same class, we may assume that $f(\underline{1}) = 0$. Given an element $\underline{g} \in G$, denote by $\underline{g}^i$ the element having the same $i$-th coordinate as $\underline{g}$ and the identity in all other coordinates. Define $\tilde{f}(\underline{g}) := \sum f(\underline{g}^i)$. This is well-defined because $f(\underline{g}^i) = f(\underline{1}) = 0$ for all but finitely many coordinates. Moreover
$$\| f - \tilde{f} \|_\infty = \sup \left|f(\underline{g}) - \sum f(\underline{g}^i) \right| \leq D(f),$$
where we used the ultrametric inequality; so $f$ and $\tilde{f}$ represent the same class. Let $f_i := \tilde{f}|_{G_i} : G_i \to \mathbb{K}$. Notice that $\tilde{f}(\underline{g}^i) = f(\underline{g}^i)$ because $f(\underline{1}) = 0$; and so $\tilde{f}(\underline{g}) = \sum f_i(g_i)$. By hypothesis there exist homomorphisms $h_i$ and bounded functions $b_i$ such that $f_i = (h_i + b_i)$. By Lemma \ref{qm add}, these can be chosen so that $\| b_i \|_\infty \leq C D(f_i) \leq C D(\tilde{f})$, for some constant $C$. We have $h_i(1) = 0$ since $h_i$ is a homomorphism, so $h : \underline{g} \mapsto \sum h_i(g_i)$ is a well-defined homomorphism. Since $f_i(1) = 0$ we also have $b_i(1) = 0$, and so $b : \underline{g} \mapsto \sum b_i(g_i)$ is well-defined and bounded by $C D(\tilde{f}) < \infty$. Finally, $\tilde{f} = (h + b)$; therefore $\tilde{f}$ is trivial, and so $f$ is too.
\end{proof}

Note that the proof relies heavily on the ultrametric inequality, which is used in two passages: to show that an infinite sum of quasimorphisms of uniformly bounded defect is a quasimorphism; and to use the defect inequality with arbitrarily large finite products. Both of these do not happen in the real case. \\

As an application we determine the exact second bounded cohomology of divisible groups. Recall that $EH^2_b(\mathbb{Z}(\ell^\infty), \mathbb{Q}_p) = 0$ for all $\ell \neq p$ (Example \ref{EH2b Zli}), that $EH^2_b(\mathbb{Z}(p^\infty), \mathbb{Q}_p) \cong \mathbb{Q}_p$ (Example \ref{EH2b Zpi}), and that $EH^2_b(\mathbb{Q}, \mathbb{Q}_p) = 0$ (Example \ref{EH2b Q}).

\begin{example}
Let $G$ be a discrete divisible group. By the Classification Theorem from Subsection \ref{ss_div} there exist index sets $I_0$ and $I_p$ for all primes $p$ such that $G \cong \mathbb{Q}^{I_0} \bigoplus \left( \bigoplus\limits_p \mathbb{Z}(p^\infty)^{I_p} \right)$. By the previous examples and by Proposition \ref{EH add}: 
$$EH^2_{cb}(G, \mathbb{Q}_p) = EH^2_{cb}(\mathbb{Z}(p^\infty)^{I_p}, \mathbb{Q}_p) \cong \{ (x_i) \in \prod\limits_{I_p} \mathbb{Q}_p \mid \sup_i |x_i|_p < \infty \}.$$
\end{example}

We now look at additivity in terms of the target space.

\begin{proposition}
\label{qm Kn}

Let $G$ be a group, $\mathbb{K}$ a (not necessarily non-Archimedean) valued field, and equip $\mathbb{K}^n$ with the $\ell^\infty$-norm. Then the projection maps induce a vector space isomorphism
$$EH^2_{cb}(G, \mathbb{K}^n) \xrightarrow{\cong} EH^2_{cb}(G, \mathbb{K})^n.$$
\end{proposition}

\begin{proof}
Given a continuous quasimorphism $f : G \to \mathbb{K}^n$, let $f_i : G \to \mathbb{K}$ denote the composition with the $i$-th projection, which is a continuous quasimorphism of defect $D(f_i) \leq D(f)$. This gives a well-defined linear map as in the statement at the level of quasimorphisms. If $f = (h + b)$, where $h$ is a continuous homomorphism and $b$ is a continuous bounded function, then $f_i = (h_i + b_i)$ is also a trivial quasimorphism for all $i$. So the map is well-defined and linear at the level of exact bounded cohomology, too.

Given a continuous quasimorphisms $f_i : G \to \mathbb{K}$, the map $f : G \to \mathbb{K}^n : g \mapsto (f_i(g))_{i = 1}^n$ is a continuous quasimorphism of defect $D(f) \leq \max D(f_i)$. This proves that the map is surjective. Finally, suppose that $f : G \to \mathbb{K}$ is a quasimorphism such that $f_i = h_i + b_i$ for all $i$, where $h_i$ is a continuous homomorphism and $b_i$ is a continuous bounded function. Then defining the continuous homomorphism $h : G \to \mathbb{K}^n : g \mapsto (h_i(g))_{i = 1}^n$ and the continuous bounded function $b$ similarly, we obtain $f = (h + b)$. Therefore the map is injective.
\end{proof}

As an application, we can determine the exact second bounded cohomology over finite extensions of $\mathbb{Q}_p$ in terms of that over $\mathbb{Q}_p$:

\begin{corollary}
\label{qm ext Qp}

Let $\mathbb{K}$ be a finite extension of $\mathbb{Q}_p$ of degree $n$. Then 
$$EH^2_{cb}(G, \mathbb{K}) \cong EH^2_{cb}(G, \mathbb{Q}_p)^n.$$
In particular, $EH^2_{cb}(G, \mathbb{K}) = 0$ if and only if $EH^2_{cb}(G, \mathbb{Q}_p) = 0$.
\end{corollary}

\begin{proof}
This follows from the previous proposition, since $\mathbb{K}$ is bi-Lipschitz equivalent to $\mathbb{Q}_p^n$ as a normed $\mathbb{Q}_p$-vector space by Theorem \ref{norm fin dim}.
\end{proof}

This corollary, combined with Theorem \ref{class qm}, characterizes groups admitting a non-trivial quasimorphism taking values in a non-Archimedean local field of characteristic $0$. The non-Archimedean local fields of characteristic $p$ are contained in Theorem \ref{qm split}. \\

It would be of interest to know whether $EH^2_{cb}(G, \mathbb{C}_p)$ admits a description in terms of $EH^2_{cb}(G, \mathbb{Q}_p)$. As we did for finite extensions, we may regard any complete valued field $\mathbb{K}$ with $\mathbb{K}_0 \cong \mathbb{Q}_p$ as a $\mathbb{Q}_p$-Banach space. There is a Classificiation Theorem for these: namely for any $\mathbb{Q}_p$-Banach space there exists a set $X$ such that the space is bi-Lipschitz equivalent to
$$c_0(X) := \{ f : X \to \mathbb{Q}_p \mid \# \{x \in X : |f(x)|_p \geq \varepsilon \} < \infty \text{ for all } \varepsilon > 0 \}$$
equipped with the supremum norm (this is a combination of \cite[Theorem 2.1.1]{NFA} and \cite[Theorem 2.5.4]{NFA}). Let us try to go over the proof of Proposition \ref{qm Kn}, with $\mathbb{K}^n$ replaced by $c_0(X)$. The maps $f_i$ are replaced by the maps $f_x = \delta_x \circ f$, where $\delta_x : c_0(X) \to \mathbb{Q}_p$ is the evaluation map at $x \in X$. We obtain a well-defined linear map
$$EH^2_{cb}(G, c_0(X)) \to \prod_X EH^2_{cb}(G, \mathbb{Q}_p)$$
whose image consists of sequences $(f_x)_{x \in X}$ of quasimorphisms satisfying $\# \{ x \in X \mid |f_x(g)|_p \geq \varepsilon \} < \infty$ for all $g \in G$ and all $\varepsilon > 0$. 

The problem is when trying to prove injectivity. Lemma \ref{qm add} allows to approximate each $f_x$ uniformly by a homomorphism $h_x$, which gives a natural candidate for a homomorphism $h$ approximating $f$. However, there is no reason a priori as to why $\{x \in X \mid |h_x(g)|_p \geq \varepsilon \}$ should be finite for all $\varepsilon > 0$, and so $h$ does not necessarily take values in $c_0(X)$. Even if it did, $h$ would be continuous with respect to the weak topology relative to the evaluation maps $\delta_x$, but not necessarily with respect to the norm topology on $c_0(X)$. If $G$ is not discrete, one would need the $h_x$ to be equicontinuous, and so by Arzela-Ascoli for bounded set $\{ b_x := f_x - h_x \mid x \in X \}$ to be precompact in the compact-open topology of $C_b(G, \mathbb{Q}_p)$. All of this suggests that the following question should be much harder than Corollary \ref{qm ext Qp}.

\begin{question}
Does $EH^2_{cb}(G, c_0(X))$ admit a description in terms of $EH^2_{cb}(G, \mathbb{Q}_p)$?
\end{question}

\subsection{Higher-dimensional quasicocycles}

The techniques introduced to study quasimorphisms carry over to higher degrees, even though the results we can prove are not quite as strong. As usual, $\mathbb{K}$ is a non-Archimedean valued field, with defect group $D(\mathbb{K})$. Denote by $\pi : \mathbb{K} \to D(\mathbb{K})$ the canonical projection, and by $\iota : D(\mathbb{K}) \to \mathbb{K}$ a section. We start with a generalization of Lemma \ref{qm key}:

\begin{lemma}
\label{qz key}

Let $n \geq 1$ and let $f : G^n \to \mathbb{K}$ be a continuous $n$-quasicocycle, with $\| \delta^n f \|_\infty \leq 1$. Then $\pi f : G^n \to D(\mathbb{K})$ is a continuous $n$-cocycle. Conversely, let $c : G^n \to D(\mathbb{K})$ be a continuous $n$-cocycle. Then $\iota c : G^n \to \mathbb{K}$ is a continuous $n$-quasicocycle and $\| \delta^n (\iota c) \|_\infty \leq 1$.

Moreover, $\| \iota \pi f - f \|_\infty \leq 1$ and $\pi \iota c = c$.
\end{lemma}

\begin{proof}
The proof is the same as in Lemma \ref{qm key}; it carries over to higher dimensions with the same bound of $1$ because of the ultrametric inequality.
\end{proof}

This allows to generalize Theorem \ref{qm split}:

\begin{theorem}
\label{qz split}

Let $\mathbb{K}$ be a non-Archimedean valued field, and suppose that $\chi(\mathbb{K}) = \chi(\mathfrak{r})$. Then for any group $G$ and all $n \geq 1$, the comparison map $c^n : H^n_{cb}(G, \mathbb{K}) \to H^n_c(G, \mathbb{K})$ is injective.
\end{theorem}

\begin{proof}
For $n = 1$ this is true in any case, so let $n \geq 2$. Recall that $c^n$ is injective if and only if any $(n-1)$-quasicocycle is at a bounded distance from a true $(n-1)$-cocycle. So let $f : G^{n-1} \to \mathbb{K}$ be a quasicocycle, and suppose up to scalar that $\| \delta^{n-1} f \|_\infty \leq 1$. By Lemma \ref{qz key}, the map $\pi f : G^{n-1} \to D(\mathbb{K})$ is a cocycle, meaning that $\delta^{n-1} (\pi f) = 0$. By Proposition \ref{class def}, there exists a homomorphic section $\iota : D(\mathbb{K}) \to \mathbb{K}$, and so $\delta^{n-1}(\iota \pi f) = \iota \delta^{n-1}(\pi f) = 0$. So $\iota \pi f$ is a cocycle, and again by Lemma \ref{qz key} it is at distance at most $1$ from $f$.
\end{proof}

For general non-Archimedean valued fields, we can still generalize a consequence of Proposition \ref{class qm 0} in the following statement:

\begin{theorem}
\label{surj inj}

Let $\mathbb{K}$ be a non-Archimedean valued field. Let $n \geq 2$ and suppose that $G$ is such that the comparison map $c^{n-1} : H^{n-1}_{cb}(G, D(\mathbb{K})) \to H^{n-1}_c(G, D(\mathbb{K}))$ is surjective. Then $c^n : H^n_{cb}(G, \mathbb{K}) \to H^n_c(G, \mathbb{K})$ is injective.
\end{theorem}

\begin{remark}
The continuous (bounded) cohomology with coefficients in $D(\mathbb{K})$, and the corresponding comparison map, are defined as usual: here $D(\mathbb{K})$ is equipped with the discrete topology and the norm induced by $\mathbb{K}$. The comparison map $c^{n-1} : H^{n-1}_{cb}(G, D(\mathbb{K})) \to H^{n-1}_c(G, D(\mathbb{K}))$ is then surjective if and only if every $(n - 1)$-cocycle $G^{n-1} \to D(\mathbb{K})$ is at a bounded distance from a coboundary.
\end{remark}

\begin{proof}
We need to show that any $(n - 1)$-quasicocycle $G^{n-1} \to \mathbb{K}$ is at a bounded distance from a true $(n - 1)$-cocycle. By Lemma \ref{qz key}, it suffices to consider quasicocycles of the form $\iota c$, where $c : G^{n-1} \to D(\mathbb{K})$ is a cocycle. Since the comparison map $c^{n-1} : H^{n-1}_{cb}(G, D(\mathbb{K})) \to H^{n-1}_c(G, D(\mathbb{K}))$ is surjective, there exists a cochain $b : G^{n-2} \to D(\mathbb{K})$ such that $c$ is at a bounded distance from $\delta^{n-2}b$. Then $\iota c$ is at a bounded distance from $\iota \delta^{n-2}b$, which in turn is at a bounded distance from the true $(n - 1)$-cocycle $\delta^{n-2} \iota b : G^{n-2} \to \mathbb{K}$.
\end{proof}

For $n = 2$, we recover the comment after Proposition \ref{class qm 0}; namely, if $G$ does not admit unbounded subgroups of $D(\mathbb{K})$ as discrete quotients, then $c^2$ is injective. Indeed, if all continuous homomorphisms $G \to D(\mathbb{K})$ are bounded, then the comparison map $c^1 : H^1_{cb}(G, D(\mathbb{K})) \to H^1_c(G, D(\mathbb{K}))$ is surjective, and so the statement follows from Theorem \ref{surj inj}. \\

For discrete groups, Theorem \ref{surj inj} may be combined with Proposition \ref{surj_comp 2} to obtain the following Corollary:

\begin{corollary}
\label{inj Fn}
Let $\mathbb{K}$ be a non-Archimedean valued field. Let $G$ be a discrete group such that $H_{n-1}(G, \mathbb{Z})$ is a finitely generated abelian group. Then the comparison map $c^n : H^n_b(G, \mathbb{K}) \to H^n(G, \mathbb{K})$ is injective.
\end{corollary}

\begin{proof}
By Theorem \ref{qz split} we may assume that $\chi(\mathbb{K}) = 0$ and $\chi(\mathfrak{r}) = p > 0$. Proposition \ref{surj_comp 2} applies with $R = \mathbb{Z}$ and $M = D(\mathbb{K})$, since $D(\mathbb{K})$ is divisible by Proposition \ref{class def}. Therefore the comparison map $c^{n-1} : H^{n-1}_b(G, D(\mathbb{K})) \to H^{n-1}(G, D(\mathbb{K}))$ is surjective. We conclude by Theorem \ref{surj inj}.
\end{proof}

This applies in particular to groups of type $F_n$:

\begin{corollary}
\label{inj Fn 2}

Let $G$ be a group of type $F_n$. Then the comparison map $c^{n+1} : H^{n+1}_b(G, \mathbb{K}) \to H^{n+1}(G, \mathbb{K})$ is injective.
\end{corollary}

Combined with Corollary \ref{Fn surj}, this implies:

\begin{corollary}
\label{iso Fn}

Let $X$ be an aspherical CW-complex of type $F_\infty$. Then $c^n : H^n_b(\pi_1(X), \mathbb{K}) \to H^n(\pi_1(X), \mathbb{K})$ is an isomorphism for all $n \geq 1$.
\end{corollary}

\begin{remark}
In the next section we will prove a topological version of this result: see Corollary \ref{cell}.
\end{remark}

\begin{proof}
The asphericity condition implies that $H_n(\pi_1(X), \mathbb{Z}) \cong H_n(X, \mathbb{Z})$, which is finitely generated by the finiteness condition on cells. The same holds with $\mathbb{F}_p$-coefficients. Then Corollary \ref{Fn surj} implies that $c^n$ is surjective, and Corollary \ref{inj Fn 2} implies that $c^{n+1}$ is injective.
\end{proof}

These last corollaries show once again a big contrast with the real setting. We focus on two specific group-theoretic aspects, but there are of course many more. We will look at topological aspects in the next section. \\

By a Theorem of Huber \cite[Theorem 2.14]{huber}, given discrete groups $G, H$ and a surjective homomorphism $\varphi : G \to H$, the induced map $H^2_b(\varphi) : H^2_b(H, \mathbb{R}) \to H^2_b(G, \mathbb{R})$ is an injective homomorphism (in fact, an isometric embedding with respect to the Gromov norm). In particular, if $H$ is any finitely generated group, then $H^2_b(H, \mathbb{R})$ embeds into $H^2_b(G, \mathbb{R})$, where $G$ is a free group of large enough rank. This fails when $\mathbb{R}$ is replaced by a non-Archimedean valued field $\mathbb{K}$. For instance, we can let $H$ be the fundamental group of a closed surface of genus $g \geq 1$, and then Corollary \ref{iso Fn} implies that $H^2_b(H, \mathbb{K}) \neq 0$ while $H^2_b(G, \mathbb{K}) = 0$ for any free group $G$ of finite rank. \\

A very mysterious quantity related to bounded cohomology of discrete groups with real coefficients is the \emph{(real) bounded cohomological dimension}. This is defined as $bcd_\mathbb{R}(G) := \inf \{ n \geq 0 \mid H^n_b(G, \mathbb{R}) \neq 0 \} \in [0, \infty]$. This quantity is explored in the paper \cite{bcd}, where the author points out that there are no known examples of groups whose real bounded cohomological dimension is positive and finite. On the other hand, if we similarly define $bcd_\mathbb{K}$ for a non-Archimedean valued field $\mathbb{K}$, then there are groups that attain all possible bounded cohomological dimensions. By Corollary \ref{iso Fn}, if $G$ is the fundamental group of a compact oriented aspherical $n$-manifold, then $bcd_\mathbb{K}(G) = n$. An example attaining infinite bounded cohomological dimension is Thompson's group $F$. Indeed by \cite[Theorem 7.1]{thompson}, for all $n \geq 1$ we have $H_n(F, \mathbb{Z}) \cong \mathbb{Z}^2$, which is finitely generated. By the Universal Coefficient Theorem in Homology $H_n(F, \mathbb{Q})$ and $H_n(F, \mathbb{F}_p)$ are finite-dimensional (we are using that $\mathbb{Z}^2$ is free for the case of $\mathbb{F}_p$). So by Proposition \ref{surj_comp} and Corollary \ref{inj Fn} the comparison map is always an isomorphism, which implies that $bcd_\mathbb{K}(F) = \infty$.

\pagebreak

\section{Bounded cohomology of topological spaces}
\label{s_top}

We have seen in Corollary \ref{iso Fn} that the comparison map is an isomorphism for fundamental groups of aspherical CW-complexes of type $F_\infty$. This section justifies this result by taking a topological approach. Our goal is to define bounded cohomology for pairs of topological spaces and prove the following:

\begin{theorem}
\label{main_top}

Bounded cohomology with coefficients in a non-Archimedean valued field $\mathbb{K}$ satisfies the Eilenberg--Steenrod axioms for (unreduced) cohomology, except additivity. This is replaced by an injective homomorphism
$$\prod\limits_{i \in I} H^n_b(j_i) : H^n_b \left( \bigsqcup\limits_{i \in I} X_i , \bigsqcup\limits_{i \in I} A_i; \mathbb{K} \right) \to \prod\limits_{i \in I} H^n_b(X_i, A_i; \mathbb{K})$$
whose image is the subspace of sequences admitting representatives with uniformly bounded norm.
\end{theorem}

The axioms of dimension, homotopy invariance and exactness are all satisfied by bounded cohomology with real coefficients, but excision is not. This is what makes real bounded cohomology of topological spaces so hard to compute, since algorithmic methods to compute cohomology of spaces generally rely on the Mayer--Vietoris sequence. However, with the presence of an ultrametric inequality, the technical difficulties that make excision fail in the real case disappear. This readily implies the existence of a Mayer--Vietoris sequence \cite[p. 203]{Hat}:

\begin{corollary}
\label{MVseq}

Let $X$ be the union of the interiors of the subspaces $A, B$. Then there is a Mayer--Vietoris sequence:
$$\cdots \to H^n_b(X; \mathbb{K}) \to H^n_b(A; \mathbb{K}) \oplus H^n_b(B; \mathbb{K}) \to H^n_b(A \cap B; \mathbb{K}) \to H^{n+1}_b(X; \mathbb{K}) \to \cdots.$$
\end{corollary}

The modification of the additivity axiom is significant only when the union is infinite, so for most practical purposes bounded cohomology with non-Archimedean coefficients behaves like a standard cohomology theory. \\

Here is a far-reaching example of this. When $X$ is a CW-complex, any cohomology theory satisfying the Eilenberg--Steenrod axioms with a given coefficient group is naturally isomorphic to cellular cohomology with the same coefficient group. Given the modified additivity axiom, bounded cohomology with coefficients in $\mathbb{K}$ is not a bona fide cohomology theory. However, since the additivity axiom holds with respect to finite disjoint unions, the proof carries over whenever $X$ is of type $F_\infty$. Indeed, going over the proof of the aforementioned natural isomorphism (see \cite[Section 10]{Bred} for the proof of the corresponding statement in homology, which can be completely dualized), we see that additivity comes into play only when expressing the $n$-skeleton of $X$ as a quotient of the disjoint union of its $n$-cells. We thus have:

\begin{corollary}
\label{cell}

Let $X$ be a CW-complex of type $F_\infty$. Then $H^n_b(X; \mathbb{K})$ is naturally isomorphic to the cellular cohomology of $X$ with coefficient group $\mathbb{K}$.
\end{corollary}

This implies that the comparison map is an isomorphism for such spaces, and thus this result can be seen as a topological version of Corollary \ref{iso Fn}. \\

One may want to apply such strong results on bounded cohomology of topological spaces to a suitable non-Archimedean notion of simplcial volume. We will discuss this in the last subsection, showing that the most natural definition that is compatible with bounded cohomology in terms of duality - different from the one studied in \cite{pSV} - is identically $1$.

\subsection{Definitions}

We fix a non-Archimedean valued field $\mathbb{K}$ for the rest of this section and omit it from the notation. \\

Let $X$ be a topological space, and denote by $S_n(X)$ the collection of singular $n$-simplices in $X$, that is, continuous maps $\sigma : \Delta^n \to X$ where $\Delta^n \subset \mathbb{R}^{n+1}$ is the standard $n$-simplex. Define the complex of \emph{bounded singular cochains} by:
$$C^n_b(X) := \{ \varphi : S_n(X) \to \mathbb{K} : \| \varphi \|_\infty < \infty \}.$$
The supremum $\| \varphi \|_\infty$ is called the \emph{norm} of $\varphi$. This is a subcomplex of the singular cochain complex $C^n(X)$ with the usual coboundary map $\delta^n : C^n(X) \to C^{n+1}(X)$. We denote bounded cocycles by $Z^\bullet_b(X)$, bounded coboundaries by $B^\bullet_b(X)$, and \emph{bounded cohomology} by $H^\bullet_b(X)$. \\

Given a pair $(X, A)$, where $A$ is a subspace of $X$, the inclusion $S_n(A) \subset S_n(X)$ defines the subcomplex of relative bounded cochains
$$C^n_b(X, A) := \{ \varphi \in C^n_b(X) : \varphi|_{S_n(A)} \equiv 0 \}.$$
In other words $C^n_b(X, A)$ is the kernel of the restriction map $C^n_b(X) \to C^n_b(A)$. We denote relative bounded cocycles by $Z^\bullet_b(X, A)$, relative bounded coboundaries by $B^\bullet_b(X, A)$, and \emph{relative bounded cohomology} by $H^\bullet_b(X, A)$. We will always tacitly identify a space $X$ with the pair $(X, \emptyset)$. \\

Bounded cohomology of pairs is functorial: given $f : (X, A) \to (Y, B)$, the induced homomorphism at the level of cochain complexes preserves bounded cochains, and so it descends to a homomorphism $H^\bullet_b(f) : H^\bullet_b(Y, B) \to H^\bullet_b(X, A)$ at the level of bounded cohomology, which enjoys the usual functorial properties. In order to prove Theorem \ref{main_top} we need to prove that the following axioms are satisfied:

\begin{enumerate}
\item \textbf{Dimension:} If $\{ * \}$ denotes the one-point space, then $H^n_b(\{ * \}) = 0$ for all $n \geq 1$.
\item \textbf{Homotopy:} If $f, g : (X, A) \to (Y, B)$ are homotopic, then $H^n_b(f) = H^n_b(g)$ for all $n \geq 0$.
\item \textbf{Exactness:} There exists a natural transformation $d^n : H^n_b(A) \to H^{n+1}_b(X, A)$, such that the inclusions $i : A \to X$ and $j : (X, \emptyset) \to (X, A)$ induce a long exact sequence
$$\cdots \to H^n_b(X, A) \xrightarrow{H^n_b(j)} H^n_b(X) \xrightarrow{H^n_b(i)} H^n_b(A) \xrightarrow{d^n} H^{n+1}_b(X, A) \to \cdots$$
\item \textbf{Excision:} Given a space $X$ with subspaces $A, B$ whose interiors cover $X$, the inclusion $i : (B, A \cap B) \to (X, A)$ induces isomorphisms
$$H^n_b(i) : H^n_b(X, A) \to H^n_b(B, A \cap B)$$
for all $n \geq 0$.
\item \textbf{Additivity (modified):} If $(X, A)$ is the disjoint union of $(X_i, A_i)_{i \in I}$, then the inclusions $j_i : (X_i, A_i) \to (X, A)$ induce an injective homomorphism
$$\prod\limits_{i \in I} H^n_b(j_i) : H^n_b(X, A) \to \prod\limits_{i \in I} H^n_b(X_i, A_i)$$
for all $n \geq 0$, whose image is the subspace of sequences admitting representatives with uniformly bounded norm.
\end{enumerate}

Note how this modified version of additivity appeared in the statement of Proposition \ref{EH add}, which concerned additivity of exact bounded cohomology in degree $2$ for discrete groups.

\subsection{Proof of Theorem \ref{main_top}}

We will prove that the axioms are satisfied by making reference to the classical proofs in singular cohomology, and explaining how all the passages carry over to this setting. We will be referencing \cite{Hat} throughout. 

Given a continuous map $f : (X, A) \to (Y, B)$, the induced map at the level of singular chains will be denoted by $f_n : C_n(X, A) \to C_n(Y, B)$, the one at the level of singular cochains by $f^n : C^n(Y, B) \to C^n(X, A)$, and the restriction to bounded cochains by $f^n_b : C^n_b(Y, B) \to C^n_b(X, A)$. \\

The \textbf{dimension} axiom is easy to prove directly, but it also follows immediately from the one in singular cohomology since all singular cochains on a point are trivially bounded. \\

For the next axioms, the proofs basically reduce to the following simple observation:

\begin{lemma}
\label{key_top}

Let $(X, A)$ and $(Y, B)$ be pairs of spaces, and let $h_n : C_n(X, A; \mathbb{Z}) \to C_{n+1}(Y, B; \mathbb{Z})$ be homomorphisms, which induce $\mathbb{K}$-linear maps $h_n : C_n(X, A) \to C_{n+1}(Y, B)$. Then the duals $h^n : C^n(Y, B) \to C^{n-1}(X, A)$ restrict to $h^n_b : C^n_b(Y, B) \to C^{n-1}_b(X, A)$, and $\| h^n_b \|_{op} \leq 1$.
\end{lemma}

\begin{remark}
In the following proofs the $h_n$ will be chain homotopies: the prism operator for the proof of homotopy equivalence, and the barycentric subdivision operator for the proof of excision.
\end{remark}

\begin{proof}
Let $\sigma \in S_n(X)$, and let $h_n(\sigma) \in C_n(Y, B)$ be represented by a sum $\sum \alpha_i \sigma_i$, where $\alpha_i \in \mathbb{Z}$ and $\sigma_i \in S_n(Y)$. Then, given $\varphi \in C^n_b(Y, B)$, we have
$$| h^n(\varphi)(\sigma) |_\mathbb{K} = |\varphi(h_n(\sigma))|_\mathbb{K} \leq \max \{ |\alpha_i|_\mathbb{K} \cdot |\varphi(\sigma_i)|_\mathbb{K} \} \leq \| \varphi \|_\infty,$$
where we used that the map $\mathbb{Z} \to \mathbb{K}$ (be it injective or with kernel $p \mathbb{Z}$) has image in the $1$-ball of $\mathbb{K}$. It follows that $\| h^n(\varphi) \|_\infty \leq \| \varphi \|_\infty$, so $h^n_b$ is well-defined and has operator norm $1$.
\end{proof}

The key point in this lemma is that there is no assumption on how many simplices are needed to represent $h_n(\sigma)$. This relies heavily on the ultrametric inequality, and it is what makes the proof of excision work. \\

We can now prove the \textbf{homotopy} axiom, where we spell out the details a little more than in the next proofs.

\begin{proof}
Let $f, g : (X, A) \to (Y, B)$ be homotopic. The proof of homotopy invariance of singular cohomology goes as follows \cite[Theorem 2.10]{Hat}. There exist maps $P_n : C_n(X, A) \to C_{n+1}(Y, B)$ providing a chain homotopy between $f_n$ and $g_n$, namely $g_n - f_n = P_{n-1} \partial_n + \partial_{n+1} P_n$. Denoting by $P^n : C^n(Y, B) \to C^{n-1}(X, A)$ the dual, the relation above dualizes to $g^n - f^n = \delta^{n-1} P^n + P^{n+1} \delta^n$, thus providing a chain homotopy between $f^n$ and $g^n$ and proving that $H^n(f) = H^n(g)$.

In order to prove homotopy invariance of bounded cohomology, we need to show that $P^n$ provides a chain homotopy between the restrictions of $f^n$ and $g^n$ to the bounded singular cochains. Now $P_n$ and $\partial_n$ are induced by corresponding maps on integral chains, and so by Lemma \ref{key_top} both $P^n$ and $\delta^n$ send bounded cochains to bounded cochains. We obtain $g^n_b - f^n_b = \delta^{n-1} P^n_b + P^{n+1}_b \delta^n$, which implies that $H^n_b(f) = H^n_b(g)$ and concludes the proof.
\end{proof}

Now let us move to \textbf{exactness}.

\begin{proof}
Let $(X, A)$ be a pair and denote $i : A \to X$ and $j : (X, \emptyset) \to (X, A)$. This gives the short exact sequence \cite[p. 199]{Hat}: $0 \to C^n(X, A) \xrightarrow{j^n} C^n(X) \xrightarrow{i^n} C^n(A) \to 0$.
When restricting to bounded cochains we have a sequence:
$$0 \to C^n_b(X, A) \xrightarrow{j^n_b} C^n_b(X) \xrightarrow{i^n_b} C^n_b(A) \to 0.$$
We claim that this sequence is also exact. First, $j^n_b$ is injective being the restriction of $j^n$. Secondly, any bounded cochain $\varphi \in C^n_b(A)$ may be extended to a bounded cochain $\hat{\varphi} \in C^n_b(X)$ by setting it to be equal to zero on all simplices not supported in $A$, and so $i^n_b$ is surjective. Finally
$$\ker(i^n_b) = \ker(i^n) \cap C^n_b(X) = \im(j^n) \cap C^n_b(X) = \im(j^n_b).$$
The last equality is implied by the fact that $\| j^n(\varphi) \|_\infty = \| \varphi \|_\infty$, and so $\varphi \in C^n(X, A)$ has bounded image only if it is bounded in the first place.

This short exact sequence now gives rise to the desired long exact sequence in bounded cohomology, by the standard homological algebraic machinery.
\end{proof}

By reconstructing explicitly all the arguments, the reader can check that the ultrametric inequality was not strictly needed in the previous proofs. This makes it apparent that these axioms are satisfied by bounded cohomology with real coefficients as well. It is with \textbf{excision} that the ultrametric inequality used in Lemma \ref{key_top} becomes really crucial.

\begin{proof}
Let $A, B \subset X$ be such that the interiors of $A$ and $B$ cover $X$. Denote by $C_n(A + B)$ the span of $S_n(A) \cup S_n(B) \subset C_n(X)$, and let $\iota_n : C_n(A + B) \to C_n(X)$ denote the inclusion. The proof of excision for singular cohomology relies on the fact that this inclusion is a chain homotopy equivalence, namely that there exists a chain map $\rho_n : C_n(X) \to C_n(A + B)$ such that $\iota_n \rho_n$ and $\rho_n \iota_n$ are chain homotopic to the identity. More precisely we have $\rho_n \iota_n = id_n$ and there exist maps $D_n : C_n(X) \to C_{n+1}(X)$ such that $id_n - \iota_n \rho_n = \partial_{n+1} D_n + D_{n-1} \partial_n$. The map $D$ is constructed via barycentric subdivision, and the map $\rho$ is defined in terms of $D$ so that it satisfies the above relation \cite[Proposition 2.21]{Hat}.

Denote by $C^n(A + B) := \{ \varphi : S_n(A) \cup S_n(B) \to \mathbb{K}\}$ and by $C^n_b(A + B)$ the subcomplex of bounded cochains, and the corresponding relative cochain complexes by $C^n(A + B, A), C^n_b(A + B, A)$. Dualizing the above relations we obtain $\iota^n \rho^n = id^n$ and $id^n - \rho^n \iota^n = D^{n+1} \delta^n + \delta^{n-1} D^n$. Now $D_n$ is induced by the corresponding map on integral chains, and so by Lemma \ref{key_top} the dual $D^n$ sends bounded cochains to bounded cochains. The same is true of $\rho^n$, since $\rho_n$ is defined in terms of $D_n$. Therefore $D^n_b$ and $\rho^n_b$ are well defined, and the $D^n_b$ provide a chain homotopy proving that $\iota^n_b$ and $\rho^n_b$ are homotopy inverse to each other. It follows that $\iota^n_b : C^n_b(X) \to C^n_b(A + B)$ induces an isomorphism at the level of bounded cohomology.

Now the maps $\iota, \rho, D$ all send chains supported in $A$ to chains supported in $A$ (as noted in \cite[p. 124]{Hat}), therefore the duals send cochains vanishing on $A$ to cochains vanishing on $A$. Thus there is a well-defined map $\iota^n_b : C^n_b(X, A) \to C^n_b(A + B, A)$ with a homotopy inverse $\rho^n_b : C^n_b(A + B, A) \to C^n_b(X, A)$, which implies that $H^n_b(\iota) : H^n_b(X, A) \to H^n_b(A + B, A)$ is also an isomorphism. Finally we notice that the map $C^n_b(A + B, A) \to C^n_b(B, A \cap B)$ given by inclusion is an isomorphism, and conclude that $H^n_b(i) : H^n_b(X, A) \to H^n_b(B, A \cap B)$ is an isomorphism, too.
\end{proof}

Recovering the explicit definition of $D_n$, it is easy to see that the number of simplices in the expression of $D_n(\sigma)$ can be arbitrarily large. Therefore the above proof does not adapt to the real setting. This cannot be solved by means of another proof: letting $A, B$ be two circles and $X$ their wedge sum, we have that $H^n_b(A; \mathbb{R}) = H^n_b(\mathbb{Z}; \mathbb{R}) = 0$, since $\mathbb{Z}$ is amenable, while $H^n_b(X; \mathbb{R}) = H^n_b(F_2; \mathbb{R})$ is infinite-dimensional, for instance by \cite{Rolli}. It is clear that the corresponding Mayer--Vietoris sequence (see Corollary \ref{MVseq}) fails to be exact in this case. \\

We are left to prove the modified \textbf{additivity} axiom.

\begin{proof}
Let $(X, A)$ be the disjoint union of $(X_i, A_i)$. Since the simplex $\Delta^n$ is connected, any singular simplex $\sigma : \Delta^n \to X$ is contained in a single $X_i$. It follows that $C_n(X)$ is the direct sum of the $C_n(X_i)$, and this decomposition is preserved by the boundary map. Dualizing shows that the inclusion maps induce an isomorphism $C^n(X) \to \prod C^n(X_i)$, which restricts to an isomorphism $C^n(X, A) \to \prod C^n(X_i, A_i)$. Restricting to bounded cochains, we obtain an injective map
$$\prod\limits_{i \in I} (j_i)^n_b : C^n_b(X, A) \to \prod\limits_{i \in I} C^n_b(X_i, A_i).$$
The relation $\| \varphi \|_\infty = \sup\limits_{i \in I} \| \varphi|_{X_i} \|_\infty$ implies that the image is precisely the subspace of sequences of cochains $\varphi_i \in C^n_b(X_i, A_i)$ with bounded supremum. Since this isomorphism is compatible with the coboundary map $\delta^n$, it descends to the level of bounded cohomology.
\end{proof}

\subsection{Non-Archimedean simplicial volume}

Corollary \ref{cell} implies in particular that if $M$ is an oriented closed connected $n$-manifold, then the comparison map $c^n : H^n_b(M, \mathbb{K}) \to H^n(M, \mathbb{K})$ is surjective. This could also be shown directly with the same approach as in Subsection \ref{ss_surj}. In the real setting, the surjectivity of this comparison map is equivalent to the non-vanishing of simplicial volume. This seems to suggest that a suitable notion of $\mathbb{K}$-simplicial volume should never vanish. This subsection is devoted to making this intuition rigorous. \\

The (real) \emph{simplicial volume} of an oriented closed connected manifold $M$ is defined as
$$\| M \| := \inf \{ \sum |\alpha_i| : \sum \alpha_i \sigma_i \text{ is a real fundamental cycle of } M \},$$
where by real fundamental cycle we mean a representative of the real fundamental class $[M] \in H_n(M, \mathbb{R}) \cong \mathbb{R}$.

Simplicial volume was introduced by Gromov in \cite{Grom} as a tool to estimate the volume of a Riemannian manifold, and it lies at the origins of bounded cohomology (see \cite[Chapter 7]{Frig} for a detailed account). As we just mentioned, it turns out that $\| M \| > 0$ if and only if the comparison map $c^n : H^n_b(M, \mathbb{R}) \to H^n(M, \mathbb{R})$ is surjective \cite[Corollary 7.11]{Frig}. \\

When defining a $\mathbb{K}$-simplicial volume, one has two options. The first is to take the same definition, letting the infimum run over all $\mathbb{K}$-fundamental cycles of $M$. This is the approach taken in \cite{pSV}, leading to an interesting invariant, about which little is known. The difficulty in this approach is that it combines the non-Archimedean norm $| \cdot |_\mathbb{K}$ with the Archimedean nature of the $\ell^1$-norm of a fundamental cocycle.

The other option is to consider instead a non-Archimedean norm on the chain complex of $M$, namely the $\ell^\infty$ norm, as we did at the end of Subsection \ref{preli BC}. This is the correct approach to take if one wants to prove results via duality, since the $\ell^\infty$ norm on the chain complex is the pre-dual of the supremum norm on the cochain complex. This leads to defining
$$\| M \|_\mathbb{K}^{NA} := \inf \{ \max |\alpha_i|_\mathbb{K} : \sum \alpha_i \sigma_i \text{ is a } \mathbb{K}\text{-fundamental cycle} \},$$
where NA stands for non-Archimedean. We could then use the same approach as in the real case to deduce the non-vanishing of this invariant from the surjectivity of the comparison map. However, this approach employs the Hahn--Banach Theorem (see \cite[Lemma 6.1]{Frig}), which only holds when $\mathbb{K}$ is spherically complete. Therefore we will take a direct approach, which even allows to pin down the precise value of $\| M \|_\mathbb{K}^{NA}$. We wish to thank Clara L\"oh for providing this argument.

\begin{proposition}
\label{NASV}

For any oriented closed connected manifold $M$ we have $\| M \|_\mathbb{K}^{NA} = 1$.
\end{proposition}

\begin{proof}
We first prove that the infimum is $1$ when considering $\mathfrak{o}$-fundamental cycles, where $\mathfrak{o}$ is the ring of integers of $\mathbb{K}$. Indeed, given a cycle $\sum \alpha_i \sigma_i \in C_n(M, \mathfrak{o})$ such that $\max |\alpha_i|_\mathbb{K} =: |\alpha|_\mathbb{K} < 1$, it belongs to the subcomplex $C_n(M, \mathfrak{o} \cdot \alpha)$, and so the corresponding homology class will generate a proper ideal of $H_n(M, \mathfrak{o}) \cong \mathfrak{o}$. Since the fundamental class generates the whole of $H_n(M, \mathfrak{o})$, it follows that any fundamental cycle must have elements of norm $1$ in its support.

Now given a $\mathbb{K}$-fundamental cycle $\sum \alpha_i \sigma_i$, if there exists $\alpha_i \notin \mathfrak{o}$, then $\sup |\alpha_i| > 1$. It follows that any $\mathbb{K}$-fundamental cycle minimizing $\| M \|_\mathbb{K}^{NA}$ is actually an $\mathfrak{o}$-fundamental cycle, and so we conclude by the previous paragraph.
\end{proof}

\subsection{Relation to bounded cohomology of groups}
\label{ss_bcXG}

Corollary \ref{iso Fn} and Corollary \ref{cell}, together with the corresponding statement for ordinary cohomology, imply that if $X$ is an aspherical CW-complex of type $F_\infty$, then the bounded cohomology of $\pi_1(X)$ is naturally isomorphic to that of $X$. This result extends to the case of general aspherical spaces admitting a universal covering, by the same arguments as in the real case (see \cite[Theorem 5.5]{Frig}). Indeed, this proof only uses the functorial characterization of bounded cohomology in the case of discrete groups (see \cite[Corollary 4.15]{Frig}), which unlike the corresponding characterization in the continuous setting (see \cite[Chapter 7]{Monod}) does not use any tools from harmonic analysis. It is easy to check that indeed all of the proofs carry over verbatim to the non-Archimedean setting. \\

Our results show that the asphericity condition is necessary, even in the case of compact CW-complexes. For example, if $M$ is an oriented closed connected $n$-manifold, then by Corollary \ref{cell} we have $H^n_b(M, \mathbb{K}) = H^n(M, \mathbb{K}) \neq 0$. If now $M$ is simply connected, this implies that $H^n_b(M, \mathbb{K}) \neq H^n_b(\pi_1(M), \mathbb{K}) = 0$. On the other hand, by the Gromov Mapping Theorem \cite{Grom}, for any space $X$ admitting a universal covering we have a natural isomorphism $H^n_b(\pi_1(X), \mathbb{R}) \cong H^n_b(X, \mathbb{R})$ (see \cite{Ivanov, Frig, multicomplexes} for proofs). \\

We wish to conclude this paper with this as the last of many instances in which the non-Archimedean world has proven strikingly different from the real one.

\pagebreak

\bibliographystyle{alpha}
\bibliography{References}

\end{document}